\documentclass[11pt]{article}

\usepackage{lmodern}
\usepackage{hyperref}
\usepackage{microtype}
\usepackage{amsfonts,latexsym,amssymb,amsthm,amsmath,graphicx,cases,comment}
\usepackage{amsmath}
\usepackage{paralist}
\usepackage{graphics} 
\usepackage{ulem}
\usepackage[titletoc,title]{appendix}
\usepackage{empheq}
\usepackage{cancel}
\usepackage{tikz}
\usetikzlibrary{arrows,shapes}
\usetikzlibrary{decorations.pathmorphing,decorations.pathreplacing}
\usetikzlibrary{calc,patterns,angles,quotes}

\usepackage{lineno}

\allowdisplaybreaks

\newtheorem{theorem}{Theorem}[section]

\newtheorem{thm}{Theorem}

\theoremstyle{definition}
\newtheorem{definition}[theorem]{Definition}
\newtheorem{rmk}{Remark}

\newcommand{\be}{\begin{equation}}
\newcommand{\ee}{\end{equation}}
\newcommand{\bsubeq}{\begin{subequations}}
	\newcommand{\esubeq}{\end{subequations}}
\renewcommand{\div}{\text{div}}
\newcommand{\ds}{\displaystyle}

\newcommand{\calL}{{\mathcal{L}}}

\newcommand{\calN}{{\mathcal{N}}}

\newcommand{\calF}{{\mathcal{F}}}
\newcommand{\calB}{{\mathcal{B}}}
\newcommand{\calD}{{\mathcal{D}}}

\newcommand{\calA}{{\mathcal{A}}}

\newcommand{\calC}{{\mathcal{C}}}

\newcommand{\calM}{{\mathcal{M}}}

\newcommand{\calV}{{\mathcal{V}}}
\newcommand{\BA}{\mathbb{A}}

\newcommand{\BR}{\mathbb{R}}
\newcommand{\BC}{\mathbb{C}}

\newcommand{\wti}{\widetilde}
\newcommand{\what}{\widehat}

\newcommand{\bpm}{\begin{pmatrix}}
	\newcommand{\epm}{\end{pmatrix}}

\newcommand{\bbm}{\begin{bmatrix}}
	\newcommand{\ebm}{\end{bmatrix}}

\numberwithin{equation}{section}
\numberwithin{thm}{section}
\numberwithin{rmk}{section}

\newcommand\rfrac[2]{{}^{#1}\!/_{#2}}
\newcommand{\lso}{\bs{L}^{q}_{\sigma}(\Omega)}
\newcommand{\lqo}{L^{q}(\Omega)}
\newcommand{\lo}[1]{\bs{L}^{#1}_{\sigma}(\Omega)}

\newcommand{\Bso}{\bs{B}^{2-\rfrac{2}{p}}_{q,p}(\Omega)}

\newcommand{\BsO}{B^{2-\rfrac{2}{p}}_{q,p}(\Omega)}
\newcommand{\Bt}{\widetilde{B}^{2-\rfrac{2}{p}}_{q,p}}
\newcommand{\Bto}{\widetilde{\bs{B}}^{2-\rfrac{2}{p}}_{q,p}(\Omega)}

\newcommand{\xbttpq}{\wti{\bs{X}}^T_{p,q}}

\newcommand{\xttpqs}{\bs{\wti{X}}^T_{p,q,\sigma}}
\newcommand{\xipqs}{\bs{X}^{\infty}_{p,q,\sigma}}

\newcommand{\yttpq}{\wti{Y}^T_{p,q}}

\newcommand{\xipq}{X^{\infty}_{p,q}}
\newcommand{\xbipq}{\bs{X}^{\infty}_{p,q}}

\newcommand{\Fs}{F_{\sigma}}
\newcommand{\Fbs}{\bs{F}_{\sigma}}

\newcommand{\lqaq}{\big( \lso, \calD(A_q) \big)_{1-\frac{1}{p},p}}

\newcommand{\norm}[1]{\left\lVert#1\right\rVert}
\newcommand{\abs}[1]{\left\lvert#1\right\rvert}

\newcommand{\lplqs}{L^p \big( 0,\infty; \lso \big)}

\newcommand{\ip}[2]{\left\langle #1, #2 \right\rangle}

\newcommand{\bs}[1]{\boldsymbol{#1}}
\newcommand{\bls}{\bs{L}^q_{\sigma}(\Omega)}

\newcommand{\SqsO}{\bs{W}^q_{\sigma}(\Omega)}

\newcommand{\VqpO}{\bs{V}^{q,p}(\Omega)}
\newcommand{\VbqpO}{\bs{V}_b^{q,p}(\Omega)}
\newcommand{\WqsuN}{\left( \bs{W}^q_{\sigma} \right)^u_N}
\newcommand{\Wqs}{\bs{W}^q_{\sigma}}
\newcommand{\WqsO}{\bs{W}^q_{\sigma}(\Omega)}
\newcommand{\Wqss}{\left( \bs{W}^q_{\sigma} \right)^*}

\newcommand{\lsoo}{\bs{L}^{q'}_{\sigma}(\Omega)}
\newcommand{\lqoo}{L^{q'}(\Omega)}
\newcommand{\nin}{\noindent}

\newcommand{\thistheoremname}{}

\newtheorem*{genericthm*}{\thistheoremname}
\newenvironment{namedthm*}[1]
{\renewcommand{\thistheoremname}{#1}%
	\begin{genericthm*}}
	{\end{genericthm*}}

\newcommand{\lplq}{L^p \big( 0,\infty; L^q(\Omega) \big)}

\begin{document}

	
   

  	\title{Finite dimensional boundary uniform stabilization of the Boussinesq system in Besov spaces by critical use of Carleman estimate-based inverse theory \thanks{The research of I. L. and R. T. was partially supported by the National Science Foundation under Grant DMS-1713506. The research of B. P. was supported by the ERC advanced grant 668998 (OCLOC) under the EU's H2020 research program.}}
  
  \author{Irena Lasiecka \thanks{Department of Mathematical Sciences, The University of Memphis, Memphis, TN 38152, USA; IBS, Polish Academy of Sciences	Warsaw, Poland}
  \and Buddhika Priyasad \thanks{Institute for Mathematics and Scientific Computing, University of Graz, Heinrichstrasse 36, A-8010 Graz, Austria. (b.sembukutti-liyanage@uni-graz.at).}
\and Roberto Triggiani \thanks{Department of Mathematical Sciences, The University of Memphis, Memphis, TN 38152, USA.}}
\maketitle	
  \abstract{We consider the d-dimensional Boussinesq system defined on a sufficiently smooth bounded domain, and subject to a pair $\{ v, \bs{u} \}$ of controls localized on $\{ \wti{\Gamma}, \omega \}$. Here, $v$ is a scalar Dirichlet boundary control for the thermal equation, acting on an arbitrary small connected portion $\wti{\Gamma}$ of the boundary $\Gamma = \partial \Omega$. Instead, $\bs{u}$ is a $d$-dimensional internal control for the fluid equation acting on an arbitrary small collar $\omega$ supported by $\wti{\Gamma}$ (Fig 1). The initial conditions for both fluid and heat equations are taken of low regularity. We then seek to uniformly stabilize such Boussinesq system in the vicinity of an unstable equilibrium pair, in the critical setting of correspondingly low regularity spaces, by means of an explicitly constructed, finite dimensional feedback control pair $\{ v, \bs{u} \}$ localized on $\{ \wti{\Gamma}, \omega \}$. In addition, they will be minimal in number, and of reduced dimension: more precisely, $\bs{u}$ will be of dimension $(d-1)$, to include necessarily its $d$\textsuperscript{th} component, and $v$ will be of dimension $1$. The resulting space of well-posedness and stabilization is a suitable, tight Besov space for the fluid velocity component (close to $\bs{L}^3(\Omega$) for $ d = 3 $) and a corresponding Besov space for the thermal component, $ q > d $. Unique continuation inverse theorems for suitably over determined adjoint static problems play a critical role in the constructive solution. Their proof rests on Carleman-type estimates, a topic pioneered by M. V. Klibanov since the early 80's, after the 1939- breakthrough publication \cite{Car}.}

\section{Introduction.}\label{Sec-1}
	
	\subsection{Controlled dynamic Boussinesq equations.}\label{Sec-1.1}
	
	\noindent In this paper, we consider the following Boussinesq approximation equations in a bounded connected region $\Omega$ in $\BR^d$ with sufficiently smooth boundary $\Gamma = \partial \Omega$. More specific requirements will be given below.  Let $Q \equiv (0,T) \times \Omega$ and $\Sigma \equiv (0,T) \times \partial \Omega $ where $T > 0$. Further, let $\omega$ be an arbitrary small open smooth subdomain of the region $\Omega$, $\omega \subset \Omega$, thus of positive measure which is a local collar supported by a corresponding connected arbitrary small portion $\wti{\Gamma}$ of the boundary $\Gamma = \partial \Omega$, Fig 1. 
	
	\begin{center}
		\begin{tikzpicture}[x=10pt,y=10pt,>=stealth, scale=.53]
		\draw[thick]
		(-15,5)
		.. controls (-20,-9) and (-5,-18) .. (15,-7)
		.. controls  (24,-1) and (15,13).. (6,6)
		.. controls (3,3) and (-3,3) .. (-6,6)
		.. controls (-9,9) and (-13,9) .. (-15,5);
		\draw[thick,shade,opacity=.5]
		(-15.8,-2)
		.. controls (-11,-4) and (-10,-7) .. (-10,-10)
		.. controls (-15,-7) and (-15.5,-4) .. (-15.8,-2);

		
		\draw (-9,0) node[scale = 1] {$\omega$};
		\draw (-14.5,-8.5) node[scale = 1] {$\wti{\Gamma}$};
		\draw (-15.8,-2) node {$\bullet$};
		\draw (-10,-10) node {$\bullet$};
		\draw (5,-2) node[scale = 1] {$\Omega$};
		\draw (15,-10) node[scale = 1] {$\Gamma$};
		\draw[->,line width = 0.8pt]  (-10,-1) -- (-13.5,-5);
		
		\draw (0,-15) node[scale=1] {Fig 1};
		\end{tikzpicture}
	\end{center}	
	Let $m$ denote the characteristic function of $\omega$: $ m(\omega) \equiv 1, \ m(\Omega / \omega) \equiv 0$.\\
		
	\noindent \textbf{Notation:} Vector-valued functions and corresponding function spaces will be boldfaced. Thus, for instance, for the vector valued ($d$-valued) velocity field or external force, we shall write say $\bs{y,f} \in \bs{L}^q(\Omega) $ rather than $ y,f \in (L^q(\Omega))^d$.\\
		
	\nin We consider the Boussinesq system under the action of a control pair $\{ v, \bs{u} \}$ localized on $\{ \wti{\Gamma}, \omega \}$. Here $v$ is a scalar Dirichlet boundary control for the thermal equation acting on $\wti{\Gamma}$, while $\bs{u}$ is a $d$-dimensional vector interior control acting as $m(x)\bs{u}(t,x)$ on $\omega$:
	\begin{subequations}\label{1.1}
		\begin{empheq}[left=\empheqlbrace]{align}
		\bs{y}_t - \nu \Delta \bs{y} + (\bs{y} \cdot \nabla) \bs{y} - \gamma (\theta - \bar{\theta}) \bs{e}_d+ \nabla \pi &= m(x)\bs{u}(t,x) + \bs{f}(x)   \text{ in } Q \label{1.1a}\\
		\theta_t - \kappa \Delta \theta + \bs{y} \cdot \nabla \theta &=  g(x) \text{ in } Q \label{1.1b}\\
		\text{div }\bs{y} &= 0   \text{ in } Q \label{1.1c}\\
		\bs{y} = 0, \ \theta &= v \text{ on } \Sigma \label{1.1d}\\
		\bs{y}(0,x) = \bs{y}_0, \quad \theta(0,x) & = \theta_0 \text{ on } \Omega. \label{1.1e}
		\end{empheq}
	\end{subequations}	
	
	\noindent In the Boussinesq approximation system, $\bs{y} = \{y_1, \dots, y_d\}$ represents the fluid velocity, $\theta$ the scalar temperature of the fluid, $\nu$ the kinematic viscosity coefficient, $\kappa$ the thermal conductivity. The scalar function $\pi$ is the unknown pressure. The term $\bs{e}_d$ denotes the vector $(0, \dots, 0, 1)$. Moreover $\ds \gamma = \rfrac{\bar{g}}{\bar{\theta}}$ where $\bar{g}$ is the acceleration due to gravity and $\bar{\theta}$ is the reference temperature. The $d$-vector valued function $\bs{f}(x)$ and scalar function $g(x)$ correspond to an external force acting on the Navier-Stokes equations and a heat source density acting on the heat equation, respectively. They are given along with the I.C.s $\bs{y}_0$ and $\theta_0$, which are assumed of low regularity. Note that $ \bs{y} \cdot \nabla \theta = \div (\theta \bs{y})$.\\
	
	\noindent The Boussinesq system  models heat transfer in a viscous incompressible heat conducting fluid. It consists of  the Navier-Stokes equations (in the vector velocity $\bs{y}$) \cite{Te:1979} coupled with the convection-diffusion equation (for the scalar temperature $\theta$). The external body force $\bs{f}(x)$ and the heat source density $g(x) $ may render the overall system unstable in the technical sense described below by \eqref{1.38}. The goal of the paper is to  exploit  the  localized  controls $\bs{u}$ on $\omega$ and $v$ on $\wti{\Gamma}$, sought to be finite dimensional and in  feedback form, in order to stabilize the overall system. Their minimal number will be equal to the maximal geometric multiplicity of the unstable eigenvalues in \eqref{1.38}. As an additional benefit of our investigation, the feedback fluid component of $ \bs{u} $ will be of reduced dimension $(d-1)$ rather than $d$, to include necessarily the $d$\textsuperscript{th} component of $\bs{u}$, while the feedback heat component of $v$ will be 1-dimensional. This is a consequence of the Unique Continuation Inverse Theory type of property expressed by Theorem \ref{Thm-B.1}, Appendix \ref{app-B} for the adjoint static problem. Its proof \cite{TW.1} after \cite{RT:2009} rests on Carleman-type estimates, a topic consistently pursued by M. V. Klibanov since his pioneering work in the early 80's, after the breakthrough publication \cite{Car}. This theme will be further explored in Section \ref{Sec-2.4}. For now, we note that a Unique Continuation Property is not only a critical theoretical tool. It is also fundamental in computations, where it is realized numerically, as in recent work of M. V. Klibanov and his team (See papers \cite{K.4} and \cite{K.5} and references cited therein) on coefficient inverse problems. Here use is made of a globally convergent numerical method, referred as convexification principle, to solve numerically boundary value problems with Cauchy boundary data for coupled systems of quasilinear elliptic equations. To this end, the convexification constructs a weighted globally strict convex Tikhonov-like functional, which involves a Carleman-type weight. As far as practical applications of the Boussinesq model are concerned, one may consider the situation of controlling the temperature and humidity in a bounded  restricted environment - see \cite{BH:2013}, \cite{BHH:2016} for an eloquent description  of  the physical phenomenon. Due to the physical  significance of  the Boussinesq system, the problem of its stabilization has been considered in the literature - with both localized and boundary controls - following of course prior developments concerning the Navier-Stokes model alone. See subsection \ref{Sec-2.4} for a review of the literature.\\
	
	\nin \textbf{Motivation: why studying uniform stabilization of the Boussinesq problem (\ref{1.1}) in the Besov functional setting of the present paper?}\\
	
	\noindent In short: Stimulated by recent research achievements \cite{LPT.1}, \cite{LPT.3} on the uniform stabilization of the Navier-Stokes equations - to be elaborated below in Section \ref{Sec-2.4} - the present paper sets the stage as a preliminary, needed step toward the authors' final goal of solving the Boussinesq uniform stabilization problem with a set of three localized controls $\{ v_h, \bs{v}_f, \bs{u} \}$ all \uline{finite dimensional}, in \uline{feedback} form and of minimal number: a scalar control $v_h$ acting on heat $\theta$-component as in interior control localized on the small set $\omega$, and a pair $\{ \bs{v}_f, \bs{u} \}$ of controls acting on the Navier-Stokes fluid $\bs{y}$-component, as in \cite{LPT.3}: that is, $\bs{v}_f$ as a boundary tangential control acting on $\wti{\Gamma}$ and $\bs{u}$ as an interior tangential-like control acting on $\omega$. See Fig 1. To obtain finite dimensionality of $\bs{v}_f$ for $d = 3$, it is critical to seek uniform stabilization on the tight Besov setting of the present paper, which for $d = 3$ is ``close" to the space $\bs{L}^3$, well-known  \cite{ESS:1991}, \cite{JS:2013}, \cite{RS:2009} to be a critical space for the well-posedness of the \textit{uncontrolled} $3$-d Navier-Stokes equation in $\BR^3$. Thus, the present paper serves also as a testing ground toward such final goal.
	
	\subsection{Stationary Boussinesq equations.}
	
	\noindent Our starting point is the following result.
	
	\begin{thm} \label{Thm-1.1}
		Consider the following steady-state Boussinesq system in $\Omega$
		\begin{subequations}\label{1.2}
			\begin{empheq}[left=\empheqlbrace]{align}
			- \nu \Delta \bs{y}_e + (\bs{y}_e \cdot \nabla) \bs{y}_e - \gamma (\theta_e - \bar{\theta}) \bs{e}_d+ \nabla \pi_e &= \bs{f}(x)   &\text{in } \Omega \label{1.2a}\\
			-\kappa \Delta \theta_e + \bs{y}_e \cdot \nabla \theta_e &= g(x) &\text{in } \Omega \label{1.2b}\\
			\text{div } \bs{y}_e &= 0   &\text{in } \Omega \label{1.2c}\\
			\bs{y}_e = 0, \ \theta_e &= 0 &\text{on } \partial \Omega. \label{1.2d}
			\end{empheq}
		\end{subequations}
		Let $1 < q < \infty$. For any $\bs{f},g \in \bs{L}^q(\Omega), L^q(\Omega)$, there exists a solution (generally not unique) $(\bs{y}_e,\theta_e, \pi_e) \in (\bs{W}^{2,q}(\Omega) \cap \bs{W}^{1,q}_{0}(\Omega)) \times (W^{2,q}(\Omega) \cap W^{1,q}_{0}(\Omega)) \times (W^{1,q}(\Omega)/\mathbb{R})$.
	\end{thm}
	\noindent See \cite{Ac}, \cite{AAC.1}, \cite{AAC.2} for $ q \ne 2$. In the Hilbert space setting, see \cite{CF:1980}, \cite{FT:1984}, \cite{S-R-R}, \cite{V-R-R}, \cite{Kim}.
	 
	\begin{rmk}\label{Rmk-1.1}
		\begin{enumerate}[(i)]
			\item Let the given external force $\bs{f}$ be conservative, $\bs{f} = \nabla \phi$ and the heat source density $g \equiv 0$. Then, $\bs{y}_e \equiv 0, \ \theta_e \equiv 0, \ \nabla \pi_e = \nabla f - \bar{g} \bs{e}_d$ is a solution of problem (\hyperref[1.2]{1.2a-d}). 
			\item In the case of equilibrium solution $\bs{y}_e$ for just the Navier-Stokes equation, it is well-known \cite{Lad:1969}, \cite{Li:1969}, \cite{Te:1979} that the stationary solution is unique when ``the data is small enough, or the viscosity is large enough" \cite[p 157; Chapt 2]{Te:1979} that is, if the ratio $\ds \rfrac{\norm{f}}{\nu_o^2}$ is smaller than some constant that depends only on $\Omega$ \cite[p 121]{FT:1984}. When non-uniqueness occurs, the stationary solutions depend on a finite number of parameters \cite[Theorem 2.1, p 121]{FT:1984} asymptotically, in the time dependent case.
			\item In this paper, we take one equilibrium solution $\{\bs{y}_e, \theta_e \}$, when non-unique and make the subsequent analysis related to such choice.
		\end{enumerate}	
	\end{rmk} 
	
	\subsection{A first quantitative description of the main goal of the present paper.}\label{Sec-1.3}
	
	The starting point of the present paper is the following: that under a given external force $\bs{f}(x)$ for the fluid equations, a given heat source $g(x)$ for the thermal equation, and given viscosity coefficient $\nu$ and thermal conductivity $\kappa$, the equilibrium solution $\{\bs{y}_e, \theta_e\}$ is unstable, in a quantitative sense to be made more precise in sub-section \ref{Sec-1.7}, specifically in \eqref{1.38}. This will mean that the free dynamics \underline{linear} operator $\mathbb{A}_q$ defined in \eqref{1.34} - which has compact resolvent, and is the  generator of a s.c. analytic semigroup in the appropriate functional setting (Theorem \ref{Thm-1.2}) -  has $N$ unstable eigenvalues.\\
	
	\noindent The main goal of the present paper is then - at first qualitatively - \uline{to feedback stabilize the non-linear Boussinesq model (\ref{1.1}) subject to rough (non-smooth) initial conditions $\{\bs{y}_0,\theta_0\}$, in the vicinity of an (unstable) equilibrium solution $\{\bs{y}_e,\theta_e\}$ in (\ref{1.2})}, by means of a \uline{finite dimensional} localized feedback control pair $\{v, m\bs{u} \}$ as acting on $\{ \wti{\Gamma}, \omega \}$, see Fig 1. Thus this paper pertains to the general issue of ``turbulence suppression or attenuation" in fluids. The general topic of turbulence suppression (or attenuation) in fluids has been the object of many studies over the years, mostly in the engineering literature – through experimental studies and via numerical simulation - and under different geometrical and dynamical settings. The references cited in the present paper by necessity pertain mostly to the mathematical literature. A more precise description of our paper is as follows: establish localized exponential stabilization of problem (\ref{1.1}) near an unstable equilibrium solution $\{\bs{y}_e,\theta_e\}$ by means of a \underline {finite dimensional localized}, spectral-based feedback control pair $\{ v, m\bs{u} \}$, in the important case of initial condition $\bs{y}_0$ of low regularity, as technically expressed by $\bs{y}_0$ being in a suitable Besov space with tight indices, and $\theta_0$ being in a corresponding $L^q$-space $ q > d $, or even in a corresponding Besov space, see Remark \ref{Rmk-1.3}. More precisely, the resulting state space for the pair $\{\bs{y},\theta\}$, where uniform stabilization will be achieved is the space
	\begin{subequations}\label{1.3}
		\begin{equation}\label{1.3a}
		\quad \VqpO \equiv \Bto \times \lqo, \ 1 < p < \frac{2q}{2q - 1} ; \ q > d , \, d = 2,3,
		\end{equation}
		\nin or even
		\begin{equation}\label{1.3b}
		\quad \VbqpO \equiv \Bto \times B^{2-\rfrac{2}{p}}_{q,p}(\Omega), \ 1 < p < \frac{2q}{2q - 1} ; \ q > d , \, d = 2,3,
		\end{equation}		
	\end{subequations}
	where $ \Bto $ is a suitable subspace, see below in \eqref{1.11}, of the Besov space
	\begin{equation}\label{1.4}
	\big( \bs{L}^q(\Omega), \bs{W}^{2,q}(\Omega) \big)_{1-\rfrac{1}{p},p} = \Bso, \quad 1 < p < \frac{2q}{2q -1}; \ q > d,\ d = 2,3,
	\end{equation}
	as a real interpolation space between $ \bs{L}^q (\Omega)$ and $ \bs{W}^{2,q}(\Omega)$. Similarly for $\ds B^{2-\rfrac{2}{p}}_{q,p}(\Omega)$.	This setting will be further elaborated after introducing the Helmholtz decomposition below. In particular, local exponential stability for the velocity field $\bs{y}$ near an equilibrium solution $\bs{y}_e$ will be achieved in the topology of the Besov subspace $ \Bto $ in \eqref{1.11}. Note the tight index: $\ds 1 < p < \rfrac{6}{5} $ for $q > d = 3$, and $ 1 < p <  \rfrac{4}{3}$ for $ d = 2 $. For $d = 3$, such space is ``close" to $\bs{L}^3(\Omega)$. It will be documented below in Remark \ref{Rmk-1.2} that in such a setting, the compatibility conditions on the boundary of the initial conditions \uline{are not recognized}. This feature is \uline{precisely our key objective within the stabilization problem} of the $3$-dimensional Navier-Stokes equations per se \cite{LPT.3}, or for the Boussinesq system with controls $\{ \bs{v}_f, \bs{u} \}$ on the N-S component, as described in the ``motivation" paragraph at the end of Section \ref{Sec-1.1}.
	
	\subsection{Helmholtz decomposition.}\label{Sec-1.5}
	\nin A first difficulty one faces in extending the local exponential stabilization results for fluids such as the Navier-Stokes equations or Boussinesq systems from the Hilbert-space setting as in \cite{Lef}, \cite{Wang} to the $\bs{L}^q$-based setting is the question of the existence of a Helmholtz (Leray) projection for the domain $\Omega$ in $\mathbb{R}^d$. More precisely: Given an open set $\Omega \subset \mathbb{R}^d$, the Helmholtz decomposition answers the question as to whether $\bs{L}^q(\Omega)$ can be decomposed into a direct sum of the solenoidal vector space $\lso$ and the space $\bs{G}^q(\Omega)$ of gradient fields. Here,	
	\begin{equation}\label{1.5}
	\begin{aligned}
	\lso &= \overline{\{\bs{y} \in \bs{C}_c^{\infty}(\Omega): \div \ \bs{y} = 0 \text{ in } \Omega \}}^{\norm{\cdot}_q}\\
	&= \{\bs{g} \in \bs{L}^q(\Omega): \div \ \bs{g} = 0; \  \bs{g}\cdot \nu = 0 \text{ on } \partial \Omega \},\\
	& \hspace{3cm} \text{ for any locally Lipschitz domain } \Omega \subset \mathbb{R}^d, d \geq 2 \\
	\bs{G}^q(\Omega) &= \{\bs{y} \in \bs{L}^q(\Omega):\bs{y} = \nabla p, \ p \in W_{loc}^{1,q}(\Omega) \} \ \text{where } 1 \leq q < \infty.
	\end{aligned}
	\end{equation}
	
	\noindent Both of these are closed subspaces of $\bs{L}^q$.
	
	\begin{definition}\label{Def-1.1}
		Let $1 < q < \infty$ and $\Omega \subset \mathbb{R}^n$ be an open set. We say that the Helmholtz decomposition for $\bs{L}^q(\Omega)$ exists whenever $\bs{L}^q(\Omega)$ can be decomposed into the direct sum (non-orthogonal for $q \ne 2$, orthogonal for $q = 2$)
		\begin{equation}
		\bs{L}^q(\Omega) = \lso \oplus \bs{G}^q(\Omega).\label{1.6}
		\end{equation}
		The unique linear, bounded and idempotent (i.e. $P_q^2 = P_q$) projection operator $P_q:\bs{L}^q(\Omega) \longrightarrow \lso$ having $\lso$ as its range and $\bs{G}^q(\Omega)$ as its null space is called the Helmholtz projection. For additional information we refer to \cite{Ga2:1994}, \cite[Appendix B]{LPT.2} and references therein. In particular, throughout the paper we shall use freely that
		$\ds \big( \lso \big)' = \lo{q'}, \ \rfrac{1}{q} + \rfrac{1}{q'} = 1$.
	\end{definition}
	\nin While for $q = 2$ a Helmholtz decomposition (in fact orthogonal decomposition) exists for any open set $\Omega \subset \BR^d$, this is not the case for $q \neq 2$ \cite{MS:1986}. However, for a bounded convex domain $\Omega \subset \BR^d, \ d \geq 2, \ 1 < q < \infty$ \cite{FMM:1998} or for a bounded $C^1$-domain in $\BR^d$ \cite{FMM:1998}, the Helmholtz decomposition is known to be true. This is the case in the present paper. \noindent We can now provide further critical information on the Besov space  $ \Bto$ which is the fluid component of the state space \eqref{1.3} where well-posedness and uniform stabilization take place for the resulting closed-loop feedback problem.\\
	
	\noindent \textbf{Definition of Besov spaces $\bs{B}^s_{q,p}(\Omega)$ on domains of class $C^1$ as real interpolation of Sobolev spaces:} Let $m$ be a positive integer, $m \in \mathbb{N}, 0 < s < m, 1 \leq q < \infty,1 \leq p \leq \infty,$ then we define the Besov space \cite{HA:1995}
	
	\begin{equation} \label{1.7}
	\bs{B}^{s}_{q,p}(\Omega) = (\bs{L}^q(\Omega),\bs{W}^{m,q}(\Omega))_{\frac{s}{m},p} 
	\end{equation}
	as a real interpolation space between $\bs{L}^q(\Omega)$ and $\bs{W}^{m,q}(\Omega)$.This definition does not depend on $\ds m \in \mathbb{N}$ \cite[p xx]{W:1985}. This clearly gives
	\begin{equation} \label{1.8}
	\bs{W}^{m,q}(\Omega) \subset \bs{B}_{q,p}^s(\Omega) \subset \bs{L}^q(\Omega) \quad \text{ and } \quad \norm{\bs{y}}_{\bs{L}^q(\Omega)} \leq C \norm{\bs{y}}_{\bs{B}_{q,p}^s(\Omega)}.
	\end{equation}

	\noindent We shall be particularly interested in the following special real interpolation space of the $\bs{L}^q$ and $\bs{W}^{2,q}$ spaces $\Big( m = 2, s = 2 - \frac{2}{p} \Big)$:
	\begin{equation}\label{1.9}
	\bs{B}^{2-\frac{2}{p}}_{q,p}(\Omega) = \big(\bs{L}^q(\Omega),\bs{W}^{2,q}(\Omega) \big)_{1-\frac{1}{p},p}.
	\end{equation}
	\noindent Our interest in \eqref{1.9} is due to the following characterization \cite[Thm 3.4]{HA:2000}: if $A_q$ denotes the Stokes operator to be introduced in \eqref{1.14} below, then

	\begin{align}
	\Big( \lso,\mathcal{D}(A_q) \Big)_{1-\frac{1}{p},p} &= \Big\{ \bs{g} \in \Bso : \text{ div } \bs{g} = 0, \ \bs{g}|_{\Gamma} = 0 \Big\} \nonumber \\
	& \hspace{4cm} \text{if } \frac{1}{q} < 2 - \frac{2}{p} < 2 \label{1.10}\\
	\Big( \lso,\mathcal{D}(A_q) \Big)_{1-\frac{1}{p},p} &= \Big\{ \bs{g} \in \Bso : \text{ div } \bs{g} = 0, \ \bs{g}\cdot \nu|_{\Gamma} = 0 \Big\} \nonumber\\
	 \equiv \Bto & \text{ if } 0 < 2 - \frac{2}{p} < \frac{1}{q}; \text{ or } 1 < p < \frac{2q}{2q - 1}.  \label{1.11}
	\end{align}	
	
	\begin{rmk}\label{Rmk-1.2}
		Notice that, in \eqref{1.11}, the condition $\ds \bs{g} \cdot \nu |_{\Gamma} = 0$ is an intrinsic condition of the space $\ds \lso$ in (\ref{1.3}), not an extra boundary condition as $\ds \bs{g}|_{\Gamma} = 0$ in \eqref{1.10}. \qedsymbol
	\end{rmk}	
		
	\noindent \textbf{Orientation:} As already noted,ultimately, we shall seek to obtain uniform feedback stabilization of the fluid component $\bs{y}$ in the Besov subspace $\Bto$, dim $\ds \Omega = d < q < \infty$, $ \ds 1 < p < \rfrac{2q}{2q-1}$, defined by real interpolation in \eqref{1.4}, \eqref{1.11}; The reason being that such a space \uline{ does not recognize boundary conditions}, as noted above in Remark \ref{Rmk-1.2}. Analyticity and maximal regularity of the Stokes problem will require $q > 1$, Appendix \ref{app-A}.\\
	
	\noindent By way of orientation, we state at the outset two main points. For the linearized $\bs{w}$-problem \eqref{1.31} or \eqref{1.32} below in the feedback form \eqref{2.4} or \eqref{5.3}, the corresponding well-posedness and global feedback uniform stabilization result, Theorem \ref{Thm-2.1} or Theorems \ref{Thm-6.1} and \ref{Thm-6.2}, hold in general for $1 < q < \infty$. Instead, the final, main well-posedness and feedback uniform, local stabilization results, Theorems \ref{Thm-2.2} and \ref{Thm-2.3}, the latter for the nonlinear feedback problem (\ref{2.11}) or (\ref{8.3}) corresponding to the original problem (\ref{1.1}) via its translation \eqref{1.13}, will require $q > 3$ to obtain the embedding $\ds \bs{W}^{1,q}(\Omega) \hookrightarrow \bs{L}^{\infty}(\Omega)$ in our case of interest $d = 3$, see (\ref{8.24}), hence $\ds 1 < p < \rfrac{6}{5}$; and $q > 2$, hence $\ds 1 < p < \rfrac{4}{3}$ in the $d = 2$-case. The ultimate main result for the original problem \eqref{1.1} is Theorem \ref{Thm-2.4}.
		
	\subsection{Translated nonlinear Boussinesq problem and its abstract model.}
	
	\noindent \textbf{PDE Model:} We return to Theorem \ref{Thm-1.1} which provides an equilibrium triplet $\{\bs{y}_e, \theta_e, \pi_e\}$. Then, we translate by $\{\bs{y}_e, \theta_e, \pi_e\}$ the original Boussinesq problem (\ref{1.1}). Thus we introduce new variables
	\begin{align}\label{1.12}
	\bs{z} = \bs{y} - \bs{y}_e\ \  (\mbox{a $d$-vector}), \quad h = \theta - \theta_e \ \ (\mbox{a scalar}), \quad \chi = \pi - \pi_e \ \  (\mbox{a scalar})
	\end{align}
	and obtain the translated problem
	\begin{subequations}\label{1.13}
		\begin{align}
		\bs{z}_t - \nu \Delta \bs{z} + (\bs{y}_e \cdot \nabla)\bs{z} + (\bs{z} \cdot \nabla)\bs{y}_e + (\bs{z} \cdot \nabla) \bs{z} - \gamma h \bs{e}_d + \nabla \chi &= m\bs{u} \text{ in } Q \label{1.13a}\\
		h_t - \kappa \Delta h + \bs{y}_e \cdot \nabla h + \bs{z} \cdot \nabla h + \bs{z} \cdot \nabla \theta_e &= 0 \text{ in } Q \label{1.13b}\\
		\div \ \bs{z} &= 0  \text{   in } Q \label{1.13c}\\
		\begin{picture}(0,0)
		\put(-218,12){$\left\{\rule{0pt}{60pt}\right.$}\end{picture}
		\bs{z} = 0, \ h & = v  \text{ on } \Sigma \label{1.13d}\\
		\bs{z}(0,x) = \bs{z}_0 = \bs{y}_0 - \bs{y}_e, \quad h(0,x) = h_0 = \theta &- \theta_e \text{ on } \Omega \label{1.13e} \\[1mm]
		L_e(\bs{z}) = (\bs{y}_e \cdot \nabla)\bs{z} + (\bs{z} \cdot \nabla)\bs{y}_e \quad (\mbox{Oseen perturbation}). & \label{1.13f}
		\end{align}
	\end{subequations}
	
	\noindent \textbf{Abstract Nonlinear Translated Model.}	First, for $1 < q < \infty$ fixed, the Stokes operator $A_q$ in $\lso$ with Dirichlet boundary conditions  is defined by
	\begin{equation}\label{1.14}
	A_q \bs{z} = -P_q \Delta \bs{z}, \quad
	\mathcal{D}(A_q) = \bs{W}^{2,q}(\Omega) \cap \bs{W}^{1,q}_0(\Omega) \cap \lso.
	\end{equation}
	The operator $A_q$ has a compact inverse $A_q^{-1}$ on $\lso$, hence $A_q$ has a compact resolvent on $\lso$. Its properties are collected in Appendix \ref{app-A}. Next, we introduce the first order operator $A_{o,q}$,	
	\begin{equation}\label{1.15}
	A_{o,q} \bs{z} = P_q[(\bs{y}_e \cdot \nabla )\bs{z} + (\bs{z} \cdot \nabla )\bs{y}_e], \quad \mathcal{D}(A_{o,q}) = \mathcal{D}(A_q^{\rfrac{1}{2}}) \subset \lso,
	\end{equation}
	where the $\ds \calD(A^{\rfrac{1}{2}}_q)$ is defined explicitly by complex interpolation
	\begin{equation}\label{1.16}
	[ \calD(A_q), \lso ]_{\frac{1}{2}} = \calD(A_q^{\rfrac{1}{2}}) \equiv \bs{W}_0^{1,q}(\Omega) \cap \lso.
	\end{equation}
	Thus, $A_{o,q}A_q^{-\rfrac{1}{2}}$ is a bounded operator on $\lso$, and thus $A_{o,q}$ is bounded on $\calD(A_q^{\rfrac{1}{2}})$
	\begin{equation*}
	\norm{A_{o,q}f} = \norm{A_{o,q} A_q^{-\rfrac{1}{2}} A_q^{-\rfrac{1}{2}} A_q f} \leq C_q \norm{A_q^{\rfrac{1}{2}} f}, \quad f \in \calD(A_q^{\rfrac{1}{2}}).
	\end{equation*}
	This leads to the definition of the Oseen operator for the fluid
	\begin{equation}\label{1.17}
	\calA_q  = - (\nu A_q + A_{o,q}), \quad \calD(\calA_q) = \calD(A_q) \subset \lso.
	\end{equation}
	\nin We next define the differential operator of the heat component in \eqref{1.13b}
	\begin{equation}\label{1.18}
		\calB_q f = -\kappa \Delta f + \bs{y}_e \cdot \nabla f, \quad \mathcal{D}(\calB_q) = W^{2,q}(\Omega) \cap W^{1,q}_0(\Omega).
	\end{equation}
	Then, we define the projection of the nonlinear portion of the fluid operator in \eqref{1.13a}
	\begin{equation}\label{1.19}
	\calN_q(\bs{z}) = P_q [(\bs{z} \cdot \nabla) \bs{z}], \quad \calD(\calN_q) = \bs{W}^{1,q}(\Omega) \cap \bs{L}^{\infty} (\Omega) \cap \lso.	
	\end{equation}
	(Recall that $W^{1,q}(\Omega) \hookrightarrow L^{\infty} (\Omega)$ for $q > d = \mbox{dim } \Omega$ \cite[p. 74]{SK:1989}). Next, we define the nonlinear coupled term of the heat equation as
	\begin{equation}\label{1.20}
	\calM_q[\bs{z}](h) = \bs{z} \cdot \nabla h, \quad \calD(\calM_q[\bs{z}]) = W^{1,q}(\Omega) \cap L^{\infty} (\Omega).  	
	\end{equation}
	Finally, we define the coupling linear terms as bounded operators on $\lqo, \lso$ respectively, $q > d$:
	\begin{align}
	\text{[from the NS equation]} \quad \calC_{\gamma} h &= -\gamma P_q (h \bs{e}_d), \ \calC_{\gamma} \in \calL (L^q(\Omega),\lso), \label{1.21}\\
	\text{[from the heat equation]} \quad \calC_{\theta_e} \bs{z} &=  \bs{z} \cdot \nabla \theta_e, \ \calC_{\theta_e} \in \calL(\lso,\lqo). \label{1.22}
	\end{align}	 
	\nin Next, in preparation to the abstract version of the non-linear $\bs{z}$-system \eqref{1.13}, we introduce the Dirichlet map $D$ \cite{LT3:2015} with reference to the Dirichlet boundary controlled thermal equation \eqref{1.13b} in $h$:
	\begin{subequations}\label{1.23}
		\begin{empheq}[left=\empheqlbrace]{align}
		\psi &= Dv \iff \left\{ \kappa \Delta \psi - \bs{y}_e \cdot \nabla \psi = 0 \text{ in } \Omega, \ \psi|_{\Gamma} = v \text{ on } \Gamma \right\} \label{1.23a}\\
		D &:L^q(\Gamma) \longrightarrow W^{\rfrac{1}{q}, q}(\Omega) \subset W^{\rfrac{1}{q} - 2 \varepsilon, q}(\Omega) \equiv \calD \left( B_q^{\rfrac{1}{2q} - \varepsilon} \right)\nonumber \\ & \hspace{7cm} \text{continuously}. \label{1.23b}
		\end{empheq}
		\nin \cite[Theorem III.2.3, p91]{W:1985} where $B_q$ is the Dirichlet Laplacian on $L^q(\Omega)$
		\begin{equation}\label{1.23c}
			B_q f = -\Delta f, \quad \mathcal{D}(B_q) = \mathcal{D}(\calB_q) = W^{2,q}(\Omega) \cap W^{1,q}_0(\Omega).
		\end{equation}
	\end{subequations}
	\nin Accordingly, we rewrite Eq \eqref{1.13b} via \eqref{1.18}, \eqref{1.23}
	\begin{equation}\label{1.24}
		h_t - (\kappa \Delta - \bs{y}_e \cdot \nabla) (h - Dv) + \bs{z} \cdot \nabla h + \bs{z} \cdot \nabla \theta_e = 0 \text{ in } Q
	\end{equation}
	
	\nin where $\ds [h - Dv]_{\Gamma} = 0$ by \eqref{1.13d} and \eqref{1.23a}. Accordingly, invoking the operators $\calB_q$ from \eqref{1.18} as well as $\calM_q[\bs{z}], \ \calC_{\theta_e}$ from \eqref{1.20}, \eqref{1.22} respectively, we can rewrite Eq \eqref{1.24} abstractly as
	\begin{equation}\label{1.25}
		h_t + \calB_q (h - Dv) + \calM_q[\bs{z}]h + \calC_{\theta_e} \bs{z} = 0.
	\end{equation}
	\nin Extending as usual \cite{LT3:2015}, the original operator $\ds \calB_q: L^q(\Omega) \supset \calD(\calB_q) \longrightarrow L^q(\Omega)$ into $\ds (L^q(\Omega))' = L^{q'}(\Omega) \longrightarrow \left[\calD(\calB_q^*) \right]'$ (duality w.r.t. $L^q$) and retaining the same symbol $\calB_q$ for the extension, we see that we can rewrite Eq \eqref{1.25} as 
	\begin{equation}\label{1.26}
	h_t + \calB_qh + \calM_q[\bs{z}]h + \calC_{\theta_e} \bs{z} = \calB_qDv \in \left[\calD(B_q^*) \right]'.
	\end{equation}
	\nin Next we apply the Helmholtz projector $P_q$ on the coupled N-S equation \eqref{1.13a}, invoke the operators introduced above - specifically, $\calA_q$ from \eqref{1.17}, $\calN_q$ from \eqref{1.19}, and $\calC_{\gamma}$ from \eqref{1.21}- and obtain the following abstract version of the controlled fluid equation
	\begin{equation}\label{1.27}
		\bs{z}_t - \calA_q \bs{z} + \calN_q \bs{z} + \calC_{\gamma} h = P_q(m\bs{u}) \text{ in } \lso\\
	\end{equation}
	\nin In conclusion, combining Eq \eqref{1.27} with Eq \eqref{1.26} we obtain the abstract version of the original boundary controlled Boussinesq system (\hyperref[1.13]{1.13a-e})
	\begin{subequations}\label{1.28}
		\begin{empheq}[left=\empheqlbrace]{align}
		\bs{z}_t - \calA_q \bs{z} + \calN_q \bs{z} + \calC_{\gamma} h &= P_q(m\bs{u}) &\text{ in } \lso\\
		h_t + \calB_q h + \calM_q[\bs{z}] h + \calC_{\theta_e} \bs{z} &= \calB_qDv &\text{ in } \left[\calD(B_q^*) \right]' \subset L^{q'}(\Omega)\\
		\bs{z}(x,0) &= \bs{z}_0(x) &\text{ in } \lso\\
		h(x,0) &= h_0(x) &\text{ in } \lqo;
		\end{empheq}
	\end{subequations}
	\nin or in matrix form
	\begin{subequations}\label{1.29}
		\begin{align}
		\frac{d}{dt}
		\bbm \bs{z} \\ h \ebm &= \bbm \calA_q & -\calC_{\gamma} \\ -\calC_{\theta_e} & 0 \ebm \bbm \bs{z} \\ h \ebm - \bbm  \calN_q  & 0  \\ 0 & \calM_q[\bs{z}] \ebm \bbm \bs{z} \\ h \ebm + \bbm P_q (m\bs{u}) \\ -\calB_q(h - Dv) \ebm \text{ in } \Wqs(\Omega) \nonumber\\[2mm]
		\begin{picture}(0,0)
		\put(-4,17){$\left\{\rule{0pt}{35pt}\right.$}\end{picture}
		\bbm \bs{z}(0) \\ h(0) \ebm &= \bbm \bs{z}_0 \\ h_0 \ebm \in \bs{W}^q_{\sigma}(\Omega),\label{1.29a}
		\end{align}
		\nin or alternatively
		\begin{multline}\label{1.29b}
			\frac{d}{dt}
			\bbm \bs{z} \\ h \ebm = \bbm \calA_q & -\calC_{\gamma} \\ -\calC_{\theta_e} & -\calB_q \ebm \bbm \bs{z} \\ h \ebm - \bbm  \calN_q  & 0  \\ 0 & \calM_q[\bs{z}] \ebm \bbm \bs{z} \\ h \ebm + \bbm P_q (m\bs{u}) \\ \calB_qDv \ebm \\ \text{ in } \lso \times \left[\calD(B_q^*) \right]'
		\end{multline}		
	\end{subequations}
	\begin{equation}\label{1.30}
		\bs{W}^q_{\sigma}(\Omega) \equiv \lso \times \lqo. 
	\end{equation}	

	\subsection{The linearized $\bs{w}$-problem of the translated $\bs{z}$-model.}
	\noindent Next, still for $1 < q < \infty$, we introduce the linearized controlled system of the translated PDE-model (\ref{1.13}), in the variable $\bs{w} = \{\bs{w}_f, w_h\} \in \lso \times \lqo \equiv \bs{W}^q_{\sigma}(\Omega)$:
	\begin{subequations}\label{1.31}
		\begin{empheq}[left=\empheqlbrace]{align}
		\frac{d \bs{w}_f}{dt} - \nu \Delta \bs{w}_f + L_e(\bs{w}_f) - \gamma w_h \bs{e}_d + \nabla \chi &= m \bs{u}   &\text{ in } Q \label{1.31a}\\
		\frac{d w_h}{dt} - \kappa \Delta w_h + \bs{y}_e \cdot \nabla w_h + \bs{w}_f \cdot \nabla \theta_e &= 0 &\text{ in } Q \label{1.31b}\\
		\text{div } \bs{w}_f &= 0   &\text{ in } Q \label{1.31c}\\
		\bs{w}_f \equiv 0, \ w_h &\equiv v &\text{ on } \Sigma \label{1.31d}\\
		\bs{w}_f(0,\cdot) = \bs{w}_{f,0}; \quad w_h(0,\cdot) & = w_{h,0} &\text{ on } \Omega. \label{1.31e}
		\end{empheq}
	\end{subequations}	
	\noindent with I.C. $\{\bs{w}_f(0), w_h(0)\} \in \bs{W}^q_{\sigma}(\Omega) = \lso \times \lqo$. Its corresponding abstract version is, referring also to \eqref{1.20}, (\hyperref[1.29]{1.29a-b}) rewritten, for $\bs{w} = \{ \bs{w}_f, w_h \}$ as
	\begin{subequations}\label{1.32}
		\begin{equation}\label{1.32a}
		\frac{d \bs{w}}{dt}
		= \bbm \ds \frac{d\bs{w}_f}{dt} \\[4mm] \ds \frac{dw_h}{dt} \ebm = \bbm \ds \calA_q \bs{w}_f -\calC_{\gamma} w_h + P_q (m\bs{u}) \\[2mm] \ds -\calB_q (w_h - Dv) - \calC_{\theta_e} \bs{w}_f \ebm \text{ in } \Wqs(\Omega)
		\end{equation}
		\begin{equation}\label{1.32b}
		\frac{d \bs{w}}{dt} = \frac{d}{dt} \bbm \bs{w}_f \\ w_h \ebm
		= \bbm \calA_q & -\calC_{\gamma} \\ -\calC_{\theta_e} & 0 \ebm \bbm \bs{w}_f \\ w_h \ebm + \bbm P_q (m\bs{u}) \\ -\calB_q (w_h - Dv) \ebm \text{ in } \Wqs(\Omega)
		\end{equation}
		\nin We may also write
		\begin{equation}
			\frac{d \bs{w}}{dt}
			= \bbm \ds \frac{d\bs{w}_f}{dt} \\[4mm] \ds \frac{dw_h}{dt} \ebm = \bbm \calA_q & -\calC_{\gamma} \\ -\calC_{\theta_e} & -\calB_q \ebm \bbm \bs{w}_f \\ w_h \ebm + \bbm P_q (m\bs{u}) \\ \calB_qDv \ebm  \text{ in } \lso \times \left[\calD(B_q^*) \right]'
		\end{equation}		
	\end{subequations}
	\nin The corresponding linearized uncontrolled problem $\{ v \equiv 0, \bs{u} \equiv 0 \}$ is
	\begin{equation}\label{1.33}
	\frac{d \bs{w}}{dt} = \frac{d}{dt} \bbm \bs{w}_f \\ w_h \ebm
	= \BA_q \bbm \bs{w}_f \\ w_h \ebm = \bbm \calA_q & -\calC_{\gamma} \\ -\calC_{\theta_e} & -\calB_q \ebm \bbm \bs{w}_f \\ w_h \ebm
	\end{equation}
	\begin{multline}\label{1.34}
	\BA_q = \bbm \calA_q & -\calC_{\gamma} \\ -\calC_{\theta_e} & -\calB_q \ebm : \bs{W}^q_{\sigma}(\Omega) \equiv \lso \times \lqo \supset \calD(\BA_q) = \calD(\calA_q) \times \calD(\calB_q) \\ = (\bs{W}^{2,q}(\Omega) \cap \bs{W}^{1,q}_{0}(\Omega) \cap \lso) \times (W^{2,q}(\Omega) \cap W^{1,q}_{0}(\Omega)) \longrightarrow \bs{W}^q_{\sigma}(\Omega).
	\end{multline}	
	\nin Properties of the operator $\ds \BA_q$ in \eqref{1.34}-critical for the proper setting of the present stabilization analysis will be given in the next Section \ref{Sec-1.7}.
	
	\subsection{Properties of the operator $\BA_q$ in \eqref{1.34}.}\label{Sec-1.7}
	We shall use throughout the following notation (recall \eqref{1.3}, \eqref{1.5}, (\ref{1.32}) and \eqref{1.11})
	\begin{equation}\label{1.35}
	\bs{W}^q_{\sigma}(\Omega) \equiv \lso \times \lqo; \quad
	\bs{V}^{q,p}(\Omega) \equiv  \Bto \times \lqo.
	\end{equation}	
	\begin{rmk}\label{Rmk-1.3}
		By using the maximal regularity of the heat equation, instead of the state space $ \VqpO $ in \eqref{1.3}, \eqref{1.35} we could take the state space  $\bs{V}_{b}^{q,p}(\Omega)$ to be the product of two Besov spaces, i.e. $\bs{V}_{b}^{q,p}(\Omega) \equiv \Bto \times B^{ 2-2/p}_{q,p}( \Omega)$ as in \eqref{1.3b}. Here, the second Besov component is the real interpolation between $\lqo$ and $\calD(\calB_q)$, see \cite{PSch2001}. This remark applies to all results involving $\VqpO$, but it will not necessarily be noted explicitly case by case, in order not to overload the notation. \qedsymbol
	\end{rmk}	
	\noindent Accordingly, we shall look at the operator $\BA_q$ in \eqref{1.34} as defined on either space
	\begin{equation}\label{1.36}
	\BA_q: \bs{W}^q_{\sigma}(\Omega) \supset \calD(\BA_q) \to \bs{W}^q_{\sigma}(\Omega) \
	\mbox{or} \
	\BA_q: \bs{V}^{q,p}(\Omega) \supset \calD(\BA_q) \to \bs{V}^{q,p}(\Omega).
	\end{equation}
	The following result collects basic properties of the operator $\BA_q$. It is essentially a corollary of Theorems \ref{A-Thm-1.4} and \ref{A-Thm-1.5} in Appendix \ref{app-A} for the Oseen operator $\calA_q$, as similar results hold for the operator $\calB_q$, while the operator $\calC_{\gamma}$ and $\calC_{\theta_e}$ in the definition \eqref{1.32} of $\BA_q$ are bounded operators, see \eqref{1.21}, \eqref{1.22}.
	
	\begin{thm}\label{Thm-1.2}
		With reference to the Operator $\BA_q$ in \eqref{1.34}, \eqref{1.36}, the following properties hold true:
		\begin{enumerate}[(i)]
			\item $\ds \BA_q$ is the generator of strongly continuous analytic semigroup on either $\bs{W}^q_{\sigma}(\Omega)$ or $\VqpO$ or $\VbqpO$ for $t > 0$;
			\item $\BA_q$ possesses the $L^p$-maximal regularity property on either $\bs{W}^q_{\sigma}(\Omega)$ or $\VqpO$ or $\VbqpO$ over a finite interval:
			\begin{equation}\label{1.37}
			\BA_q \in MReg (L^p(0,T;*)), \ 0 < T < \infty, \quad
			(*) = \bs{W}^q_{\sigma}(\Omega) \mbox{ or } \bs{V}^{q,p}(\Omega).
			\end{equation}
			\item $\ds \BA_q$ has compact resolvent on either $\bs{W}^q_{\sigma}(\Omega)$ or $\bs{V}^{q,p}(\Omega)$.
		\end{enumerate}
	\end{thm}
	
	\noindent Analyticity of $\ds e^{\calA_qt}$ (resp. $e^{\calB_qt}$) in $\lso$ (resp. $L^q(\Omega)$) implies analyticity of $\ds e^{\calA_qt}$ (resp. $\ds e^{\calB_qt}$) on $\ds \calD(\calA_q) = \calD(A_q)$ (resp. $\ds \calD(\calB_q) = \calD(B_q)$), hence analyticity of $\ds e^{\calA_qt}$ (resp. $\ds e^{\calB_qt}$) on the interpolation space $\ds \Bto$  in \eqref{1.11}. (or in $\ds \Bso$) in \eqref{1.9} (resp. \eqref{1.4} in the scalar case).\\
	
	\noindent For the notation of, and the results on, maximal regularity, see \cite{HA:1995}, \cite{Dore:2000}, \cite{GGH:2012}, \cite{HS:2016}, \cite{KW:2001}, \cite{KW:2004}, \cite{PS:2016}, \cite{We:2001}, \cite{W:1985}, etc. In particular, we recall that on a Banach space, maximal regularity implies analyticity of the semigroup, \cite{DeS} but not conversely \cite{Dore:2000}, \cite{KW:2004}. We refer to Appendix \ref{app-A}.\\
	
	\noindent \textbf{Basic assumption:} By Theorem \ref{Thm-1.2}, the operator $ \BA_q $ in \eqref{1.34} has the eigenvalues (spectrum) located in a triangular sector of well-known type. Then our basic assumption - which justifies the present paper - is that such operator $\ds \BA_q$ is unstable: that is, $\BA_q$ has a finite number, say $N$, of eigenvalues $\lambda_1, \lambda_2 ,\lambda_3 ,\dots,\lambda_N$ on the complex half plane $\{ \lambda \in \mathbb{C} : Re~\lambda \geq 0 \}$ which we then order according to their real parts, so that	
	\begin{equation}\label{1.38}
	\ldots \leq Re~\lambda_{N+1} < 0 \leq Re~\lambda_N \leq \ldots \leq Re~\lambda_1,
	\end{equation}
	
	\noindent each $\lambda_i, \ i=1,\dots,N$, being an unstable eigenvalue repeated according to its geometric multiplicity $\ell_i$. Let $M$ denote the number of distinct unstable eigenvalues $\lambda_i$ of $\BA_q$, so that $\ell_i$ is equal to the dimension of the eigenspace corresponding to $\lambda_i$. Instead, $\ds N = \sum_{i = 1}^{M} N_i$ is the sum of the corresponding algebraic multiplicity $N_i$ of $\lambda_i$, where $N_i$ is the dimension of the corresponding generalized eigenspace.\\
	
	\begin{rmk} \label{Rmk-1.4}
		Condition \eqref{1.38} is intrinsic to the notion of `stabilization', whereby then one seeks to construct a feedback control that transforms an original unstable problem (with no control) into a stable one. However, as is well-known \cite{LPT.1}, the same entire procedure can be employed to \uline{enhance at will the stability} of an originally stable system ($ Re~\lambda_1 < 0 $) by feedback control. This is the case of Remark \ref{Rmk-1.1}(i).
	\end{rmk}	
	
	\section{Main results.}\label{Sec-2}
	
	\nin As in our past work \cite{BT:2004}, \cite{BLT1:2006}, \cite{BLT22:2006}, \cite{BLT33:2006}, \cite{LT2:2015}, \cite{LT3:2015}, \cite{LPT.1}, \cite{LPT.2}, we shall henceforth let $\lso$ denote the complexified space $\lso + i\lso$, and similarly for $\lqo$, whereby then we consider the extension of the linearized problem \eqref{1.32} to such complexified space $\lso \times \lqo$. Thus, henceforth, $ \bs{w}$ will mean $ \bs{w} + i \wti{\bs{w}}$, $\bs{u}$ will mean $\bs{u} + i \wti{\bs{u}}$, $\bs{w}_0$ will mean $\bs{w}_0 + i \wti{\bs{w}}_0$. Our results would be given in this complexified setting. How to return to the real-valued formulation of the results was done in these past reference (see e.g. \cite{BT:2004}, \cite{BLT1:2006}, \cite{LT2:2015}, \cite[Section 2.7]{LPT.1}). Because of space constraints, such real-valued statements will not be explicitly listed on the present paper. We refer to the above references.
	
	\subsection{Orientation.}\label{Sec-2.1}
	\nin A main additional feature of the results below is that the feedback control $ \bs{u}_k$ corresponding to the fluid equation is of \uline{reduced} dimension: that is, of dimension $(d-1)$ rather than of dimension $d$. More precisely, setting $\ds \bs{u} = \{ u^{(1)}, u^{(2)}, \dots, u^{(d)} \}$ to express the vector control $\bs{u}$ acting on $\omega$ in terms of its $d$ coordinates, the only constraint is that the last component $u^{(d)}$ is always needed, along with additional $(d-2)$ components with no preference. Thus, for $d = 2$, a feedback control $\bs{u}$ may be used invoking only the component $u^{(2)}$, with no need of component $u^{(1)}$. For $d = 3$, a feedback control $\bs{u}$ may be used involving either the components $\{ u^{(1)}, u^{(3)} \}$ or else components $\{ u^{(2)}, u^{(3)} \}$. This is due to the UCP of Theorem \ref{Thm-B.1} reported in Appendix \ref{app-B}. To express above facts, we introduce appropriate notation. Recall that the vector $\{ \bs{\varphi}, \psi\} = \{\varphi^{(1)}, \varphi^{(2)}, \dots, \varphi^{(d)}, \psi\}$  in $\ds \Wqs(\Omega) \equiv \lso \times L^q(\Omega)$ has $(d+1)$ coordinates, the first $d$ coordinates corresponds to the fluid space, while the last coordinate corresponds to heat space. Motivated by the above considerations, ultimately by the UCP of Theorem \ref{Thm-B.1} of Appendix \ref{app-B}, we shall introduce the notation $\widehat{ \bs{L}}^q_{\sigma} (\Omega)$ to denote	
	\begin{multline}\label{2.1}
	\widehat{ \bs{L}}^q_{\sigma} (\Omega) \equiv \text{ any $(d-1)$-dimensional sub-space obtained from } \bs{L}^q_{\sigma} (\Omega)\\ \text{ after omitting one specific coordinate, except the $d$\textsuperscript{th} coordinate,} \\ \text{from the vectors of }\bs{L}^q_{\sigma} (\Omega).
	\end{multline}  
	
	\subsection{Global well-posedness and uniform exponential stabilization of the linearized $\bs{w}$-problem \eqref{1.32} on either the space $ \bs{W}^q_{\sigma}(\Omega) \equiv \lso \times \lqo$ or the space $\ds \bs{V}^{q,p}(\Omega) \equiv  \Bto \times \lqo$, $\ds 1 < q < \infty, 1< p < \rfrac{2q}{2q-1}$.}\label{Sec-2.2}
		
	\begin{thm}\label{Thm-2.1}
		Let the operator $\BA_q$ in \eqref{1.34} have $N$ possibly repeated unstable eigenvalues $\ds \{ \lambda_j \}_{j = 1}^N$  as in \eqref{1.38}, of which $M$ are distinct. Let $\ell_i$ denote the geometric multiplicity of $\lambda_i$. Set $K = \sup \{ \ell_i; \ i = 1, \dots, M \}$. Let $\ds \WqsuN$ be the $N$-dimensional subspace of $\Wqs(\Omega)$ defined in \eqref{3.2} below, which is the generalized eigenspace of $\BA_q$ corresponding to the unstable eigenvalues $\ds \{ \lambda_j \}_{j = 1}^N$. Recall the space $ \widehat{ \bs{L}}^q_{\sigma} ( \Omega )$ from \eqref{2.1} and let likewise $\ds (\widehat{\bs{W}}^q_{\sigma})^u_N$ be any space obtained from $ (\bs{W}^q_{\sigma})^u_N$ by omitting one specific coordinate from the vectors of $\ds  (\bs{W}^q_{\sigma})^u_N $ except the $d$\textsuperscript{th} coordinate. Then, one may construct finite dimensional feedback operators $F$ and $J$, as desired
		\begin{subequations}\label{2.2}
		\begin{align}
		v &= F \bs{w} = \sum_{k = 1}^K \ip{P_N \bs{w}}{\bs{p}_k}f_k, \ f_k \in \calF \subset W^{2 - \rfrac{1}{q},q}(\Omega), \nonumber \\& \hspace{1cm} \bs{p}_k \in (\bs{W}^u_N)^* \subset \bs{L}^{q'}_{\sigma}(\Omega) \times L^q(\Omega), \ q \geq 2, f_k \text{ supported on } \wti{\Gamma}. \label{2.2a}\\
		Dv &= DF \bs{w} \in \bs{W}^{2,q}(\Omega) \label{2.2b} 
		\end{align}		
		\end{subequations}		
		\begin{multline}			
			J \bs{w} = P_q m (\bs{u}) = P_qm \left(\sum_{k = 1}^K \ip{P_N \bs{w}}{\bs{q}_k} \bs{u}_k \right), \ \bs{u}_k \in \widehat{ \bs{L}}^q_{\sigma} (\Omega), \\ \bs{q}_k \in (\bs{W}^u_N)^* \subset \bs{L}^{q'}_{\sigma}(\Omega) \times L^q(\Omega),\ \bs{u}_k \text{ supported on } \omega. \label{2.3}
		\end{multline}		
		such that, with $\ds \bs{w}=\{ \bs{w}_f,w_h\},\ \bs{w}_N = P_N\bs{w}$, $P_N$ the projector in \eqref{3.1a}, once inserted in the $\bs{w}$-problem \eqref{1.32}, yield a
		\nin resulting closed-loop linearized $\bs{w}$-problem in feedback form
		\begin{subequations}\label{2.4}
			\begin{equation}\label{2.4a}
			\frac{d \bs{w}}{dt} = \bbm \calA_q & -\calC_{\gamma} \\ -\calC_{\theta_e} & 0 \ebm \bs{w} + \bbm P_q \left( m \ds \sum_{k = 1}^K \ip{P_N \bs{w}}{\bs{q}_k}\bs{u}_k \right) \\ -\calB_q \left(w_h - D \ds \sum_{k = 1}^K \ip{P_N \bs{w}}{\bs{p}_k}\bs{f}_k \right) \ebm
			\end{equation}
			\nin or by \eqref{2.2}, \eqref{2.3}
			\begin{equation}\label{2.4b}
			\frac{d \bs{w}}{dt} = \bbm \calA_q & -\calC_{\gamma} \\ -\calC_{\theta_e} & 0 \ebm \bs{w} + \bbm J \bs{w} \\ -\calB_q(w_h - DF\bs{w}) \ebm \equiv \BA_{_{F,q}} \bs{w}.
			\end{equation}
			\nin This way, the feedback operator $\ds \BA_{_{F,q}}$ is defined for the linearized $\bs{w}$-problem in feedback form, as 
			\begin{equation}\label{2.4c}
			\BA_{_{F,q}} = \hat{\BA}_{_{F,q}} + \Pi
			\end{equation}
		\end{subequations}	
		\nin so that \eqref{2.4b} is rewritten as
		\begin{equation}\label{2.5}
		\frac{d \bs{w}}{dt} = \BA_{_{F,q}} \bs{w} = \hat{\BA}_{_{F,q}} \bs{w} + \Pi \bs{w}
		\end{equation}
		\begin{subequations}\label{2.6}
			\begin{align}
			\hat{\BA}_{_{F,q}} \bs{w} &= \bbm -A_q \bs{w}_1 \\ -\calB_q \left( w_2 - DF\bs{w} \right)  \ebm, \quad \Pi \bs{w} = \bbm A_{o,q} & -\calC_{\gamma} \\ -\calC_{\theta_e} & 0 \ebm \bs{w} + \bbm J \bs{w} \\ 0 \ebm \label{2.6a}\\
			\calD \left( \BA_{_{F,q}} \right) &= \calD \left( \hat{\BA}_{_{F,q}} \right) = \left\{ \bs{w} = \bbm \bs{w}_1  \\  w_2 \ebm \in \bs{W}^q_{\sigma}(\Omega) = \bls \times L^q(\Omega): \right. \nonumber \\ &\hspace{2cm} \bs{w}_1 \in \calD(A_q),\ \left( w_2 - DF\bs{w} \right) \in \calD(\calB_q) = \calD(B_q)  \bigg\} \label{2.6b} \\
			\calD(\BA_{_F}) &\subset \calD(A_q) \times \bs{W}^{2,q}(\Omega), \  \calD(\Pi) = \calD(A_{o,q}) \times L^q(\Omega) \label{2.6c}
			\end{align}
		\end{subequations}
		\nin recalling \eqref{1.15}, \eqref{1.17}, \eqref{1.18}, \eqref{1.23c}, \eqref{2.2b}. Moreover;
		\begin{enumerate}[(i)]
			\item The operator $\ds \BA_{_{F,q}}$ in \eqref{2.4}, \eqref{2.6} generates a s.c. analytic semigroup $\ds e^{\BA_{F,q}t}$ on $\bs{W}^q_{\sigma}(\Omega) \equiv \lso \times \lqo$ as well as in the space $\bs{V}^{q,p}(\Omega) \equiv  \Bt \times \lqo$, or $\VbqpO$ see also Remark \ref{Rmk-1.3}.
			\item Such semigroup $\ds e^{\BA_{F,q}t}$ is uniformly (exponentially) stable in  either of these spaces
			\begin{equation}\label{2.7}
			\norm{e^{\BA_{F,q}t} \bs{w}_0}_{(\cdot)} \leq C_{\gamma_0} e^{-\gamma_0 t}\norm{\bs{w}_0}_{(\cdot)}, \ t \geq 0
			\end{equation}
			\noindent where $(\cdot)$ denotes either $\ds \lso \times \lqo \equiv \SqsO$ or else $\ds \Bto \times \lqo \equiv \VqpO$, or $\VbqpO$. In \eqref{2.7}, $\gamma_0$ is any positive number such that $Re~\lambda_{N+1} < -\gamma_0 < 0$.			
			\item Finally, $\BA_{F,q}$ has maximal $L^p$-regularity up to $T = \infty$ on either of these spaces:
			\begin{empheq}[left=\empheqlbrace]{align}\label{2.8}
			\BA_{F,q} &\in MReg (L^p(0,\infty; \ \cdot \ )), \text{ where } (\cdot) \text{ denotes} \nonumber\\
			\text{either } &\lso \times \lqo \equiv \SqsO \nonumber\\
			\text{ or else } &\ds \Bto \times \lqo \equiv \VqpO, \text{ or } \VbqpO.
			\end{empheq}
		\end{enumerate} 
	\end{thm}
	\nin The proof of Theorem \ref{Thm-2.1} begins in Section \ref{Sec-3} and proceeds through Section \ref{Sec-7}. Such proof gives the feedback vectors $\bs{u}_k$ in \eqref{2.3} as being in $\lso$. The refinement of asserting that such vectors $\bs{u}_k$ can be taken in $\widehat{\bs{L}}^q_{\sigma}(\Omega)$ defined in \eqref{2.1} is given in Appendix \ref{app-C}. More precisely, analyticity in (i) is proved in Theorem \ref{Thm-6.1}; uniform decay in (ii) is proved in Theorem \ref{Thm-6.2}, while maximal $L^p$-regularity in (iii) is established in Theorem \ref{Thm-7.1}.
	
	\subsection{Local well-posedness and uniform (exponential) null stabilization of the translated nonlinear $\{\bs{z}, h\}$-problem \eqref{1.29} or \eqref{1.13} by means of a finite dimensional explicit, spectral based feedback control pair $\ds \{ v, \bs{u} \}$ localized on $\{ \wti{\Gamma}, \omega \}$.}\label{Sec-2.3}
	
	Starting with the present subsection, the nonlinearity of problem (\ref{1.1}) will impose for $d = 3$ the requirement $q > 3$, see \eqref{8.24} below. As our deliberate goal is to obtain the stabilization result for the fluid component $\bs{y}$ in the space $\ds \Bto$ \underline{which does not recognize boundary conditions}, Remark \ref{Rmk-1.2}, then the limitation $\ds p < \rfrac{2q}{2q - 1}$ of this space applies, see \eqref{1.11}. In conclusion, our well-posedness and stabilization results will hold under the restriction $\ds q > 3,\ 1 < p < \rfrac{6}{5}$ for $d = 3$, and $\ds q > 2, 1 < p < \rfrac{4}{3}$ for $d = 2$.
	
	\begin{thm}\label{Thm-2.2}
		Let $\ds d = 2,3, \ q > d, \ 1 < p < \frac{2q}{2q-1}$. Consider the nonlinear open-loop $\{\bs{z},h\}$-problem \eqref{1.29a}
		\begin{subequations}\label{2.9}
			\begin{align}
			\frac{d}{dt}
			\bbm \bs{z} \\ h \ebm &= \bbm \calA_q & -\calC_{\gamma} \\ -\calC_{\theta_e} & 0 \ebm \bbm \bs{z} \\ h \ebm - \bbm  \calN_q  & 0  \\ 0 & \calM_q[\bs{z}] \ebm \bbm \bs{z} \\ h \ebm + \bbm P_q (m\bs{u}) \\ -\calB_q (h - Dv) \ebm \text{ in } \Wqs(\Omega) \label{2.9a}\\[2mm]
			\begin{picture}(0,0)
			\put(-4,25){$\left\{\rule{0pt}{45pt}\right.$}\end{picture}
			\bbm \bs{z}(0) \\ h(0) \ebm &= \bbm \bs{z}_0 \\ h_0 \ebm \in \bs{W}^q_{\sigma}(\Omega) \label{2.9b}
			\end{align}
		\end{subequations}
	\nin which, upon application of feedback control pair $\ds \left \{ v, P_qm(\bs{u}) \right\} = \left\{ F \bbm \bs{z} \\ h \ebm, J \bbm \bs{z} \\ h \ebm  \right\}$ of the same structure \eqref{2.2a}, \eqref{2.3} as for the linearized feedback system \eqref{1.32b}
	\begin{equation}\label{2.10}
	\bbm P_q m(\bs{u}) \\ v \ebm = \bbm J \bbm \bs{z} \\ h \ebm \\[3mm] F \bbm \bs{z} \\ h \ebm \ebm = \bbm P_q \Bigg( m \bigg(\ds \sum_{k = 1}^{K} \ip{P_N \bbm \bs{z} \\ h \ebm}{\bs{q}_k} \bs{u}_k \bigg) \Bigg) \\[2mm]  \ds \sum_{k = 1}^{K} \ip{P_N \bbm \bs{z} \\ h \ebm}{\bs{p}_k} f_k  \ebm
	\end{equation}
	\nin $P_N$ defined in \eqref{3.1a} is rewritten in a closed loop feedback form as 
	\begin{subequations}\label{2.11}
		\begin{empheq}[left=\empheqlbrace]{align}
		\frac{d\bs{z}}{dt} - \calA_q \bs{z} + \calC_{\gamma} h + \calN_q \bs{z} &= P_q \Bigg( m \bigg( \sum_{k = 1}^{K} \ip{P_N \bbm \bs{z} \\ h \ebm}{\bs{q}_k} \bs{u}_k \bigg) \Bigg) \label{2.11a}\\
		\frac{dh}{dt} + \calC_{\theta_e} \bs{z} + \calM_q[\bs{z}]h &= -\calB_q \bigg( h - D \ds \sum_{k = 1}^{K} \ip{P_N \bbm \bs{z} \\ h \ebm}{\bs{p}_k} f_k \bigg). \label{2.11b}
		\end{empheq}
	\end{subequations}
	\nin or in matrix form, recalling \eqref{2.4b}, \eqref{2.6b}
	\begin{multline}\label{2.12}
	\frac{d}{dt} \bbm \bs{z} \\ h \ebm = \BA_{F,q} \bbm \bs{z} \\ h \ebm - \bbm \calN_q & 0 \\ 0 & \calM_q[\bs{z}] \ebm \bbm \bs{z} \\ h \ebm; \ \BA_{F,q} \bbm \bs{z} \\ h \ebm = \bbm \calA_q & -\calC_{\gamma} \\ -\calC_{\theta_e} & 0 \ebm \bbm \bs{z} \\ h \ebm \\+ \bbm J \bbm \bs{z} \\ h \ebm \\ -\calB_q \left(h - DF \bbm \bs{z} \\ h \ebm \right) \ebm.
	\end{multline}	
	\nin Let $\ds d = 2,3, \ q > d, \ 1 < p < \frac{2q}{2q-1}$. There exists a positive constant $r_1 > 0$ (identified in the proof below in \eqref{8.31} such that if
	\begin{equation}\label{2.13}
		\norm{\{ \bs{z}_0, h_0 \}}_{\Bto \times \BsO} < r_1,
	\end{equation}
	then Eq \eqref{2.11}, or \eqref{2.12} has a unique fixed point non-linear semigroup solution,
	\begin{equation}\label{2.14}
	\bbm \bs{z} \\ h \ebm (t) =  e^{\BA_{F,q}t} \bbm \bs{z}_0 \\ h_0 \ebm - \int_{0}^{t} e^{\BA_{F,q}(t - \tau)} \bbm \calN_q \bs{z}(\tau) \\ \calM_q[\bs{z}]h(\tau) \ebm d \tau, \ \text{in the space } \xipqs \times \xipq
	\end{equation}
	\begin{equation}\label{2.15}
	\{\bs{z},h\} \in \xipqs \times \xipq \equiv L^p(0,\infty; \calD(\BA_{F,q})) \cap W^{1,p}(0,\infty;\bs{W}^q_{\sigma}(\Omega)),
	\end{equation}	
	\nin Moreover, (referred to as trace theorem)
	\begin{align}
	\xipqs & \hookrightarrow C([0, \infty ];\Bso) \label{2.16}\\[3mm]
	\xipq & \hookrightarrow  C([0, \infty ];\BsO)\label{2.17}
	\end{align}
	so that
	\begin{equation}\label{2.18}
	\xipqs \times \xipq \hookrightarrow C([0, \infty ];\VbqpO),
	\end{equation}
	\nin The space $ \xipqs \times  \xipq $  defined above is the space of $L^p$-maximal regularity for the generator $\BA_{F,q}$.
	\end{thm}
	
	\begin{thm}\label{Thm-2.3}
		Assume the setting of Theorem \ref{Thm-2.2}, in particular $\ds d = 2,3, \ q > d, \ 1 < p < \frac{2q}{2q - 1}$ and the smallness condition \eqref{2.13} for the I.C. Then, the solution guaranteed by Theorem \ref{Thm-2.2} is 
		 uniformly stable in the space
		$\VbqpO \equiv  \Bto \times \BsO$, see also Remark \ref{Rmk-1.3}: there exists $\widetilde{\gamma}>0, \ M_{\widetilde{\gamma}} > 0$ such that said solution satisfies
		\begin{equation}
		\label{2.19}
		\norm{\bbm \bs{z} \\ h \ebm(t)}_{\VbqpO} \leq M_{\widetilde{\gamma}} e^{-\widetilde{\gamma} t}         \norm{\bbm \bs{z}_0 \\ h_0 \ebm}_{\VbqpO}, \quad
		t \geq 0.
		\end{equation}
	\end{thm}	
	\nin See Remark \ref{9.1} for the relationship between $\gamma_0$ in \eqref{2.7} and $\wti{\gamma}$ in \eqref{2.19}.
	
	\subsection{Local well-posedness and uniform (exponential) stabilization near an unstable equilibrium solution $\{ \bs{y}_e, \theta_e \}$ of the original Boussinesq system \eqref{1.1} by means of a finite dimensional, explicit, spectral based feedback control pair $\{ v, \bs{u} \}$ localized on $\{ \wti{\Gamma}, \omega \}$.}\label{Sec-2.4.N}
	
	\nin The results of this subsection - the main ones of the present paper - are an immediate corollary of Theorem \ref{Thm-2.2} and \ref{Thm-2.3} of Section \ref{Sec-2.3} via the change of variable \eqref{1.12}.
	
	\begin{thm}\label{Thm-2.4}
		Let $\ds d = 2,3, \ q > d, \ 1 < p < \frac{2q}{2q-1}$. Consider the original Boussinesq problem \eqref{1.1}. Let $\{ \bs{y}_e, \theta_e \}$ be a given unstable equilibrium solution as guaranteed by Theorem \ref{Thm-1.1} for a steady state problem \eqref{1.2}: i.e. assume \eqref{1.38}. For a constant $\rho > 0$, let the initial condition $\{ \bs{y}_0, \theta_0 \}$ in \eqref{1.1e} be in $\ds \Bto \times \BsO$ and satisfy		
		\begin{multline}
			\calV_{\rho} = \left\{ \{ \bs{y}_0, \theta_0 \} \in \Bto \times \BsO: \right.\\
			\left. \norm{\{ \bs{y}_0, \theta_0 \} - \{ \bs{y}_e, \theta_e \}}_{\Bto \times \BsO} \leq \rho \right\}, \ \rho > 0.
		\end{multline}
		If $\rho > 0$ is sufficiently small, then: for each $\ds \{ \bs{y}_0, \theta_0 \} \in \calV_{\rho}$, let
		\begin{equation}\label{2.21}
			v = F \bbm \bs{y} - \bs{y}_e \\ \theta - \theta_e \ebm = \sum_{k = 1}^K \ip{P_N \bbm \bs{y} - \bs{y}_e \\ \theta - \theta_e \ebm}{\bs{p}_k} \bs{f}_k
		\end{equation}
		\vspace{-0.5cm}
		\begin{subequations}\label{2.22}
			\begin{align}
				\bs{u} &= \wti{J} \bbm \bs{y} - \bs{y}_e \\ \theta - \theta_e \ebm = \sum_{k = 1}^K \ip{P_N \bbm \bs{y} - \bs{y}_e \\ \theta - \theta_e \ebm}{\bs{q}_k} \bs{u}_k \label{2.22a}\\
				J \bbm \bs{y} - \bs{y}_e \\ \theta - \theta_e \ebm &= P_q \left( m \left( \wti{J} \bbm \bs{y} - \bs{y}_e \\ \theta - \theta_e \ebm \right) \right) \label{2.22b}
			\end{align}
		\end{subequations}
		that is, the finite dimensional operators $\ds F \in \calL(\WqsO, L^q(\Gamma))$ and $\ds J \in \calL(\WqsO, \what{\bs{L}}^q_{\sigma}(\Omega))$ be the same as in \eqref{2.2a}, \eqref{2.3} for the linearized $\bs{w}$-problem \eqref{2.4}; or as in \eqref{2.10} for the translated $\{ \bs{z}, h \}$-problem \eqref{2.11}, having the same boundary vectors $\bs{f}_k$ and interior vectors $\bs{p}_k, \bs{q}_k \in (\WqsuN)^*$ and $\bs{u}_k \in \what{\bs{L}}^q_{\sigma}(\Omega)) $ as in Theorem \ref{Thm-2.1}. Substitute $v$ in \eqref{2.21} and $\bs{u}$ in \eqref{2.22a} in \eqref{1.1d} and \eqref{1.1a} respectively, to obtain the original problem \eqref{1.11} in feedback form:
		\begin{subequations}\label{2.23}
			\begin{empheq}[left=\empheqlbrace]{align}
			\bs{y}_t - \nu \Delta \bs{y} + (\bs{y} \cdot \nabla) \bs{y} - \gamma (\theta - \bar{\theta}) \bs{e}_d+ \nabla \pi &= \nonumber \\ m\left(\sum_{k = 1}^K \ip{P_N \bbm \bs{y} - \bs{y}_e \\ \theta - \theta_e \ebm}{\bs{q}_k} \bs{u}_k\right) &+ \bs{f}(x)   \text{ in } Q \\
			\theta_t - \kappa \Delta \theta + \bs{y} \cdot \nabla \theta &=  g(x) \text{ in } Q\\
			\text{div }\bs{y} &= 0   \text{ in } Q\\
			\bs{y} = 0, \ \theta = \sum_{k = 1}^K \ip{P_N \bbm \bs{y} - \bs{y}_e \\ \theta - \theta_e \ebm}{\bs{p}_k} & \bs{f}_k \text{ on } \Sigma\\
			\bs{y}(0,x) = \bs{y}_0, \quad \theta(0,x) & = \theta_0 \text{ on } \Omega.
			\end{empheq}
		\end{subequations}
	\end{thm}
	re-written abstractly after application of the Helmholtz projection $P_q$ on the $\bs{y}$-equation, as
	\begin{subequations}\label{2.24}
		\begin{empheq}[left=\empheqlbrace]{align}
		\frac{d\bs{y}}{dt} - \calA_q \bs{y} + \calC_{\gamma} \theta + \calN_q \bs{y} &= P_q \Bigg( m \bigg( \sum_{k = 1}^{K} \ip{P_N \bbm \bs{y} - \bs{y}_e \\ \theta - \theta_e \ebm}{\bs{q}_k} \bs{u}_k \bigg) \Bigg) \label{2.24a}\\
		\frac{d \theta}{dt} + \calC_{\theta_e} \bs{y} + \calM_q[\bs{y}] \theta = &-\calB_q  \bigg( \theta - D \ds \sum_{k = 1}^{K} \ip{P_N \bbm \bs{y} - \bs{y}_e \\ \theta - \theta_e \ebm}{\bs{p}_k} f_k \bigg). \label{2.24b}
		\end{empheq}
	\end{subequations}
	Let $\ds d = 2,3, \ q > d, \ 1 < p < \frac{2q}{2q-1}$. If $\rho > 0$ is sufficiently small, then 
	\begin{enumerate}[(i)]
		\item the feedback problem \eqref{2.24} admits a unique (fixed point nonlinear semigroup) solution
		\begin{align}
			\hspace{-0.5cm} \{ \bs{y}, \theta \} \in \xipqs \times \xipq &\equiv L^p(0,\infty; \calD(\BA_{F,q})) \cap W^{1,p}(0,\infty;\bs{W}^q_{\sigma}(\Omega)) \\
			\xipqs \times  \xipq & \hookrightarrow C([0, \infty ];\VbqpO),
		\end{align}
		as in \eqref{2.18};
		\item there exists $\wti{\gamma} > 0, M_{\wti{\gamma}} > 0$ such that said solution $\{ \bs{y}, \theta \}$ satisfies
		\begin{equation}\label{2.27}
		\hspace{-0.5cm} \norm{\bbm \bs{y}(t) - \bs{y}_e \\ \theta(t) - \theta_e \ebm}_{\VbqpO} \leq M_{\widetilde{\gamma}} e^{-\widetilde{\gamma} t}         \norm{\bbm \bs{y}_0 - \bs{y}_e \\ \theta_0 - \theta_e \ebm}_{\VbqpO}, \quad
		t \geq 0. 		\qedsymbol
		\end{equation}
	\end{enumerate}

	\subsection{Comparison with the literature.}\label{Sec-2.4}
	\begin{enumerate}[1.]
		\item \textit{The critical feature - the ignition key - to get the analysis of the present paper succeed: an inverse theory problem of the Unique Continuation Property type, which is established via Carleman-type estimates} \cite{TW.1}. These tools are used in the preliminary Section \ref{Sec-4}, dealing with the finite dimensional unstable projected problem \eqref{3.6a} at the level of establishing the essential Kalman controllability condition \eqref{4.31}. This ultimately rests on the UCP reported in Appendix \ref{app-B}, one of the several UCPs proved in \cite{TW.1}, by means of Carleman-type inequalities. The original Carleman estimates (characterized by a suitable exponential weight) were introduced in \cite{Car} in 1939 to establish uniqueness of a PDE on two variables. A historical account is given in \cite[Section 1.1, pp2-3]{K-T.1} from which we quote: ``\textit{In 1979, the growing interest to global uniqueness results for multidimensional CIPs with the lateral data stimulated Klibanov to apply the Carleman estimates for establishing such results. At the same time, a similar global uniqueness theorem was independently formulated and proven by Bukhgeim. As a result, the method of Carleman estimates was presented to the inverse problem community in the joint paper \cite{B-K.1}. The detailed proofs were published in \cite{B.1} and \cite{K.1}}". See also \cite{K.2}, \cite{K.3} for early literature and \cite{K-K.1} for a hyperbolic paper. The authors of the present paper are most pleased to offer a contribution to the special volume of JIIP to recognize the pioneering work- and subsequent intensive production-of M. V. Klibanov on Carleman estimates and inverse theory. For a discussion on Carleman-type inequalities in integral form with lower-order terms and pointwise Carleman estimates with no lower order terms, we refer to \cite{L-R-S.1} and to \cite[Section 1.5]{L-T.10}. A recent book on these topics is \cite{I.1}.
		 
		\item \textit{The Besov space setting}. With reference to both the ``Motivation" of Subsection \ref{Sec-1.1} as well as Subsection \ref{Sec-1.3}, it was already emphasized that all prior literature on the problem of feedback stabilization of either the Navier-Stokes equations or, subsequently, the Boussinesq system is carried out in a Hilbert-Sobolev setting. As already noted, this treatment is inadequate to obtain \underline{finite dimensionality} in full generality of the localized \underline{tangential boundary } feedback control for the 3d-Navier-Stokes equations. This obstacle then motivated the introduction of the $\bs{L}^q$-Sobolev/Besov setting in \cite{LPT.1}, \cite{LPT.3}  with tight indices, see \eqref{1.4}, \eqref{1.11}, that does not recognize boundary conditions, see Remark \ref{Rmk-1.2} and subsequent orientation, in order to solve affirmatively such open problem on the \underline{finite dimensionality} in 3d-Navier-Stokes \underline{tangential boundary feedback} stabilization. Reference  \cite{LPT.1} on localized interior controls sets the preparatory stage for the more demanding boundary control case in \cite{LPT.3}.
		
		\item \textit{The Boussinesq system}. As already noted in the ``Motivation" of Subsection \ref{Sec-1.1}, the present paper is the first contribution in the Besov setting toward the uniform stabilization of the Boussinesq system with a localized \textit{boundary} stabilizing feedback control (finite dimensional): namely, the Dirichlet boundary control for the heat component acting on $\wti{\Gamma}$. It follows the preliminary test case with two stabilizing feedback controls - one for the fluid component and one for the heat component - localized on an arbitrary small set $\omega \subset \Omega$. In turn, the present paper sets the stage for attacking our final goal: uniform stabilization of the Boussinesq system by means of a localized boundary, finite dimensional feedback stabilizing control acting this time on the Navier-Stokes component, of the same type as in the case of the N-S alone \cite{LPT.3}. The topic of stabilization (by open loop controls) of Navier-Stokes equations originated in the pioneering work of \cite{F}.
		
		\item \textit{Review of the literature on the feedback stabilization of the Boussinesq system.} The first contribution is due to \cite{Wang} via internal feedback controls in the Hilbert setting $\bs{H}\times L^2(\Omega$), where $\bs{H}$ is the usual Hilbert $\bs{L}^2_{\sigma}(\Omega)$-specialization of \eqref{1.5} \cite{CF:1980}. \uline{The feedback controllers are both infinite-dimensional} of the type: $\bs{u}= -k(\bs{y}-\bs{y}_e)$ and $ v= -k(\theta-\theta_e),$ for large $k$, under technical assumptions on the localization of the controls. Paper \cite{Lef} also studies the feedback stabilization problem with, eventually, localized controls, which again are \uline{infinite dimensional}; eg defined in terms of a sub-differential. These results are in stark contrast with the present work. Here - first, we establish that the stabilizing control is  finite dimensional; and, second, we show that its fluid component is of reduced dimension, as expressed only by means of $(d-1)$-components to include necessarily the $d$\textsuperscript{th} component. Finally paper \cite{RRR:2019} studies the feedback stabilization problem with, this time, mixed boundary and claims to provide a rigorous treatment of problems studied in \cite{BH:2013}, \cite{BHH:2016}.
		
		\item It is as a by-product of the described UCP, as applied to the adjoint static problem rather than the original static problem, that we can extract  the further benefit of obtaining the closed-loop feedback control acting on the fluid component of reduced dimension: $(d-1)$ rather than $d$. There is however a key difference with respect to a similar property in \cite{LPT.2} with two localized internal controls. Namely, in \cite{LPT.2} the $d$\textsuperscript{th} component of the fluid control is not needed; in the present case, it is essential. This dimension reduction in the closed-loop feedback stabilizing fluid control is in line with the open-loop controllability results in \cite{cg2}, \cite{cl1}. No such reduction of components is present in the case of uniform stabilization of the Navier-Stokes equations per-se.
		
		\item \textit{Geometric Multiplicity.} We also point out that, as in \cite{LT1:2015}, \cite{LT2:2015} and \cite{LPT.1}, \cite{LPT.3} for the Navier-Stokes equations, in the case of the Boussinesq system, the number of needed controls will be related to the more desirable \underline{geometric} multiplicity, not the larger \underline{algebraic} multiplicity as in prior treatments, \cite{BT:2004}, \cite[p 276]{Barbu} of the unstable eigenvalues, by using the classical test for controllability of a system in Jordan form \cite{CTC:1984}, \cite[p 464]{B-M1}. This allows one also to obtain constructively an explicit form of the finite dimensional feedback control; and, moreover, to show that the \uline{feedback control acting on the fluid may be taken of reduced dimension}: one unit less, i.e. $d-1$, than the fluid component of dimension $d$. As noted in point 5. above, this is due to the Unique Continuation Property of the adjoint problem, reported in Appendix \ref{app-B}. It represents an additional contribution of the present paper.
		
		\item \textit{Maximal regularity of the linearized feedback problem.} In the present contribution, following \cite{LPT.1}, \cite{LPT.2}, \cite{LPT.3}, the analysis of the original non-linear problem is carried out in the context of the property of \uline{maximal regularity of the linearized feedback} $\bs{w}$-problem \eqref{5.3a}; that is, of the operator  $\BA_{F,q}$ in 
		\eqref{5.3a}-\eqref{5.5b}. This applies to both well-posedness (in Section \ref{Sec-8}) as well as stabilization (in Section \ref{Sec-9}) of the nonlinear feedback problem. This is in contrast with prior treatments, such as \cite{BT:2004}, \cite{Barbu}, which often rely on chopping the nonlinearity and carrying out a limit process. The \uline{maximal regularity approach} introduced for stabilization problems in \cite{LPT.1} is cleaner and more desirable both technically and conceptually. On the other hand, applicability of Maximal Regularity requires well balanced spaces. Recent developments in this particular area \cite{We:2001}, \cite{SS:2007}, \cite{PS:2016}, \cite{KW:2001}, \cite{KW:2004}, \cite{Dore:2000} etc. were critical for carrying out our analysis.
	\end{enumerate}
	
	\section{Beginning with the proof of Theorem \ref{Thm-2.1}, Spectral decomposition of the linearized $\bs{w}$-problem \eqref{1.31} or \eqref{1.32}.}\label{Sec-3}
	
	\nin We return to the assumed starting point of the present paper, which is that the \textit{free} (uncontrolled) dynamics operator $\BA_q$ in \eqref{1.34} is unstable, see \eqref{1.38}. Its properties are collected in Theorem \ref{Thm-1.2}. Accordingly, its eigenvalues satisfy the statement which includes their location in \eqref{1.38}. Denote by $P_N$ and $P_N^*$ (which actually depend on $q$) the projections given explicitly by \cite[p 178]{TK:1966}, \cite{BT:2004}, \cite{BLT1:2006}
	\begin{subequations}\label{3.1}
		\begin{align}
		P_N &= -\frac{1}{2 \pi i}\int_{\boldsymbol{\Gamma}}\left( \lambda I - \BA_q \right)^{-1}d \lambda : \bs{W}^q_{\sigma} \text{ onto } (\bs{W}^q_{\sigma})^u_N \subset \bs{L}^q_{\sigma}(\Omega) \times L^q(\Omega) \label{3.1a}\\
		P_N^* &= -\frac{1}{2 \pi i}\int_{\bar{\boldsymbol{\Gamma}}}\left( \lambda I - \BA_q^* \right)^{-1}d \lambda : (\bs{W}^q_{\sigma})^* \text{ onto } [(\bs{W}^q_{\sigma})^u_N]^* \subset \bs{L}^{q'}_{\sigma}(\Omega) \times L^{q'}(\Omega). \label{3.1b} 
		\end{align}
	\end{subequations}
	
	\noindent Here ${\boldsymbol{\Gamma}}$ (respectively, its conjugate counterpart $\bar{{\boldsymbol{\Gamma}}}$) is a smooth closed curve that separates the unstable spectrum from the stable spectrum of $\BA_q$ (respectively, $\BA_q^*$). As in \cite[Sect 3.4, p 37]{BLT1:2006}, following \cite{RT:1975}, \cite{RT:1980}, we decompose the space $\bs{W}^q_{\sigma}=\bs{W}^q_{\sigma}(\Omega) \equiv \lso \times \lqo$ into the sum of two complementary subspaces (not necessarily orthogonal):	
	\begin{multline}\label{3.2}
	\bs{W}^q_{\sigma} = (\bs{W}^q_{\sigma})^u_N \oplus (\bs{W}^q_{\sigma})^s_N; \ (\bs{W}^q_{\sigma})^u_N \equiv P_N \bs{W}^q_{\sigma} ;\ (\bs{W}^q_{\sigma})^s_N \equiv (I - P_N) \bs{W}^q_{\sigma};\\ \text{ dim } (\bs{W}^q_{\sigma})^u_N = N
	\end{multline}	
	\noindent where each of the spaces $(\bs{W}^q_{\sigma})^u_N$ and $(\bs{W}^q_{\sigma})^s_N$ (which depend on $q$, but we suppress such dependence) is invariant under $\BA_q$, and let	
	\begin{equation}\label{3.3}
	\BA^u_{q,N} = P_N \BA_q = \BA_q |_{(\bs{W}^q_{\sigma})^u_N} ; \quad \BA^s_{q,N} = (I - P_N) \BA_q = \BA_q |_{(\bs{W}^q_{\sigma})^s_N}
	\end{equation}	
	\noindent be the restrictions of $\BA_q$ to $(\bs{W}^q_{\sigma})^u_N$ and $(\bs{W}^q_{\sigma})^s_N$ respectively. The original point spectrum (eigenvalues) $\{ \lambda_j \}_{j=1}^{\infty} $ of $\BA_q$ is then split into two sets, recall \eqref{1.38}	
	\begin{equation}\label{3.4}
	\sigma (\BA^u_{q,N}) = \{ \lambda_j \}_{j=1}^{N}; \quad  \sigma (\BA^s_{q,N}) = \{ \lambda_j \}_{j=N+1}^{\infty},
	\end{equation}	
	\noindent and $(\bs{W}^q_{\sigma})^u_N$ is the generalized eigenspace of $\BA^u_{q,N}$ in \eqref{3.3}. The system \eqref{1.32} on $\bs{W}^q_{\sigma} \equiv \lso \times \lqo$ can accordingly be decomposed as	
	\begin{equation}\label{3.5}
	\bs{w} = \bs{w}_N + \boldsymbol{\zeta}_N, \quad \bs{w}_N = P_N \bs{w}, \quad \boldsymbol{\zeta}_N = (I-P_N)\bs{w}.
	\end{equation}	
	\noindent After applying $P_N$ and $(I-P_N)$ (which commute with $\BA_q$) on \eqref{1.32}, we obtain via \eqref{3.3}	
	\begin{subequations}\label{3.6} 
		\begin{equation}\label{3.6a}
		\text{on } (\bs{W}^q_{\sigma})^u_N: \bs{w}'_N - \BA^u_{q,N} \bs{w}_N = P_N \bbm P_{q} (m\bs{u}) \\ \calB_qDv \ebm; \quad  \bs{w}_N(0) = P_N \bbm \bs{w}_f (0) \\ w_h(0) \ebm
		\end{equation}
		\begin{equation}\label{3.6b}
		\text{on } (\bs{W}^q_{\sigma})^s_N: \bs{\zeta}'_N - \BA^s_{q,N} \bs{\zeta}_N = (I-P_N) \bbm P_{q} (m\bs{u}) \\  \calB_qDv \ebm; \quad  \bs{\zeta}_N(0) = (I - P_N) \bbm \bs{w}_f (0) \\ w_h(0) \ebm
		\end{equation}
	\end{subequations}
	\noindent respectively, where $\ds \bs{w}_N = \{ \bs{w}_{f,N}, w_{h,N} \}, \ \bs{\zeta}_N = \{ \bs{\zeta}_{f,N}, \zeta_{h,N} \}$. See \cite[Appendix A]{BLT1:2006} for the extension of $P_N$ outside $\Wqs$.
	
	\section{Proof of Theorem \ref{Thm-2.1}. Global well-posedness and uniform exponential stabilization of the linearized $\bs{w}$-problem \eqref{1.32} on the space $\bs{W}^q_{\sigma}(\Omega) \equiv \lso \times \lqo$ or the space $\VbqpO \equiv  \Bto \times \BsO$.}\label{Sec-4}
	
	\subsection{Orientation.}\label{Sec-4.1}
	 \nin We  shall appeal to several  technical developments in \cite{LPT.1}, where the case of the N-S equations has been studied. We will have to adopt and transfer some of its procedures to the  case of the Boussinesq system. Thus, properties such as maximal regularity and the entire  development for uniform stabilization of the linearized Boussinesq dynamics need to be established. We have already noted that while analyticity and maximal regularity are equivalent properties in the Hilbert setting \cite{DeS} this is not so in the Banach setting where maximal regularity is a more general and more delicate property. Moreover, it is only after we establish  uniform stabilization that we can claim maximal regularity up to infinity. This needs to be asserted by direct analysis. To proceed, we recall the state space $\ds \bs{W}^q_{\sigma}(\Omega) = \lso \times \lqo \equiv (\bs{W}^q_{\sigma})^u_N \oplus (\bs{W}^q_{\sigma})^s_N$ in \eqref{1.34}, \eqref{3.2}. The unstable uncontrolled operator is $\ds \BA_q$ in \eqref{1.34}.\\
	
	\noindent The same strategy employed in \cite{LPT.1} seeks to show that the linearized $\bs{w}$-dynamics \eqref{2.4} in feedback form - that is, the operator $\BA_{_{F,q}}$ in \eqref{2.5}, \eqref{2.6b} - of the translated non-linear $\{\bs{z}, h\}$-problem is uniformly stable in the desired setting in terms of finite dimensional feedback controls. Here $\bs{w} = \{ \bs{w}_f, w_h \}$ comprises the $d$-dimensional fluid component $\bs{w}_f$ and scalar heat component $w_h$. To this end, we first seek to establish that the finite dimensional projection - the $\bs{w}_N$-equation in \eqref{3.6a} is controllable on $(\bs{W}^q_{\sigma})^u_N$, hence exponentially stabilizable with an arbitrarily large decay rate \cite[p. 44]{Za}. This is done in Theorem \ref{Thm-4.1} (controllability) and Theorem \ref{Thm-4.2} (stabilization). Once it is established (Theorem \ref{Thm-6.1}) that the operator $\BA_{_{F,q}}$ generates a s.c. analytic semigroup, the next step is to examine the corresponding $\boldsymbol{\zeta}_N$-equation, \eqref{3.6b}, where the arbitrarily large decay rate of the feedback $\bs{w}_N$-equation combined with the exponential stability on $(\bs{W}^q_{\sigma})^s_N$ of the s.c. analytic semigroup $\ds e^{\BA_{q,N}^s t}$ (whose spectrum is on the left of $Re \ \lambda_{N+1} < 0$), yields the desired result. This is Theorem \ref{Thm-6.2}. One needs to emphasize that all this works thanks to the present corresponding Unique Continuation Property, this time of the Boussinesq system, that is Theorem \ref{Thm-B.1} in Appendix \ref{app-B}. It is this UCP that permits one to verify the Kalman algebraic condition \eqref{4.31} of Theorem \ref{Thm-4.1}. This is established asserting the Claim in the proof of Theorem \ref{Thm-4.1}. Such condition is equivalent to linear independence of certain vectors occurring in \eqref{4.31} below, so that the finite dimensional projected $\bs{w}_N$-dynamics satisfies the controllability condition of Kalman or Hautus. See corresponding cases in \cite{RT:2009}, \cite{RT:2008}. The proof of Theorem \ref{Thm-4.1} invokes all $d$-components of the fluid vectors $\bs{u}_k$. However, one can do better. The UCP of Theorem \ref{Thm-B.1}, Appendix \ref{app-B} does \uline{not} include boundary conditions. On the other hand, in our present case, the UCP is applied to the adjoint eigenproblem \eqref{4.3} below, which of course includes also homogeneous B. C. such as \eqref{4.10d}. It then turns out that by taking advantage also of these homogeneous B. C. \eqref{4.10d}, it is possible through a proof of a UCP given in \cite[Theorem 6, Theorem 7]{TW.1} to obtain an improvement whereby the new version of Theorem \ref{Thm-4.1} (verification of the Kalman algebraic condition \eqref{4.31}) involves only $(d-1)$-components of the $d$-vector $\bs{u}_k$, to include necessarily the $d$\textsuperscript{th} component. Such improvement is given in Appendix \ref{app-C}. The proof uses a variation of the proof of the UCP \cite[Theorem 6 or 7]{TW.1}. This result is in contrast with two other cases. First, the uniform stabilization study of the Boussinesq system by localized interior feedback controls for both fluid and heat components (acting on an arbitrarily small internal set $\omega \subset \Omega$) \cite{LPT.2}, again $(d-1)$-components of the fluid vector $\bs{u}_k$ are involved, but this time, it is the $d$\textsuperscript{th} component \uline{that may be omitted}, in contrast with the case of the present paper. Second, in the cases of uniform stabilization of the Navier-Stokes equations, either by internal localized \cite{LPT.1} or boundary localized controls \cite{LPT.3}, all $d$-components of the fluid vector $\bs{u}_k$ are needed. The explanation is in the nature of the Boussinesq \uline{adjoint} eigenproblem. Henceforth the conceptual advantage of having our results invoking only $(d-1)$-components of the fluid vectors $\bs{u}_k$, including the $d$\textsuperscript{th} component will be responsible for heavier notation as described below.\\
	
	\noindent For each $i = 1, \dots, M$, we denote by  $\ds \{\boldsymbol{\Phi}_{ij}\}_{j=1}^{\ell_i},  \{\boldsymbol{\Phi}^*_{ij}\}_{j=1}^{\ell_i}$ the normalized, linearly independent eigenfunctions of $\BA_q$, respectively $\BA_q^*$, say, on
	\begin{multline}\label{4.1}
	\bs{W}^q_{\sigma}(\Omega) \equiv \lso \times \lqo
	\mbox{ and } \\
	(\bs{W}^q_{\sigma}(\Omega))^* \equiv (\lso)' \times (\lqo)' = \bs{L}^{q'}_{\sigma}(\Omega) \times L^{q'}(\Omega), \quad
	\frac{1}{q} + \frac{1}{q'} = 1,
	\end{multline}
	\noindent (where in the last equality we have invoked the identity below \eqref{1.6}) corresponding to the $M$ distinct unstable eigenvalues $\lambda_1, \ldots, \lambda_M$ of $\BA_q$ and $\overline{\lambda}_1, \ldots, \overline{\lambda}_M$ of $\BA_q^*$ respectively, either on $\bs{W}^q_{\sigma}(\Omega)$ or on $\bs{V}^{q,p}(\Omega)$:
	\begin{align}
	\BA_q \boldsymbol{\Phi}_{ij} &= \lambda_i \boldsymbol{\Phi}_{ij} \in \calD(\BA_q) = [\bs{W}^{2,q}(\Omega) \cap \bs{W}^{1,q}_0(\Omega) \cap \lso] \times [W^{2,q}(\Omega) \cap W^{1,q}_0(\Omega)]  \label{4.2}\\
	\BA_q^* \boldsymbol{\Phi}_{ij}^* &= \bar{\lambda}_i \boldsymbol{\Phi}_{ij}^* \in \calD(\BA_q^*) = [\bs{W}^{2,q'}(\Omega) \cap \bs{W}^{1,q'}_0(\Omega) \cap \lo{q'}] \times [W^{2,q'}(\Omega) \cap W^{1,q'}_0(\Omega)]. \label{4.3}
	\end{align}
	\nin It is the adjoint problem \eqref{4.3} which is of our interest:
	
	\subsection{The adjoint operator $\BA_q^*$ of $\BA_q$ in \eqref{1.34}.}
	
	\nin  It is computed in \cite{LPT.2} that the adjoint $\BA_q^*$ of $\BA_q$ is given by
	\begin{multline}\label{4.4}
	\BA_q^* = \bbm \calA_q^* & -\calC_{\theta_e}^* \\ -\calC_{\gamma}^* & \calB_q^* \ebm : \bs{W}^{q'}_{\sigma}(\Omega) = \bs{L}^{q'}_{\sigma}(\Omega) \times L^q(\Omega) \supset \calD(\BA_q^*) = \calD(\calA_q^*) \times \calD(\calB_q^*) \\ = (\bs{W}^{2,q'}(\Omega) \cap \bs{W}^{1,q'}_{0}(\Omega) \cap \bs{L}^{q'}_{\sigma}(\Omega)) \times (W^{2,q'}(\Omega) \cap W^{1,q'}_{0}(\Omega)) \longrightarrow \bs{W}^{q'}_{\sigma}(\Omega).
	\end{multline}
	\nin where the adjoints $\calC_{\theta_e}^*$ and $\calC_{\gamma}^*$ of the operators $\calC_{\theta_e}$ and $\calC_{\gamma}$ in \eqref{1.21}, \eqref{1.22} are	
	\begin{align}
		\calC_{\theta_{e}}^* \psi^* &= P_{q'} ( \psi^* \nabla \theta_e ) \in \bs{L}^{q'}_{\sigma}(\Omega), \quad \calC_{\gamma}^* \boldsymbol{\varphi}^* =  - \gamma(P_{q'} \boldsymbol{\varphi}^*) \cdot\bs{e}_d , \nonumber\\
		&\hspace{5cm}  \calC_{\gamma}^* \in \mathcal{L}( \lo{q'} , L^{q'}( \Omega )). \label{4.5}\\	
		\calA_q^* &= - \left( \nu A^*_q + A^*_{o,q} \right), \ A^*_q \bs{f} = - P_{q'} \Delta \bs{f}, \nonumber \\ &\hspace{3cm} \calD(A^*_q) \equiv \bs{W}^{2,q'}(\Omega) \cap \bs{W}^{1,q'}_{0}(\Omega) \cap \bs{L}^{q'}_{\sigma}(\Omega) \label{4.6}\\
		\left( A_{o.q} \right)^* &= A^*_{o,q} = P_{q'}(L_e)^*: \bs{W}^{1,q'}(\Omega) \longrightarrow \lo{q'} \label{4.7}\\
		\calB_q^* \psi &= - \kappa \Delta \psi - \bs{y}_e \cdot \nabla \psi, \ \calD(\calB_q^*) = W^{2,q'}(\Omega) \cap W^{1,q'}_{0}(\Omega). \label{4.8}
	\end{align}
	\nin The expression in \eqref{4.8} is obtained from \eqref{1.18} integration by parts using $\ds \bs{y}_e|_{\Gamma} = 0$ and div $\bs{y}_e = 0$ from (\hyperref[1.2c]{1.2c-d}). See Appendix \ref{app-B}. Notice that in passing from $\BA_q$ to $\BA_q^*$ the operators $\calC_{\gamma}$ and $\calC_{\theta_e}$ switch places with their adjoints. This fact has a key implication on the needed Unique Continuation Property. With $\bs{\Phi^*} = [\bs{\varphi}^*, \psi^*]$, the explicit version of $\BA_q^* \bs{\Phi} = \lambda \bs{\Phi}^*$ is by \eqref{4.4}
	\begin{subequations}\label{4.9}
		\begin{empheq}[left=\empheqlbrace]{align}
			\calA_q^* \bs{\varphi}^* - \calC_{\theta_e}^* \psi^* &= \lambda \bs{\varphi}^* \label{4.9a}\\
			-\calB_q^* \psi^* - \calC_{\gamma}^* \bs{\varphi}^* &= \lambda \psi^*  \label{4.9b}
		\end{empheq}
	\end{subequations}
	\nin while its PDE-version is, by virtue of \eqref{4.4}-\eqref{4.8}
	\begin{subequations}\label{4.10}
		\begin{empheq}[left=\empheqlbrace]{align}
		-\nu \Delta \boldsymbol{\varphi^*} + L^*_e(\boldsymbol{\varphi^*})+ \psi^* \nabla \theta_e + \nabla \pi &= \lambda \boldsymbol{\varphi^*}   &\text{in } \Omega      \label{4.10a} \\
		-\kappa \Delta \psi^* - \bs{y}_e \cdot \nabla \psi^* - \gamma  \boldsymbol{\varphi^*}  \cdot\bs{e}_d &= \lambda \psi^* &\text{in } \Omega \label{4.10b}\\
		\text{div } \boldsymbol{\varphi^*} &= 0   &\text{in } \Omega  \label{4.10c} \\
		\boldsymbol{\varphi^*} = 0, \ \psi^* &= 0 &\text{on } \Gamma.    \label{4.10d}
		\end{empheq}		
	\end{subequations}
	\subsection{Critical step: uniform stabilization of the unstable finite dimensional $\bs{w}_N$-projection \eqref{3.6a} by feedback controls $\{ v, \bs{u} \}$.}\label{Sec-4.3}
	
	\nin For simplicity and space constraints, we provide here explicitly only the conceptually and computationally more amenable case, where the operator $\BA_{q,N}^u$ in\eqref{3.3} is \uline{semisimple} \cite{TK:1966} on the space $\ds \WqsuN$ of dimension $N = \ell_1 + \cdots + \ell_M$. How to treat the general case is known \cite{LT1:2015}, \cite{LT2:2015}, \cite{LPT.1}, \cite{LPT.3}. It is much more computationally intensive.\\
	
	\nin \textbf{Assumption:} Thus assume henceforth that
		\begin{equation}\label{4.11}
		\text{\nin $\BA_{q,N}^u$ is semisimple on $\ds \WqsuN$:}		
		\end{equation}
		that is, for the unstable eigenvalues $\ds \lambda_1, \dots, \lambda_M$ geometric and algebraic multiplicity coincide.\qedsymbol\\
	
	\nin Accordingly, let $\ds \bs{w} \in \WqsuN$, then \cite[p51]{TK:1966}
	\begin{equation}\label{4.12}
		\WqsuN \ni \bs{w}_N = \sum_{i, j = 1}^{M, \ell_i} \ip{\bs{w}_N}{\bs{\Phi}^*_{ij}}_{\Wqs,\Wqss}\bs{\Phi}_{ij}, \ \bs{\Phi}_{ij} = \bbm \bs{\varphi_{ij}} \\ \psi_{ij} \ebm.
	\end{equation}
	\nin Define $\beta_i$ and $\beta$ the following ordered bases of length $\ell_i$ and $N$, respectively, corresponding to the unstable eigenvalues
	\begin{subequations}\label{4.13}
		\begin{align}
			\beta_i &= \left[ \bs{\Phi}_{i1}, \dots, \bs{\Phi}_{i \ell_i} \right], \ i = 1, \dots, M \label{4.13a}\\
			\beta &= \beta_1 \cup \beta_2 \cup \dots \cup \beta_M = \left[ \bs{\Phi}_{11}, \dots, \bs{\Phi}_{1 \ell_1}, \bs{\Phi}_{21}, \dots, \bs{\Phi}_{2 \ell_2}, \dots, \bs{\Phi}_{M1}, \dots, \bs{\Phi}_{M \ell_M}  \right]. \label{4.13b}
		\end{align}
	\end{subequations}
	\nin Thus the matrix representation of the operator $\BA_{q,N}^u$ in \eqref{3.6a} on $\ds \WqsuN$ with respect to the basis $\beta$ is
	\begin{equation}\label{4.14}
		\left[ \BA_{q,N}^u \right]_{\beta} = \Lambda = \bbm
		\lambda_1 I_1 & & & \text{\huge0}\\
		& \lambda_2 I_2 & & &\\
		& & \ddots & &\\
		\text{\huge0} & & & \lambda_M I_M 
		\ebm: \ N \times N, \ I_i: \ell_i \times \ell_i.
	\end{equation}
	
	\nin Next, with reference to \eqref{3.6a} with $m \bs{u} \in \bs{L}^q(\omega), \ P_q(m\bs{u}) \in \bs{L}^q_{\sigma}(\omega)$ and $v \in L^q(\Gamma)$, we shall show that:
	\begin{align}
		P_N \bbm P_q(m\bs{u}) \\ \calB_q Dv \ebm &= \sum_{i, j = 1}^{M, \ell_i} \ip{P_N \bbm P_q(m\bs{u}) \\ \calB_q Dv \ebm}{\bs{\Phi^*_{ij}}}_{\Wqs,\Wqss} \bs{\Phi}_{\ij} \label{4.15}\\
		&= \sum_{i, j = 1}^{M, \ell_i} \left[ \ip{m \bs{u}}{\bs{\varphi}^*_{ij}}_{\bs{L}^q_{\sigma}(\omega),\bs{L}^{q'}_{\sigma}(\omega)} - \ip{v}{\left.\frac{\partial \psi_{ij}^*}{\partial \nu}\right|_{\Gamma}}_{L^q(\Gamma),L^{q'}(\Gamma)} \right] \bs{\Phi}_{ij} \label{4.16}
	\end{align}
	\nin In fact, since $\ds P_N^* \bs{\Phi}^*_{ij} = \bs{\Phi}^*_{ij}$ and $\ds P_{q'} \bs{\varphi}^*_{ij} = \bs{\varphi}^*_{ij}$, we obtain from \eqref{4.15}, recalling $\Wqs$ from \eqref{1.34}, $\ds \bs{\Phi}^*= \left[ \bs{\varphi_{ij}^*},\ \psi_{ij}^* \right]$:
	\begin{align}
		P_N \bbm P_q(m\bs{u}) \\ \calB_q Dv \ebm &= \sum_{i, j = 1}^{M, \ell_i} \left[ \ip{P_q(m \bs{u})}{\bs{\varphi}^*_{ij}}_{\bs{L}^q_{\sigma}(\omega),\bs{L}^{q'}_{\sigma}(\omega)} + \ip{v}{D^*\calB^*_q \psi_{ij}^*}_{L^q(\Gamma),L^{q'}(\Gamma)} \right] \bs{\Phi}_{ij} \label{4.17} \\
		&= \sum_{i, j = 1}^{M, \ell_i} \left[ \ip{m \bs{u}}{\bs{\varphi}^*_{ij}}_{\bs{L}^q_{\sigma}(\omega),\bs{L}^{q'}_{\sigma}(\omega)} + \ip{v}{\left. -\frac{\partial \psi_{ij}^*}{\partial \nu}\right|_{\Gamma}}_{L^q(\Gamma),L^{q'}(\Gamma)} \right] \bs{\Phi}_{ij} \label{4.18}
	\end{align}
	\nin and \eqref{4.16} is verified. In going from \eqref{4.17} to \eqref{4.18} we have recalled $\ds D^* \calB^*_qf = -\frac{\partial f}{\partial \nu}, \ f \in \calD(\calB^*_q)$, see \eqref{B.11} of Appendix \ref{app-B} for $\kappa = 1$. Next, we seek an interior control $\bs{u}$ acting on $\omega$ and a boundary control $v$ acting on $\wti{\Gamma}$ at first of the form
	\begin{align}
		\bs{u} &= \sum_{k = 1}^{K} \mu_k(t) \bs{u}_k, \ \bs{u}_k \in \WqsuN \subset \lso \label{4.19}\\
		v &= \sum_{k = 1}^{K} \nu_k(t)f_k \in \calF \subset W^{2-\rfrac{1}{q},q}(\wti{\Gamma}) \label{4.20}
	\end{align}
	\nin where 
	\begin{equation}\label{4.21}
		\calF = \text{span } \left\{ \left. \frac{ \partial \psi_{ij}^* }{\partial \nu}\right|_{\wti{\Gamma}}, \ i = 1,\dots,M; \ j = 1, \dots, \ell_i \right\} \subset W^{2-\rfrac{1}{q},q}(\wti{\Gamma}).
	\end{equation}
	\nin Thus substituting \eqref{4.19}, \eqref{4.20}, in \eqref{4.16} we obtain
	\begin{multline}
		P_N \bbm P_q(m\bs{u}) \\ \calB_q Dv \ebm = \sum_{i, j = 1}^{M, \ell_i} \sum_{k = 1}^K \left[ \ip{m \bs{u}_k}{\bs{\varphi}^*_{ij}}_{\bs{L}^q_{\sigma}(\omega),\bs{L}^{q'}_{\sigma}(\omega)}\mu_k(t) \right.\\ \left.- \ip{f_k}{\left.\frac{\partial \psi_{ij}^*}{\partial \nu}\right|_{\Gamma}}_{L^q(\Gamma),L^{q'}(\Gamma)}\nu_k(t) \right] \bs{\Phi}_{ij}. \label{4.22}
	\end{multline}
	\nin Accordingly, in view of \eqref{4.22}, we can rewrite the $\bs{w}_N$-problem \eqref{3.6a} as 
	\begin{multline}
		\text{on } \WqsuN: \bs{w}'_N = \BA_{q,N}^u \bs{w}_N + \sum_{i, j = 1}^{M, \ell_i} \sum_{k = 1}^K \left[ \ip{m \bs{u}_k}{\bs{\varphi}^*_{ij}}_{\bs{L}^q_{\sigma}(\omega),\bs{L}^{q'}_{\sigma}(\omega)}\mu_k(t)  \right.\\ -\left.\ip{f_k}{\left.\frac{\partial \psi_{ij}^*}{\partial \nu}\right|_{\Gamma}}_{L^q(\Gamma),L^{q'}(\Gamma)}\nu_k(t) \right] \bs{\Phi}_{ij}. \label{4.23}
	\end{multline}
	\nin Eq \eqref{4.23} is the perfect counterpart of \cite[Eq (6.24)]{LPT.3}.\\
	
	\nin On the basis of expansion \eqref{4.23}, we introduce the $\ell_i \times K$ matrix $W_i$ and the $\ell_i \times K$ matrix $U_i$, $i = 1, \dots, M$
	\begin{equation}\label{4.24}
		W_i = \bbm \ip{f_1}{\partial_{\nu} \psi^*_{i1}|_{\Gamma}}_{\wti{\Gamma}} & \cdots & \ip{f_K}{\partial_{\nu} \psi^*_{i1}|_{\Gamma}}_{\wti{\Gamma}} \\
		\ip{f_1}{\partial_{\nu} \psi^*_{i2}|_{\Gamma}}_{\wti{\Gamma}} & \cdots & \ip{f_K}{\partial_{\nu} \psi^*_{i2}|_{\Gamma}}_{\wti{\Gamma}}\\
		\vdots & & \vdots\\
		\ip{f_1}{\partial_{\nu} \psi^*_{i \ell_i}|_{\Gamma}}_{\wti{\Gamma}} & \cdots & \ip{f_K}{\partial_{\nu} \psi^*_{i \ell_i}|_{\Gamma}}_{\wti{\Gamma}} \ebm: \ \ell_i \times K \ \ip{ \ }{ \ }_{\wti{\Gamma}} = \ip{ \ }{ \ }_{L^q(\wti{\Gamma}), L^{q'}(\wti{\Gamma})} 
	\end{equation}
	\begin{equation}\label{4.25}
		U_i = \bbm \ip{\bs{u}_1}{\bs{\varphi}^*_{i1}}_{\omega} & \cdots & \ip{\bs{u}_K}{\bs{\varphi}^*_{i1}}_{\omega} \\
		\ip{\bs{u}_1}{\bs{\varphi}^*_{i2}}_{\omega} & \cdots & \ip{\bs{u}_K}{\bs{\varphi}^*_{i2}}_{\omega} \\
		\vdots & & \vdots\\
		\ip{\bs{u}_1}{\bs{\varphi}^*_{i \ell_i}}_{\omega} & \cdots & \ip{\bs{u}_K}{\bs{\varphi}^*_{i \ell_i}}_{\omega} \\ \ebm: \ \ell_i \times K \ \ip{ \ }{ \ }_{\omega} = \ip{ \ }{ \ }_{\bs{L}^q_{\sigma}(\omega)} 
	\end{equation}
	\begin{equation}\label{4.26}
		W = \norm{\ip{f_r}{\partial_{\nu} \psi^*_{ij}|_{\Gamma}}_{\wti{\Gamma}}} = \bbm W_1 \\ W_2 \\ \vdots \\ W_M \ebm; \ N \times K, \quad U = \norm{\ip{\bs{u}_r}{\bs{\varphi}^*_{ij}}_{\omega}} = \bbm U_1 \\ U_2 \\ \vdots \\ U_M \ebm: N \times K.
	\end{equation}
	\nin Next, we represent the $N$-dimensional vector $\bs{w}_N \in \WqsuN$ given by \eqref{4.12} as column vector $\hat{\bs{w}} = [\bs{w}_N]_{\beta}$, with respect to the basis $\beta$ in \eqref{4.13b}. Rewrite \eqref{4.12} as
	\begin{multline}\label{4.27}
		\bs{w}_N = \sum_{i, j = 1}^{M, \ell_i} \bs{w}^{ij}_{N} \bs{\Phi}_{ij}; \ \text{set } \hat{\bs{w}} = \text{col} \left[ \bs{w}^{11}_{N}, \dots, \bs{w}^{1 \ell_1}_{N}, \dots, \bs{w}^{i1}_{N}, \dots, \bs{w}^{i \ell_i}_{N}, \dots, \right.\\
		\left. \bs{w}^{M1}_{N}, \dots, \bs{w}^{M \ell_M}_{N}  \right]
	\end{multline}
	\nin $N = \ell_1 + \dots + \ell_M$. Thus, by \eqref{4.22}-\eqref{4.27} we can write 
	\begin{equation}\label{4.28}
		\bbm \ds P_N \bbm P_q(m\bs{u}) \\ \calB_q Dv \ebm \ebm_{\beta} = \bbm W_1 \\ W_2 \\ \vdots \\ W_M \ebm \hat{\nu}_K + \bbm U_1 \\ U_2 \\ \vdots \\ U_M \ebm \hat{\mu}_k: \ N \times 1, \quad \hat{\nu}_K = \bbm \nu^1_K \\ \vdots \\ \nu^k_K \\ \vdots \\ \nu^K_K \ebm, \ \hat{\mu}_K = \bbm \mu^1_K \\ \vdots \\ \mu^k_K \\ \vdots \\ \mu^K_K \ebm.
	\end{equation}
	\nin Then, in $\BC^N$, with respect to the basis $\beta$ of normalized eigenvectors of $\BA_{q,N}^u$, we may rewrite system \eqref{4.23} via \eqref{4.14} and \eqref{4.28} as
	\begin{equation}\label{4.29}
		\hat{\bs{w}}'_N = \Lambda \hat{\bs{w}}_N + W \hat{\nu}_K + U \hat{\mu}_K = \Lambda \hat{\bs{w}}_N + B \bbm \hat{\nu}_K \\ \hat{\mu}_K \ebm, \quad B = [W, U]: N \times 2K.
	\end{equation}
	\nin Then, the operator system \eqref{4.23} on $\ds \WqsuN$ has in \eqref{4.29} its representation in $\BC^N$, characterized by the pair $\ds \{ \Lambda, B \}$, with $\Lambda$ the free dynamics matrix and $B$ the control matrix. Our next goal is to guarantee that the pair $\ds \{ \Lambda, B \}$ is controllable.\\
	
	\nin We then obtain the following counterpart of \cite[Theorem 6.1]{LPT.3}.
	
	\begin{thm}\label{Thm-4.1}
		Assume \eqref{4.11}; that is, that $\BA_{q,N}^u$ is semisimple on $\WqsuN$. System \eqref{4.29} i.e. the pair  $\ds \{ \Lambda, B \}$ is controllable in case
		\begin{equation}\label{4.30}
			\text{rank } [W_i, \ U_i] = \ell_i, \ i = 1, \dots, M.
		\end{equation}
		\nin In fact, with reference to \eqref{4.19} and \eqref{4.20}, it is possible to select boundary vectors $f_1, \dots, f_K \in \calF \subset W^{2-\rfrac{1}{q}, q}(\wti{\Gamma})$ with support on $\wti{\Gamma}$, and interior vectors $\bs{u}_1, \dots, \bs{u}_K \in \bs{L}^q_{\sigma}(\omega)$ with support on $\omega$ such that \eqref{4.30} holds true, specifically 
		\begin{multline}\label{4.31}
			\text{rank } [W_i, U_i] = \\ \text{rank} \bbm \ip{f_1}{\partial_{\nu} \psi^*_{i1}|_{\Gamma}}_{\wti{\Gamma}} & \cdots & \ip{f_K}{\partial_{\nu} \psi^*_{i1}|_{\Gamma}}_{\wti{\Gamma}} & \ip{\bs{u}_1}{\bs{\varphi}^*_{i1}}_{\omega} & \cdots & \ip{\bs{u}_K}{\bs{\varphi}^*_{i1}}_{\omega}\\
			\begin{picture}(0,0)\multiput(164,22)(0,-9){7}{\line(0,-1){6}}\end{picture}
			\ip{f_1}{\partial_{\nu} \psi^*_{i2}|_{\Gamma}}_{\wti{\Gamma}} & \cdots & \ip{f_K}{\partial_{\nu} \psi^*_{i2}|_{\Gamma}}_{\wti{\Gamma}} & \ip{\bs{u}_1}{\bs{\varphi}^*_{i2}}_{\omega} & \cdots & \ip{\bs{u}_K}{\bs{\varphi}^*_{i2}}_{\omega}\\
			\vdots & & \vdots & \vdots & & \vdots\\
			\ip{f_1}{\partial_{\nu} \psi^*_{i \ell_i}|_{\Gamma}}_{\wti{\Gamma}} & \cdots & \ip{f_K}{\partial_{\nu} \psi^*_{i \ell_i}|_{\Gamma}}_{\wti{\Gamma}} & \ip{\bs{u}_1}{\bs{\varphi}^*_{i \ell_i}}_{\omega} & \cdots & \ip{\bs{u}_K}{\bs{\varphi}^*_{i \ell_i}}_{\omega} \ebm \\ = \ell_i, \ i = 1, \dots, M. 
		\end{multline}
	\end{thm}
	\begin{proof}
		This is the counterpart of the proof \cite[Theorem 6.1]{LPT.2}. As in that paper, the present proof rests on a suitable Unique Continuation Property, Theorem \ref{Thm-B.1} of Appendix \ref{app-B} as documented below.\\
		
		\nin In general. in seeking that the $\ell_i$ rows (of length $2K$) in matrix \eqref{4.31} be linearly independent, we see that the full rank statement \eqref{4.31} will hold true if and only if we can exclude that each of the two sets of vectors
		\begin{equation}\label{4.32}
			\left\{ \partial_{\nu} \psi_{i1}^*, \dots, \partial_{\nu} \psi_{i \ell_i}^* \right\} \text{ in } L^q(\wti{\Gamma}) \quad \text{and} \quad \left\{ \bs{\varphi}_{i1}^*, \dots, \bs{\varphi}_{i \ell_i}^* \right\} \text{ in } \bs{L}^q_{\sigma}(\omega)
		\end{equation}
		\nin are linearly independent with the same linear independence relation in the two cases; that is, if and only if we can establish that \textit{we cannot} have simultaneously 
		\begin{equation}\label{4.33}
			\partial_{\nu} \psi_{i \ell_i}^* = \sum_{j = 1}^{\ell_i - 1} \alpha_j \partial_{\nu} \psi_{ij}^* \text{ in } L^q(\wti{\Gamma}) \quad \text{and} \quad \bs{\varphi}_{i \ell_i}^* = \sum_{j = 1}^{\ell_i - 1} \alpha_j \bs{\varphi}_{ij}^*  \text{ in } \bs{L}^q_{\sigma}(\omega)
		\end{equation}
		\nin with the same constants $\alpha_1, \dots, \alpha_{\ell_i - 1}$ in both expansions \cite{LPT.2}.\\
		
		\nin \uline{Claim: Statement \eqref{4.33} is false.} By contradiction, suppose that both linear combinations in \eqref{4.33} hold true. Define the $(d+1)$-vector $\bs {\Phi}^* = \{ \bs{\varphi}^*, \psi^* \}$ (depending on $i$) on $\lso \times L^q(\Omega)$ by setting
		\begin{equation}\label{4.34}
			\bs{\Phi}^* \equiv \bbm \bs{\varphi}^* \\ \psi^* \ebm \equiv \sum_{j = 1}^{\ell_i - 1} \bbm \alpha_j \bs{\varphi}_{ij}^* \\ \alpha_j \psi_{ij}^* \ebm - \bbm \bs{\varphi}_{i \ell_i}^* \\ \psi_{i \ell_i}^*\ebm = \sum_{j = 1}^{\ell_i - 1} \alpha_j \bs{\Phi}^*_{ij} - \bs{\Phi}^*_{i \ell_i}, \ i = 1, \dots, M; \ q \geq 2.
		\end{equation}
		\nin Thus, in view of \eqref{4.33}, we obtain
		\begin{equation}\label{4.35}
			\bs{\varphi}^* \equiv 0 \text{ in } \omega, \quad \partial_{\nu} \psi^*|_{\wti{\Gamma}} \equiv 0.
		\end{equation}
		\nin Moreover, since $\ds \bs{\Phi}^*_{ij} = \{ \bs{\varphi}^*_{ij}, \psi^*_{ij} \}$ is an eigenvector of the operator $\ds \BA_{q,N}^*$ corresponding to the (unstable) eigenvalue $\bar{\lambda}_i$, then so is $\ds  \bs{\Phi}^*: \ \BA_{q,N}^* \bs{\Phi}^* = \bar{\lambda}_i \bs{\Phi}^*$. Then $\ds \bs{\Phi}^* = \{ \bs{\varphi}^*, \psi^* \}$ satisfies the corresponding PDE version of the eigenvector identity as given by \eqref{4.10}, see \eqref{B.4} in Appendix \ref{app-B}
		\begin{subequations}\label{4.36}
			\begin{empheq}[left=\empheqlbrace]{align}
			-\nu \Delta \boldsymbol{\varphi^*} + L^*_e(\boldsymbol{\varphi^*})+ \psi^* \nabla \theta_e + \nabla \pi &= \bar{\lambda}_i \boldsymbol{\varphi^*}   &\text{in } \Omega \label{4.36a} \\
			-\kappa \Delta \psi^* - \bs{y}_e \cdot \nabla \psi^* - \gamma  \boldsymbol{\varphi^*}  \cdot \bs{e}_d &= \bar{\lambda}_i \psi^* &\text{in } \Omega \label{4.36b}\\
			\text{div } \boldsymbol{\varphi^*} &= 0   &\text{in } \Omega  \label{4.36c} \\
			\boldsymbol{\varphi^*}|_{\Gamma} = 0, \ \psi^*|_{\Gamma} &= 0 &\text{on } \Gamma.    \label{4.36d}
			\end{empheq}
			\nin and in addition, the over-determined conditions in \eqref{4.35}:
			\begin{equation}\label{4.36e}
			\bs{\varphi}^* \equiv 0 \text{ in } \omega, \quad \partial_{\nu} \psi^*|_{\wti{\Gamma}} \equiv 0.
			\end{equation}		
		\end{subequations}
	\nin Write $\ds \bs{\varphi}^* = \{ \varphi^{*(1)}, \varphi^{*(2)}, \dots, \varphi^{*(d)} \}$ for the $d$-components of the vector $\bs{\varphi}^*$. Since $\bs{e}_d = \{ 0, \dots, 0, 1 \}$, the property $\varphi^{*d} = 0$ in $\omega$, contained in \eqref{4.36e}, used in \eqref{4.36b} implies:
	\begin{subequations}\label{4.37}
		\begin{empheq}[left=\empheqlbrace]{align}
		\kappa \Delta \psi^* + \bs{y}_e \cdot \nabla \psi^* &= -\bar{\lambda}_i \psi^* &\text{in } \omega \label{4.37a}\\
		\psi^*|_{\wti{\Gamma}} = 0, \ \partial_{\nu} \psi^*|_{\wti{\Gamma}} &= 0 &\text{on } \wti{\Gamma}. \label{4.37b}
		\end{empheq}
	\end{subequations}
	\nin recalling the boundary conditions for $\psi^*$ in \eqref{4.36d} and \eqref{4.36e} for the $\psi^*$-problem defined on $\omega$, with $\wti{\Gamma} \subset \partial \omega$, as in Fig 1. It is then a standard result \cite{Car}, \cite[Sect 19, pp 59-61]{M.1}, \cite[p 3]{H}, \cite{Kom.1} that the over-determined problem \eqref{4.37} implies
	\begin{equation}\label{4.38}
		\psi^* \equiv 0 \text{ in } \omega.
	\end{equation}
	\nin Then \eqref{4.38} and the full strength of \eqref{4.36d} give the over-determination
	\begin{equation}
	\bs{\varphi}^* \equiv 0 \text{ in } \omega, \ \psi^* \equiv 0 \text{ in } \omega
	\end{equation}
	\nin for problem \eqref{4.36}. By means of such over-determination, we can apply the Unique Continuation Property \cite[Theorem 5]{TW.1} recalled as Theorem \ref{Thm-B.1} in Appendix \ref{app-B} and conclude that
	\begin{equation}\label{4.40}
		\bs{\Phi}^* = \{ \bs{\varphi}^*, \psi^* \} \equiv 0 \text{ in } \Omega, \quad \pi \equiv \text{const.} \text{ in } \Omega,
	\end{equation}
	\nin or recalling \eqref{4.34}
	\begin{equation}
		\bs{\Phi}^*_{i \ell_i} = \sum_{j = 1}^{\ell_i - 1} \alpha_j \bs{\Phi}^*_{ij} \text{ in } \Wqs(\Omega) = \lso \times L^q(\Omega);
	\end{equation}
	\nin i.e. the set $\ds \{ \bs{\Phi}^*_{i1}, \dots, \bs{\Phi}^*_{i \ell_i} \}$ is linearly dependent on $\ds \Wqs(\Omega)$. But this is false by the very selection of such eigenvectors, see above \eqref{4.1}. Thus the two conditions in \eqref{4.33} cannot hold simultaneously. The claim is proved. Hence it is possible to select, in infinitely many ways, boundary functions $\ds f_1, \dots, f_K \in \calF \subset L^q(\Gamma)$ for $q \geq 2$ and interior $d$-vectors $\bs{u}_1, \dots, \bs{u}_K$ in $\lso$, such that the Kalman algebraic full rank conditions \eqref{4.31} hold true. Indeed they may be chosen independent of $i$. Start with \eqref{4.31} for $\bar{\imath}$ such that $\ds \ell_{\bar{\imath}} = K = \max \left\{ \ell_i, \ i = 1, \dots, M \right\}$ yielding $\ds f_1, \dots, f_K$ which are linearly independent. They also work in other $i$'s.
	\end{proof}
	
	\begin{rmk}\label{Rmk-4.1}
		The UCP in \cite[Theorem 5]{TW.1} reported as Theorem \ref{Thm-B.1} in Appendix \ref{app-B}, refers to problem (\hyperref[4.36]{4.36a-c}), without the B.C. \eqref{4.36d}, with the a-priori over-determination $\psi^* \equiv 0$ in $\omega$, and $\ds \{ \varphi^{*(1)}, \dots, \varphi^{*(d-1)} \} \equiv 0$ in $\omega$, thus not involving the last $d$-component $\varphi^{*(d)} \equiv 0$. In contrast, in the setting of proving the above Claim, we need $\varphi^{*(d)} \equiv 0$ in $\omega$ in order to deduce $\psi^* \equiv 0$ in $\omega$. The present proof on the validity of the rank conditions \eqref{4.31}, ultimately relying on the UCP of \cite[Theorem 5]{TW.1} in the Appendix \ref{app-B} requires the full strength of the fluid vectors $\bs{u}_1, \dots, \bs{u}_{\ell_i}$, each possessing $d$-components. No geometrical assumptions are involved for the pair $\{ \omega, \wti{\Gamma} \}$, except for $\omega$ being an interior subdomain touching the boundary $\wti{\Gamma}$ as in Fig 1. In Appendix \ref{app-C}, we report an improvement requiring only $(d-1)$ components from the fluid vectors $\bs{u}_1, \dots, \bs{u}_{\ell_i}$ to include necessarily the $d$\textsuperscript{th} components. \qedsymbol		
	\end{rmk}

	\begin{rmk}\label{Rmk-4.2}
		We have established Theorem \ref{Thm-4.1} under the simplifying semisimple assumption \eqref{4.11}. However, Theorem \ref{Thm-4.1} holds true in full generality. The corresponding proof is lengthy and technical. It may be given by following the scheme given in \cite{LT1:2015}, \cite{LT2:2015}, \cite{LPT.3} for the Navier-Stokes equations. It requires use of the controllability criterion for a finite dimensional pair $\{ A, B \}$ with $A$ in Jordan canonical form \cite{CTC:1984}, \cite{B-M1}. Henceforth, we proceed by taking that Theorem \ref{Thm-4.1} is true in full generality. \qedsymbol
	\end{rmk}

	\nin As a consequence Theorem \ref{Thm-4.1}, we obtain the following fundamental result on the uniform stabilization with arbitrarily preassigned decay rate, of the original unstable, finite dimensional $\bs{w}_N$-problem \eqref{3.6a}, the counterpart of \cite[Theorem 5.1]{LPT.2}.
	
	\begin{thm}\label{Thm-4.2}
		Let $\lambda_1,\ldots,\lambda_M$ be the unstable distinct eigenvalues of the operator $\ds \BA_q \ (= \BA_{q,N}^u )$ as in \eqref{1.34} with geometric multiplicity $\ell_i, \ i = 1, \dots, M$, and set $K = \sup \{ \ell_i; i = 1, \dots, M \}$. Let $\wti{\Gamma} $ be an open connected subset of the boundary $\Gamma$ of positive surface measure and $\omega$ be a localized collar supported by $\wti{\Gamma}$ (Fig.~1). Let $q \geq 2$. Given $\gamma_1 > 0$ arbitrarily large, we can construct two $K$-dimensional controllers: a boundary control $v = v_N$ acting with support on $\wti{\Gamma}$, of the form given by 				
		\begin{equation}\label{4.42}
		v= v_N = \sum^K_{k=1} \nu_k(t) f_k, \ f_k \in \calF \subset W^{2-\rfrac{1}{q},q}(\Gamma), \ q \geq 2,
		\end{equation}		
		\noindent $\calF$ defined in \eqref{4.21}, $q \geq 2$, $f_k$ supported on $\wti{\Gamma}$, and an interior control $\bs{u} = \bs{u}_N$ acting on $\omega$, of the form given by		
		\begin{equation}\label{4.43}
		\bs{u} = \bs{u}_N = \sum_{k = 1}^K \mu_k(t) \bs{u}_k, \quad \bs{u}_k \in \WqsuN \subset \lso, \quad \mu_k(t) = \text{scalar,}
		\end{equation}
		\noindent thus with interior vectors $[\bs{u}_1, \dots, \bs{u}_K]$ in the smooth subspace $\WqsuN$ of $\lso, \ 2 \leq q < \infty$, supported on $\omega$, such that, once inserted in the finite dimensional projected $\bs{w}_N$-system in \eqref{3.6a}, yields the system
		\begin{equation}\label{4.44}
		\bs{w}'_N = \BA^u_{q,N} \bs{w}_N + P_N \bbm P_q \left( m \left( \ds \sum_{k = 1}^{K} \mu_k(t) \bs{u}_k \right) \right) \\ \calB_qD \left( \ds \sum_{k = 1}^{K} \nu_k(t) f_k \right) \ebm ,
		\end{equation}
		\noindent whose solution then satisfies the estimate
		\begin{multline}\label{4.45}
		\norm{\bs{w}_N(t)}_{\lso} + \norm{v_N(t)}_{L^q(\wti{\Gamma})} + \norm{v'_N(t)}_{L^q(\wti{\Gamma})} + \\ \norm{\bs{u}_N(t)}_{\bs{L}^q_{\sigma}(\omega)} + \norm{\bs{u}'_N(t)}_{\bs{L}^q_{\sigma}(\omega)} \leq C_{\gamma_{1}} e^{-\gamma_1 t} \norm{P_Nw_0}_{\Wqs(\Omega)}, \quad t \geq 0.
		\end{multline}
		
		\noindent Moreover, such controllers $ v = v_N $ and $ \bs{u} = \bs{u}_N $ may be chosen in feedback form: that is, with reference to the explicit expressions \eqref{4.42} for $v$ and \eqref{4.43} for $\bs{u}$, of the form $\nu_k(t) = \ip{\bs{w}_N(t)}{\bs{p}_k}_{_{\WqsuN}}$ and $\mu_k(t) = \ip{\bs{w}_N(t)}{\bs{q}_k}_{_{\WqsuN}}$ for suitable vectors $\bs{p}_k \in \left( \WqsuN \right)^* \subset \lo{q'} \times L^{q'}(\Omega)$, $\bs{q}_k \in \left( \WqsuN \right)^* \subset \lo{q'} \times L^{q'}(\Omega)$ depending on $\gamma_1$, where $\ip{ \ }{\ }$ denotes the duality pairing $\ds \WqsuN \times \left( \WqsuN \right)^*$.\\
		
		\noindent In conclusion, $\bs{w}_N$ in \eqref{4.45} is the solution of the equation \eqref{4.44} on $\WqsuN$ rewritten explicitly as
		\begin{equation}\label{4.46}
		\bs{w}'_N = \BA^u_{q,N} \bs{w}_N + P_N \bbm P_q \left( m \left( \ds \sum_{k = 1}^{K} \ip{\bs{w}_N(t)}{\bs{q}_k}_{\WqsuN} \bs{u}_k \right) \right) \\ \calB_q D \left( \ds \sum_{k = 1}^{K} \ip{\bs{w}_N(t)}{\bs{p}_k}_{\WqsuN} f_k \right) \ebm ,
		\end{equation}
		\noindent $f_k$ supported on $\wti{\Gamma}$, $\bs{u}_k$ supported on $\omega$, rewritten in turn as
		\begin{equation}\label{4.47}
		\bs{w}'_N = \overline{A}^u \bs{w}_N, \ \bs{w}_N(t) = e^{\overline{A}^{u}t} P_N \bs{w}_0, \ \bs{w}_N(0) = P_N \bs{w}_0 \quad \text{on } \WqsuN.
		\end{equation}
	\end{thm}
	\begin{proof}
		The technical proof is similar circumstances was given in \cite{LT1:2015}, \cite{LT2:2015}, \cite{LPT.3}. Thus, we give here only some insight. The key new fact is the controllability test of Theorem \ref{Thm-4.1}. Having established the controllability condition for the pair $\{ \Lambda, B \}$, then by the well-known Popov's criterion on finite dimensional theory (see e.g. \cite[Theorem 2.9 p 44]{Za}), there exists a feedback matrix $Q: 2K \times N$, such that the spectrum of the matrix $[\Lambda + BQ]$ may be arbitrarily preassigned, in particular, to lie arbitrarily on the left half-plane $\ds \{\lambda: Re \ \lambda < -\gamma_1 < - Re \ \lambda_{N+1} \}$, with $\gamma_1 > 0$ preassigned, as desired. The resulting closed-loop system is
	\begin{equation}\label{4.48}
		(\hat{\bs{w}}_{N})' = \Lambda \hat{\bs{w}}_N + BQ \hat{\bs{w}}_N, \ \bbm \hat{\nu}_K \\ \hat{\mu}_K \ebm = Q \hat{\bs{w}}_N = \text{ feedback control}
	\end{equation}
	\begin{equation}\label{4.49}
		Q = \bbm \text{ row } \hat{p}_1 \\ \vdots \\ \text{ row } \hat{p}_K \\ \text{ row } \hat{q}_1 \\ \vdots \\ \text{ row } \hat{q}_K \ebm: 2K \times N; \ \hat{\nu}_K = \bbm \nu^1_K = \ip{\hat{\bs{p}}_1}{\hat{\bs{w}}_N} \\ \vdots \\ \nu^k_K = \ip{\hat{\bs{p}}_k}{\hat{\bs{w}}_N} \\ \vdots \\ \nu^K_K = \ip{\hat{\bs{p}}_K}{\hat{\bs{w}}_N} \ebm, \ \hat{\mu}_K = \bbm \mu^1_K = \ip{\hat{\bs{q}}_1}{\hat{\bs{w}}_N} \\ \vdots \\ \mu^k_K = \ip{\hat{\bs{q}}_k}{\hat{\bs{w}}_N} \\ \vdots \\ \mu^K_K = \ip{\hat{\bs{q}}_K}{\hat{\bs{w}}_N} \ebm
	\end{equation}
	\nin in the $\BC^N$-inner product. Thus, returning from $\BC^N \times \BC^N$ back to $\ds \WqsuN \times \left[ \WqsuN \right]^*$, there exist vectors $\bs{p}_1, \dots, \bs{p}_K$ and $\bs{q}_1, \dots, \bs{q}_K$ in $\ds \left[ \WqsuN \right]^*$ such that
	\begin{equation}\label{4.50}
		\nu_K^k = \ip{\bs{w}_N}{\bs{p}_k}, \ \mu_K^k = \ip{\bs{w}_N}{\bs{q}_k}, \ k = 1, \dots, K
	\end{equation}
	\nin where $\ip{ \ }{ \ }$ denotes the duality paring $\ds \WqsuN \times \left[ \WqsuN \right]^*$. This way, the closed loop system \eqref{4.46} corresponds to the $\BC^N$-system \eqref{4.48}. \eqref{4.49}.
	\end{proof}

	\section{The linearized $\bs{w}$-system \eqref{1.32} in feedback form.}\label{Sec-5}
	
	\nin We return to the open-loop linearized $\bs{w}$-problem \eqref{1.31}, \eqref{1.32b}, and use the same corresponding feedback operators
	\begin{align}
	v &= F \bs{w} = \sum_{k = 1}^K \ip{P_N \bs{w}}{\bs{p}_k}f_k, \ f_k \in \calF \subset W^{2 - \rfrac{1}{q},q}(\Gamma), \nonumber\\ & \hspace{1cm} \bs{p}_k \in (\bs{W}^u_N)^* \subset \bs{L}^{q'}_{\sigma}(\Omega) \times L^q(\Omega), \ q \geq 2,\ f_k \text{ supported on } \wti{\Gamma}.  \label{5.1}\\ 
	J \bs{w} &= P_qm(\bs{u}) = P_q m \left(\sum_{k = 1}^K \ip{P_N \bs{w}}{\bs{q}_k} \bs{u}_k\right), \nonumber\\ & \hspace{1cm} \bs{q}_k \in (\bs{W}^u_N)^* \subset \bs{L}^{q'}_{\sigma}(\Omega) \times L^q(\Omega), \ \bs{u}_k \text{ supported on } \omega, \label{5.2}
	\end{align}
	\nin that were employed in \eqref{4.46} to uniformly stabilizing the finite dimensional $\bs{w}_N$-problem \eqref{3.6a} with arbitrarily preassigned decay $\gamma_1$ in \eqref{4.45}. See also \eqref{2.2b}, \eqref{2.3}. We thus obtain from \eqref{1.32b} the resulting closed-loop linearized $\bs{w}$-problem in feedback form
	\begin{subequations}\label{5.3}
	\begin{equation}\label{5.3a}
	\frac{d \bs{w}}{dt} = \bbm \calA_q & -\calC_{\gamma} \\ -\calC_{\theta_e} & 0 \ebm \bs{w} + \bbm P_q \left( m \ds \sum_{k = 1}^K \ip{P_N \bs{w}}{\bs{q}_k}\bs{u}_k \right) \\ -\calB_q \left(w_2  - D \ds \sum_{k = 1}^K \ip{P_N \bs{w}}{\bs{p}_k}f_k\right) \ebm 
	\end{equation}
	\nin (compare with \eqref{4.46}) or by \eqref{5.1}, \eqref{5.2}
	\begin{equation}\label{5.3b}
	\frac{d \bs{w}}{dt} = \bbm \calA_q & -\calC_{\gamma} \\ -\calC_{\theta_e} & 0 \ebm \bs{w} + \bbm J \bs{w} \\ -\calB_q (w_2 - DF \bs{w}) \ebm \equiv \BA_{_{F,q}} \bs{w}
	\end{equation}
	\end{subequations}
	\nin as in \eqref{2.4b}. Thus $\ds \BA_{_{F,q}}$ defines the linearized $\bs{w}$-problem in feedback form, while $\BA_q$ in \eqref{1.34} is the free dynamics operator. We rewrite \eqref{5.3} as
	\begin{equation}\label{5.4}
	\frac{d \bs{w}}{dt} = \BA_{_{F,q}} \bs{w} = \hat{\BA}_{_{F,q}} \bs{w} + \Pi \bs{w}
	\end{equation}
	\nin where $\ds \hat{\BA}_{_{F,q}}$ is the streamlined operator that removes benign terms from $\ds \BA_{_{F,q}}$ and moves them into the perturbation operator $\Pi$
	\begin{subequations}\label{5.5}
	\begin{align}
	\hat{\BA}_{_{F,q}} \bs{w} &= \bbm -A_q \bs{w}_1 \\ -\calB_q \left( w_2 - DF \bs{w} \right) \ebm, \quad \Pi \bs{w} = \bbm A_{o,q} & -\calC_{\gamma} \\ -\calC_{\theta_e} & 0 \ebm \bs{w} + \bbm J \bs{w} \\ 0 \ebm \label{5.5a}\\
	\calD \left( \BA_{_{F,q}} \right) &= \calD \left( \hat{\BA}_{_{F,q}} \right) = \left\{ \bs{w} = \bbm \bs{w}_1  \\  w_2 \ebm \in \bs{W}^q_{\sigma}(\Omega) = \bls \times L^q(\Omega): \bs{w}_1 \in \calD(A_q), \right. \nonumber \\ &\left( w_2 - DF\bs{w} \right) \in \calD(\calB_q) \bigg\} \subset \bs{W}^{2,q}(\Omega) \cap \bs{W}^{1,q}_0(\Omega) \cap \lso \times W^{2,q}(\Omega) \label{5.5b} \\
	\calD(\calB_q) &= \calD(B_q) = W^{2,q}(\Omega) \cap W^{1,q}_0(\Omega); \ 
	 \calD(\Pi) = \calD(A_{o,q}) \times L^q(\Omega); \nonumber \\ & \hspace{5cm} Dv = DF\bs{w} \in W^{2,q}(\Omega) \label{5.5c}
	\end{align}
	\end{subequations}
	\nin recalling \eqref{1.14}, \eqref{1.15}, \eqref{1.18}, \eqref{1.23c}, \eqref{2.2b}.
	
	\section{The linearized feedback operator $\ds \BA_{_{F,q}}$ is the generator of a s.c. analytic, uniformly stable semigroup $\ds e^{\BA_{_{F,q}}t}$ on $\ds \Wqs(\Omega)$ and $\ds \VbqpO$.}
	
	\begin{thm}\label{Thm-6.1}
		The operator $\ds \BA_{_{F,q}}$ in \eqref{5.3b} generates a s.c., analytic semigroup $\ds e^{\BA_{_{F,q}}t}$ on $\ds \Wqs(\Omega) \equiv \lso \times L^q(\Omega)$ as well as in $\VbqpO \equiv \Bto \times B^{2-\rfrac{2}{p}}_{q,p}(\Omega)$.
	\end{thm}
	\begin{proof} \uline{First on $\Wqs(\Omega)$:} As the perturbation operator $\Pi$ in \eqref{5.5a} involves bounded operators $\ds \calC_{\gamma}, \ \calC_{\theta_e}, \ J$ (see \eqref{1.21}, \eqref{1.22}, \eqref{5.1}), as well as the benign operator $\ds A_{o,q}$, in \eqref{1.15} with $\ds A_{o,q} A^{-\rfrac{1}{2}}_q$ bounded on $\lso$, it suffices to show that the streamlined operator $\ds \hat{\BA}_{_{F,q}}$ in \eqref{5.5a} generates a s.c. analytic semigroup $\ds e^{\BA_{_{F,q}}t}$ on $\ds \Wqs(\Omega)$. To this end, it will be more convenient to show equivalently that the adjoint $\hat{\BA}_{_{F,q}}^*$ is the generator of a s.c. analytic semigroup $\ds e^{\hat{\BA}^*_{_{F,q}}t}$ on $\ds \left[ \Wqs(\Omega) \right]^* = \bs{W}_{\sigma}^{q'}(\Omega) = \lo{q'} \times L^{q'}(\Omega)$, since $\ds \Wqs(\Omega), \ 1 < q < \infty$, is a reflexive space. The adjoint $\hat{\BA}_{_{F,q}}^*$ of $\hat{\BA}_{_{F,q}}$ in \eqref{5.5a} is given by
		\begin{subequations}
			\begin{empheq}[left=\empheqlbrace]{align}
				\hat{\BA}_{_{F,q}}^* \bbm \bs{v}_1 \\ v_2 \ebm &= \bbm -A^*_q & 0 \\ 0 & -\calB^*_q \ebm \bbm \bs{v}_1 \\ v_2 \ebm + F^*D^*\calB^*_qv_2 \label{6.1a}\\
				\calD \left( \hat{\BA}^*_{_{F,q}} \right) &= \calD(A^*_q) \times \calD(B^*_q), \ \calD(B^*_q) = \calD(\calB^*_q) \label{6.1b} 
			\end{empheq}
		\end{subequations} 
	\nin by \eqref{1.23c}. Since the operator
	\begin{equation}\label{6.2}
		\hat{\BA}_q^* = \bbm -A^*_q & 0 \\ 0 & -\calB^*_q \ebm, \ \bs{W}_{\sigma}^{q'}(\Omega) \supset \calD \left( \hat{\BA}^*_q \right) = \calD(A_q^*) \times \calD(\calB_q^*) \longrightarrow \bs{W}_{\sigma}^{q'}(\Omega)
	\end{equation}
	\nin is plainly the generator of a s.c. analytic semigroup $\ds e^{\hat{\BA}^*_{_{F,q}}t}$ on $\ds \bs{W}^{q'}_{\sigma}(\Omega)$ (Appendix \ref{app-A}), it will suffice to show that the perturbation operator
	\begin{equation}\label{6.3}
		F^* D^* \calB_q^* \text{ is } \left( \hat{\BA}_q^* \right)^{\theta_0} = \bbm -A^{*\theta_0}_q & 0 \\ 0 & -\calB^{*\theta_0}_q \ebm-\text{bounded, for some constant } 0 < \theta_0 < 1
	\end{equation}
	\cite{P:1983}. In our case, it will be $\ds \theta_0 = 1 - \rfrac{1}{2q} + \varepsilon < 1$. In fact, recall via \eqref{1.23b} that $\ds D^* \calB^{*\gamma}_q \in \calL(L^{q'}(\Omega),L^{q'}(\Gamma)), \ \gamma = \rfrac{1}{2q} - \epsilon$. (In \eqref{6.3} and below we are taking that the fractional powers of $\calB^*_q$ are well-defined, for otherwise a translation will do it. Alternatively, in \eqref{6.1a}, write $\ds \calB^*_q = B^*_q + B^*_{o,q}$ from \eqref{1.18} as in \eqref{B.10} for $\kappa = 1$ where $\ds \calD \left(B^*_{o,q} \right) = \calD \left(B^{*\rfrac{1}{2}}_q \right)$. Do the argument with $B^*_q$ whose fractional powers are well-defined. This introduces an additional perturbation which is benign, as it can be handled by using 
	\begin{equation*}
		\norm{B^*_{o,q}v_2} = \norm{\left(B^*_{o,q}B^{*-\rfrac{1}{2}}_q \right) \left(B^{*-\rfrac{1}{2}}_q B^*_q \right)v_2} \leq C \norm{B^*_q v_2}
	\end{equation*}
	as desired.) Next, for $\ds v_2 \in \calD(\calB^{*1-\gamma}_q)$ and $\ds \bs{v}_1 \in \calD(A^{*1-\gamma}_q)$, we estimate since $\ds F^* \in \calL\left( L^{q'}(\Gamma), \bs{W}^{q'}_{\sigma}(\Omega) \right)$ in the norm of $\bs{W}_{\sigma}^{q'}(\Omega)$:
	\begin{align}
	\norm{F^*D^* \calB^*_q v_2} &= \norm{\left( F^*D^* \calB^{*^{\gamma}}_q \right) \calB^{*^{1-\gamma}}_q v_2} \leq C \norm{ \calB^{*^{1-\gamma}}_q v_2} \label{6.4}\\
	&\leq C \left[ \norm{ \calB^{*^{1-\gamma}}_q v_2} + \norm{ A^{*^{1-\gamma}}_q \bs{v}_1} \right] \label{6.5}\\
	&\leq C \norm{\bbm A^*_q & 0 \\ 0 & B^*_q \ebm^{1 - \gamma} \bbm \bs{v}_1 \\ v_2 \ebm} \label{6.6}
	\end{align}
	\nin and \eqref{6.6} proves \eqref{6.3}. In conclusion: the operator $\ds \BA^*_{_{F,q}}$ generates a s.c. analytic semigroup on $\ds \left( \Wqs(\Omega)^* \right) = \lo{q'} \times L^{q'}(\Omega)$; and hence as this space is reflexive, the operator $\BA_{_{F,q}}$ generates a s.c. analytic semigroup on $\ds \Wqs(\Omega) = \lso \times L^q(\Omega)$.\\
	
	\nin \uline{Next on $\VbqpO \equiv \Bto \times B^{2-\rfrac{2}{p}}_{q,p}(\Omega)$:} (See Remark \ref{Rmk-1.3}) 
	From the above conclusion on $\ds \BA_{_{F,q}}^*$ on $\left( \WqsO \right)^*$, it then follows that $\ds \BA_{_{F,q}}^*$ generates a s.c. analytic semigroup on $\ds \calD(\BA_{_{F,q}}^*) \equiv \calD(A^*_q) \times \calD(B^*_q)$, see \eqref{6.1b}. Then $\BA_{_{F,q}}^*$ generates a s.c. analytic semigroup on the real interpolation spaces between $\ds \calD(A^*_q) \times \calD(B^*_q)$ and $\ds \lo{q'} \times L^{q'}(\Omega)$, thus on the space $\ds \wti{\bs{B}}^{2-\rfrac{2}{p}}_{q',p}(\Omega) \times B^{2-\rfrac{2}{p}}_{q',p}(\Omega), \ 1 < p < \rfrac{2q'}{2q'-1}$. Recall the version of \eqref{1.11} for adjoints and get
	\begin{multline*}
		\left( \lo{q'}, \calD(A^*_q) \right)_{1-\frac{1}{p},p} \equiv \wti{\bs{B}}^{2-\rfrac{2}{p}}_{q',p}(\Omega), \ 1 < p < \frac{2q'}{2q'-1}, \\ 1 < q' < 2, \ q > 2, \ \frac{1}{q} + \frac{1}{q'} = 1. 
	\end{multline*}
	\nin For the second component recall Remark \ref{1.3}. It then finally follows, again by reflexivity of the space, that $\ds \BA_{_{F,q}}$ generates a s.c. analytic semigroup on the space $\ds \VbqpO \equiv \Bto \times B^{2-\rfrac{2}{p}}_{q,p}(\Omega)$.
	\end{proof}

	\nin The next result proves the sought-after uniform stabilization of the localized feedback $\bs{w}$-system \eqref{5.3} or \eqref{5.4}; that is, that the s.c. analytic semigroup $\ds e^{\BA_{_{F,q}}t}$ on $\ds \Wqs(\Omega),\ t \geq 0$ or on $\ds \VbqpO$, as guaranteed by Theorem \ref{Thm-6.1}, is uniformly stable.
	
	\begin{thm}\label{Thm-6.2}
		Under the same setting of Theorem \ref{Thm-4.2} concerning in particular the choice of the vectors $\bs{q}_k, \bs{p}_k, \bs{u}_k, f_k,$ the $\bs{w}$-problem \eqref{5.3}, \eqref{5.4} in feedback from is uniformly stable, with a decay rate $\gamma > 0, \ Re \ \lambda_{N+1} < -\gamma < 0$ either in $\ds \Wqs(\Omega) \equiv \lso \times L^q(\Omega)$ or else in $\ds \VbqpO \equiv \Bto \times B^{2-\rfrac{2}{p}}_{q,p}(\Omega)$.
		\begin{subequations}\label{6.7}
			\begin{equation}\label{6.7a}
				\norm{\bbm \bs{w}_f \\ w_h \ebm}_{(\cdot)} \leq C_{\gamma_0} e^{-\gamma_0 t} \norm{\bbm \bs{w}_f(0) \\ w_h(0) \ebm}_{(\cdot)}, \ t \geq 0.
			\end{equation}
		\nin In short, recalling the operator $\ds \BA_{_{F,q}}$ in \eqref{5.3b}
			\begin{equation}\label{6.7b}
				\norm{e^{\BA_{_{F,q}}t}}_{\calL (\cdot)} \leq C_{\gamma_0} e^{-\gamma_0 t}, \ t \geq 0.
			\end{equation} 
		\end{subequations}
	\nin Here, $(\cdot)$ denotes either the space $\Wqs(\Omega)$ or the space $\VbqpO$, see Remark \ref{Rmk-1.3}.		
	\end{thm}
	\begin{proof}
		The proof is similar to that of \cite{LPT.3}, making critical use of Theorem \ref{Thm-4.2} on the uniform stabilization \eqref{4.45} of the $\bs{w}_N$-system \eqref{4.38} in feedback form with arbitrary decay rate $\gamma_1 > 0$, in particular $-\gamma_1 < Re \ \lambda_{N+1} < 0$. Next, one examines the impact of the constructive feedback controls in \eqref{4.42}, \eqref{4.43} and \eqref{4.46} on the $\zeta_N$-dynamics \eqref{3.6b}, whose explicit solution is given by the variation of parameter formula
		\begin{equation}\label{6.8}
		\zeta_N (t) = e^{\BA_{q,N}^s} \zeta_N (0) + \int_{0}^{t} e^{\BA_{q,N}^s(t - \tau)}(I - P_N) \bbm P_q(m\bs{u}_N)(\tau) \\[1mm] \calB_q Dv_N(\tau) \ebm d \tau.
		\end{equation}
		\nin where 
		\begin{equation}\label{6.9}
		\norm{e^{\BA_{q,N}^s}}_{\calL(\cdot)} \leq C_{\gamma_0} e^{-\gamma_0 t}, \ 0 \geq t, \ 0 < \gamma_0 < \abs{Re \ \lambda_{N+1}};
		\end{equation}
		\nin while $\bs{u}_N(t), \ v_N(t)$ decay with an arbitrary large exponential rate $\gamma_1 > 0$ as in \eqref{4.45}. The validity of \eqref{6.9} on $\ds \WqsO$ follows by \eqref{3.3} and \eqref{3.4}, $\ds \BA_{q,N}^s$ being the generator of a s.c. analytic semigroup $\ds e^{\BA_q t}$ (Theorem \ref{Thm-1.2}), restricted to the stable subspace $\ds (\Wqs)^s_N$ in \eqref{3.2}. The validity of \eqref{6.9} on $\ds \VbqpO$ follows by interpolation as $\ds \calD(\BA_q) = \calD(\calA_q) \times \calD(\calB_q) = \calD(A_q) \times \calD(B_q)$, see \eqref{1.34}.
	\end{proof}
	\nin  Details are in \cite[Section 10]{LPT.3}.
	
	\section{Maximal $L^p$-regularity on $\ds \Wqs(\Omega)$ of the linearized feedback operator $\ds \BA_{_{F,q}}$ in \eqref{5.3b} up to $T = \infty$.}\label{Sec-7}
	\nin Consider the following abstract dynamics
	\begin{equation}\label{7.1}
	\zeta_t = \BA_{_{F,q}} \zeta + \chi, \quad \zeta(0) = \zeta_0,  \quad \text{in } \bs{W}^q_{\sigma}(\Omega)
	\end{equation}
	\begin{equation}\label{7.2}
	\zeta(t) = e^{\BA_{_{F,q}}}\zeta_0 + \int_{0}^{t} e^{\BA_{_{F,q}}(t-s)}\chi(s)ds
	\end{equation}
	\nin The main result of the present section is
	\begin{thm}\label{Thm-7.1}
	With reference to the $\bs{w}$-problem \eqref{5.3} in feedback form defining the feedback operator $\ds \BA_{_{F,q}}$, assume the setting of Theorem \ref{Thm-4.2} regarding the choice of the vectors $\bs{q}_k, \bs{p}_k, \bs{u}_k, f_k,$ in \eqref{5.1}, \eqref{5.2}, so that Theorem \ref{6.2} holds true. Then the operator $\ds \BA_{_{F,q}}$ has maximal $L^p$-regularity on $\bs{W}^q_{\sigma}(\Omega)$ up to $T = \infty$;
	\begin{equation}\label{7.3}
		\BA_{_{F,q}} \in MReg \left( L^p \left( 0, \infty; \bs{W}^q_{\sigma}(\Omega) \right) \right).
	\end{equation}
	\nin That is,
	\begin{equation}\label{7.4}
		(L \chi)(t) = \int_{0}^{t} e^{\BA_{_{F,q}}(t-s)}\chi(s)ds
	\end{equation}
	\end{thm}
	\nin continuous:
	\begin{equation}\label{7.5}
	L^p(0, \infty; \bs{W}^q_{\sigma}(\Omega)) \longrightarrow L^p \left(0, \infty; \calD \left(\BA_{_{F,q}} \right) \right)
	\end{equation}
	\nin so that continuously from \eqref{7.1} with $\zeta_0 = 0$:
	\begin{subequations}\label{7.6}		
	\begin{align}
	\chi \in L^p \left(0, \infty; \Wqs(\Omega) \right) &\longrightarrow \zeta \in \xipqs \times \xipq \nonumber 
	\\ & \hspace{.6cm} \in L^p \left(0, \infty; \calD \left(\BA_{_{F,q}} \right) \right) \cap W^{1,p}(0, \infty; \bs{W}^q_{\sigma}(\Omega))\\
	\longrightarrow \bs{X}^{\infty}_{p,q} \times X^{\infty}_{p,q} &\equiv L^p \left(0, \infty; \bs{W}^{2,q}(\Omega) \right) \times L^p \left(0, \infty; W^{2,q}(\Omega) \right)
	\end{align}
	\end{subequations}
	\begin{proof}
		Again, as in the proof of Theorem \ref{Thm-6.1} about analyticity of $\ds e^{\BA_{_{F,q}}t}$, it suffices to show that the streamlined operator $\ds \hat{\BA}_{_{F,q}}$ in \eqref{5.4}, \eqref{5.5} has maximal $L^p$-regularity on $\Wqs(\Omega)$ up to $T < \infty$. Equivalently, since $\Wqs(\Omega)$ is reflexive, that the adjoint operator $\ds \hat{\BA}^*_{_{F,q}}$ has maximal $L^p$-regularity on $\ds \left(\Wqs(\Omega)\right)^* = \bs{W}^{q'}_{\sigma}(\Omega)$ up to $T < \infty$. To establish this, we return to estimate \eqref{6.6} showing statement \eqref{6.3}, with $\ds \theta_0 = 1 - \gamma = 1 - \rfrac{1}{2q} + \varepsilon < 1$. Now we invoke that $\ds \BA_q^*$, hence $\hat{\BA}^*_q$ in \eqref{6.2} has maximal $L^p$-regularity on $\bs{W}^{q'}_{\sigma}(\Omega)$, (by duality on Theorem \hyperref[1.2]{1.2(ii)}, since $\Wqs(\Omega)$ is a reflexive space.) We then apply known perturbation theory \cite[Theorem 6.2, p311]{Dore:2000}, or \cite[Remark 1i, p 426 for $\beta = 1$]{KW:2001} to conclude that $\ds \hat{\BA}_{_{F,q}}^* \in MReg (L^p(0,T; \bs{W}^{q'}_{\sigma}(\Omega))), \ T < \infty$, and hence $\ds \hat{\BA}_{_{F,q}} \in MReg (L^p(0,T; \Wqs(\Omega))), \ T < \infty$. Then we obtain $\ds \BA_{_{F,q}} \in MReg (L^p(0,T; \Wqs(\Omega))), \ T < \infty$. But $\ds e^{\BA_{_{F,q}}t}$ is uniformly stable on $\ds \Wqs(\Omega)$ by Theorem \ref{Thm-6.2}. Hence we obtain $\ds \BA_{_{F,q}} \in MReg (L^p(0,\infty; \Wqs(\Omega)))$, and Theorem \ref{Thm-7.1} is proved.
	\end{proof}
	\nin We next examine the regularity of the term $\ds e^{\BA_{_{F,q}}t}\zeta_0$ due to the initial condition $\zeta_0$ as in \eqref{7.2}.
	
	\begin{thm}\label{Thm-7.2}
		\begin{enumerate}[(i)]
			\item Let $\ds 1 < p < \frac{2q'}{2q'-1},\ 1 < q' \leq 2,\ 2 \leq q,\ \frac{1}{q} + \frac{1}{q'} = 1$. Consider the adjoint operator $\ds \BA_{_{F,q}}^* = \hat{\BA}_{_{F,q}}^* + \Pi^*$ of $\ds \BA_{_{F,q}}$ in \eqref{5.4}, where $\ds \calD \left( \BA_{_{F,q}}^* \right) = \calD \left( \hat{\BA}_{_{F,q}}^* \right) = \calD \left(A^*_q \right) \times \calD \left(B^*_q \right) $ by \eqref{6.1b}. The adjoint s.c. analytic semigroup $\ds e^{\BA^*_{_{F,q}}t}$ on $\ds \bs{W}^{q'}_{\sigma}(\Omega)$ is uniformly stable, by duality on \eqref{6.7b}. Then
			\begin{multline}\label{7.7}
			e^{\BA^*_{_{F,q}}t}: \text{ continuous } \equiv\bs{V}^{q',p}_b(\Omega) \equiv \wti{\bs{B}}^{2-\rfrac{2}{p}}_{q',p}(\Omega) \times B^{2-\rfrac{2}{p}}_{q',p}(\Omega) \\
			=(\lo{q'}, \calD(A^*_q))_{1 - \frac{1}{p},p} \times ( L^{q'}(\Omega), \calD(B^*_q))_{1 - \frac{1}{p},p} \\ \longrightarrow \bs{X}^{\infty}_{p,q',\sigma} \times X^{\infty}_{p,q'} \equiv L^p(0, \infty; \calD(\BA^*_{_{F,q}})) \cap W^{1,p}(0, \infty; \bs{W}^{q'}_{\sigma}(\Omega))
			\end{multline}
			\item Consider the original s.c. analytic feedback semigroup $\ds e^{\BA_{_{F,q}}t}$ on $\ds \bs{W}^q_{\sigma}(\Omega)$, which is uniformly stable here by \eqref{6.7b}. Let $\ds 1 < p < \frac{2q}{2q-1},\ 2 \leq q$.
			\begin{multline}\label{7.8}
			e^{\BA_{_{F,q}}t}: \text{ continuous } \equiv\bs{V}^{q,p}_b(\Omega) \equiv \wti{\bs{B}}^{2-\rfrac{2}{p}}_{q,p}(\Omega) \times B^{2-\rfrac{2}{p}}_{q,p}(\Omega) \\
			= \left( \lso, \calD(A_q) \right)_{1 - \frac{1}{p},p} \times \left( L^q(\Omega), \calD(B_q) \right)_{1 - \frac{1}{p},p} \\ \longrightarrow \xipqs \times X^{\infty}_{p,q} \equiv L^p \left(0, \infty; \calD \left( \BA_{_{F,q}} \right) \right) \cap W^{1,p}\left(0, \infty; \bs{W}^q_{\sigma}(\Omega) \right)
			\end{multline}
		\end{enumerate}	
	\end{thm}
	\begin{proof}
		The proof is similar to that in \cite[Section 11]{LPT.3}.
	\end{proof}

	\newsavebox{\zhvector}
	\begin{lrbox}{\zhvector} $\bbm \bs{z} \\ h \ebm$ \end{lrbox}

	\section{Proof of Theorem \ref{Thm-2.2}. Well-posedness on $\xipqs \times \xipq$ of the non-linear \usebox{\zhvector}-dynamics in feedback form.}\label{Sec-8}
	
	\noindent In this section we return to the translated non-linear $\ds \bbm \bs{z} \\ h \ebm$-dynamics \eqref{1.28} or \eqref{1.29} and apply to it the feedback control pair $\{\bs{u},v\}$	
	\begin{equation}\label{8.1}
	\bbm P_q m(\bs{u}) \\ v \ebm = \bbm P_q \Bigg( m \bigg(\ds \sum_{k = 1}^{K} \ip{P_N \bbm \bs{z} \\ h \ebm}{\bs{q}_k} \bs{u}_k \bigg) \Bigg) \\[2mm] \ds \sum_{k = 1}^{K} \ip{P_N \bbm \bs{z} \\ h \ebm}{\bs{p}_k} f_k \ebm = \bbm J \bbm \bs{z} \\ h \ebm \\[3mm] F \bbm \bs{z} \\ h \ebm \ebm
	\end{equation}
	
	\noindent as in \eqref{2.10} that is, of the same structure as the feedback operators $J$ and $F$ in \eqref{5.1}, \eqref{5.2} identified on the RHS of the linearized $\ds \bs{w} = \bbm \bs{w}_f \\ w_h \ebm$-dynamics \eqref{5.3}. These are the feedback operators which produced the s.c. analytic, uniformly stable semigroup $\ds e^{\BA_{F,q}t}$ on $\ds \SqsO = \lso \times \lqo$ or on $ \VbqpO \equiv  \Bto \times \BsO$, (Theorem \ref{Thm-6.1} and Theorem \ref{Thm-6.2}) possessing $L^p$-maximal regularity on $\WqsO$ up to $T = \infty$; (Theorem \ref{Thm-7.1}). Thus, returning to \eqref{1.28} or \eqref{1.29}, in this section we consider the following feedback nonlinear problem, see \eqref{2.12}	
	\begin{multline}\label{8.2}
	\frac{d}{dt} \bbm \bs{z} \\ h \ebm = \BA_{F,q} \bbm \bs{z} \\ h \ebm - \bbm \calN_q & 0 \\ 0 & \calM_q[\bs{z}] \ebm \bbm \bs{z} \\ h \ebm; \quad \BA_{F,q}\bbm \bs{z} \\ h \ebm = \bbm \calA_q & -\calC_{\gamma} \\ -\calC_{\theta_e} & 0 \ebm \bbm \bs{z} \\ h \ebm \\+ \bbm J \bbm \bs{z} \\ h \ebm \\ -\calB_q\left(h - DF \bbm \bs{z} \\ h \ebm \right) \ebm,
	\end{multline}	
	\noindent specifically as in \eqref{2.11}
	\begin{subequations}\label{8.3}
		\begin{empheq}[left=\empheqlbrace]{align}
		\frac{d\bs{z}}{dt} - \calA_q \bs{z} + \calC_{\gamma} h + \calN_q \bs{z} &= P_q \left( m \left( \sum_{k = 1}^{K} \ip{P_N \bbm \bs{z} \\ h \ebm}{\bs{q}_k} \bs{u}_k \right) \right) \label{8.3a}\\
		\frac{dh}{dt} + \calC_{\theta_e} \bs{z} + \calM_q[\bs{z}]h &= -\calB_q \left( h - D \left( \ds \sum_{k = 1}^{K} \ip{P_N \bbm \bs{z} \\ h \ebm}{\bs{p}_k} f_k \right) \right). \label{8.3b}
		\end{empheq}
	\end{subequations}
	
	\noindent The variation of parameter formula for Eq \eqref{8.2} is
	
	\begin{equation}\label{8.4}
	\bbm \bs{z} \\ h \ebm (t) = e^{\BA_{F,q}t} \bbm \bs{z}_0 \\ h_0 \ebm - \int_{0}^{t} e^{\BA_{F,q}(t - \tau)} \bbm \calN_q \bs{z}(\tau) \\[1mm] \calM_q[\bs{z}]h(\tau) \ebm d \tau.
	\end{equation}
	
	\noindent We already know from \eqref{2.7} or \eqref{6.7b} that for $\ds \{\bs{z}_0, h_0\} \in \Bto \times \BsO \equiv \VbqpO, \ 1 < p < \frac{2q}{2q-1}$, we have: there is $M_{\gamma_0}$ such that
	
	\begin{equation}\label{8.5}
	\norm{e^{\BA_{F,q}t} \bbm \bs{z}_0 \\ h_0 \ebm}_{\VbqpO} \leq M_{\gamma_0} e^{-\gamma_0t} \norm{\bbm \bs{z}_0 \\ h_0 \ebm}_{\VbqpO}, \ t \geq 0
	\end{equation}
	\noindent with $M_{\gamma_0}$ possibly depending on $p,q$. Maximal $L^p$-regularity properties corresponding to the solution operator formula \eqref{8.4} were established in Theorem \ref{Thm-7.1}, Theorem \ref{Thm-7.2}. Accordingly, for	
	\begin{align}
	\bs{b}_0 &\equiv \{\bs{z}_0, h_0\} \in \Bto \times \BsO \equiv \VbqpO \label{8.6}\\
	\bs{f} &\equiv \{\bs{f}_1, f_2\} \in \xipqs \times \xipq \nonumber \\ 
	&\hspace{1cm} \equiv L^p(0, \infty, \calD (\BA_{F,q})) \cap W^{1,p}(0,\infty; \SqsO) \text{ (see \eqref{2.15})} \label{8.7}\\
	\calD(\BA_{F,q}) &\subset \calD(A_q) \times W^{2,q}(\Omega) \nonumber \\ &\hspace{1cm} = [\bs{W}^{2,q}(\Omega) \cap \bs{W}^{1,q}_0(\Omega) \cap \lso] \times W^{2,q}(\Omega)\label{8.8} \text{ (see \eqref{5.5b})}\\
	\SqsO &= \lso \times \lqo \text{ (see \eqref{1.30})} \nonumber	
	\end{align}
	\begin{subequations}\label{8.9}
		\begin{align}
			\xbipq &\equiv L^p \left(0,\infty; \left( \bs{W}^{2,q}(\Omega) \cap \bs{W}^{1,q}_0(\Omega) \cap \lso \right) \times W^{2,q}(\Omega) \right) \nonumber \\ & \hspace{6cm} \cap W^{1,p}(0,\infty; \WqsO); \label{8.9a} \\
			\xipq &\equiv L^p(0,\infty; W^{2,q}(\Omega) \cap W^{1,q}_0(\Omega)) \cap W^{1,q}(0,\infty; \lqo), \label{8.9b}
		\end{align}
	\end{subequations}
	\noindent as in \eqref{8.4}, we define the operator
	\begin{equation}\label{8.10}
	\calF (\bs{b}_0, \bs{f}) \equiv e^{\BA_{F,q}t} \bs{b}_0 - \int_{0}^{t} e^{\BA_{F,q}(t - \tau)} \bbm \calN_q \bs{f}_1(\tau) \\[1mm] \calM_q[\bs{f}_1]f_2(\tau) \ebm d \tau.
	\end{equation}
	
	\noindent The main result of this section is Theorem \ref{Thm-2.2} restated as
	
	\begin{thm}\label{Thm-8.1}
		Let $\ds d = 2,3, \ q > d, \ 1 < p < \frac{2q}{2q-1}$. There exists a positive constant $r_1 > 0$ (identified in the proof below in \eqref{8.31} such that if
		\begin{equation}\label{8.11}
		\norm{\bs{b}_0}_{\VqpO} = \norm{\{ \bs{z}_0, h_0 \}}_{\Bto \times \BsO} < r_1,
		\end{equation}
		then the operator $\calF$ in \eqref{8.10} has a unique fixed point non-linear semigroup solution in $\xipqs \times \xipq$, see \eqref{2.15}-\eqref{2.18}, or \eqref{8.7}
		\begin{equation}\label{8.12}
		\calF \bpm \bbm \bs{z}_0 \\ h_0 \ebm, \bbm \bs{z} \\ h \ebm \epm = \bbm \bs{z} \\ h \ebm, \ \text{or} \ \bbm \bs{z} \\ h \ebm (t) =  e^{\BA_{F,q}t} \bbm \bs{z}_0 \\ h_0 \ebm - \int_{0}^{t} e^{\BA_{F,q}(t - \tau)} \bbm \calN_q \bs{z}(\tau) \\[1mm] \calM_q[\bs{z}]h(\tau) \ebm d \tau,
		\end{equation}
		\noindent which therefore is the unique solution of problem \eqref{8.2} = \eqref{8.3} in $\ds \xipqs \times \xipq$.
	\end{thm}
	\noindent The proof of Theorem \ref{Thm-2.2} = Theorem \ref{Thm-8.1} is accomplished in two steps.\\	
	
	\noindent \underline{Step 1:}
	
	\begin{thm}\label{Thm-8.2}
		Let $d = 2,3,  \ q > d$ and $\ds 1 < p < \frac{2q}{2q - 1}$. There exists a positive constant $r_1 > 0$ (identified in the proof below in (\ref{8.31})) and a subsequent constant $r > 0$ (identified in the proof below in (\ref{8.29})) depending on $r_1 > 0$ and a constant $C$ in (\ref{8.28}), such that with $\ds \norm{\bs{b}_0}_{\VqpO} < r_1$ as in (\ref{8.11}), the operator $\ds \calF(\bs{b}_0,\bs{f})$ in (\ref{8.10}) maps a ball $B(0,r)$ in $\ds \xipqs \times \xipq$ into itself. \qedsymbol
	\end{thm}
	\noindent Theorem \ref{Thm-8.1} will follow then from Theorem \ref{Thm-8.2} after establishing that\\
	
	\noindent \underline{Step 2:}
	
	\begin{thm}\label{Thm-8.3}
		Let $d = 2,3,  \ q > d$ and $\ds 1 < p < \frac{2q}{2q - 1}$. There exists a positive constant $r_1 > 0$ such that if $\ds \norm{\bs{b}_0}_{\VqpO} < r_1$ as in (\ref{8.11}), there exists a constant $0 < \rho_0 < 1$ (identified in (\ref{8.56})), such that the operator $\calF(\bs{b}_0, \bs{f})$ in (\ref{8.10}) defines a contraction in the ball $B(0,\rho_0)$ of $\xipqs \times \xipq$. \qedsymbol
	\end{thm}	
	
	\noindent The Banach contraction principle then establishes Theorem \ref{Thm-8.1}, once we prove Theorems \ref{Thm-8.2} and  \ref{Thm-8.3}. These are proved below. \\
	
	\noindent \underline{Proof of Theorem \ref{Thm-8.2}}\\
	
	\noindent \textit{Step 1:} We start from the definition (\ref{8.10}) of $\ds \calF(\bs{b}_0, \bs{f})$ and invoke the maximal regularity properties \eqref{7.8} of Theorem \ref{Thm-7.2} for $\ds e^{\BA_{F,q}t}$ and \eqref{7.6} = \eqref{2.15} of Theorem \ref{Thm-7.1} for the integral term in  \eqref{8.10}. We then obtain from (\ref{8.10})	
	\begin{align}
	\norm{\calF (\bs{b}_0, \bs{f})}_{\xipqs \times \xipq} &\leq \norm{e^{\BA_{F,q}t}\bs{b}_0}_{\xipqs \times \xipq}\nonumber\\
	&+\norm{\int_{0}^{t} e^{\BA_{F,q}(t - \tau)} \bbm \calN_q \bs{f}_1 (\tau) \\[1mm] \calM_q[\bs{f}_1] f_2 (\tau) \ebm d \tau}_{\xipqs \times \xipq} \label{8.13}\\
	\leq C \bigg[ \norm{\bs{b}_0}_{\VbqpO} + &\norm{\calN_q \bs{f}_1}_{\lplqs} + \norm{\calM_q[\bs{f}_1]f_2}_{\lplq} \bigg]. \label{8.14}
	\end{align}	
	\noindent \textit{Step 2:} Regarding the term $\ds \calN_q \bs{f}_1$ we can invoke \cite[Eq (8.19)]{LPT.1} to obtain	
	\begin{equation}
	\norm{\calN_q \bs{f}_1}_{\lplqs} \leq C \norm{\bs{f}_1}^2_{\xipqs}, \ \bs{f}_1 \in \xipqs \label{8.15}.
	\end{equation}	
	\noindent Regarding the term $\ds \calM_q[\bs{f}_1]f_2$, we can trace the proof in \cite[from (8.10) $\longrightarrow $ (8.18)]{LPT.1} (which yielded estimate \eqref{8.14}). For the sake of clarity, we shall reproduce the computations in the present case with $\ds \calM_q [\bs{f}_1] f_2 = \bs{f}_1 \cdot \nabla f_2$, see \eqref{1.20}, \textit{mutatis mutandis}. We shall obtain	
	\begin{equation}\label{8.16}
	\norm{\calM_q [\bs{f}_1] f_2}_{\lplq} \leq C \norm{\bs{f}_1}_{\xipqs} \norm{f_2}_{\xipq}, \ \bs{f}_1 \in \xipqs, f_2 \in \xipq.
	\end{equation}	
	\noindent In fact, let us compute from \eqref{1.20}	
	\begin{align}
	\norm{\calM_q [\bs{f}_1] f_2}_{\lplq}^p &\leq \int_{0}^{\infty} \norm{\bs{f}_1 \cdot \nabla f_2}^p_{L^q(\Omega)} dt \label{8.17}\\
	&\leq \int_{0}^{\infty} \bigg\{ \int_{\Omega} \abs{\bs{f}_1(t,x)}^q \abs{\nabla f_2(t,x)}^q d\Omega \bigg\}^{\rfrac{p}{q}}dt\\
	\leq \int_{0}^{\infty} \bigg\{  \bigg[ \sup_{\Omega} &\abs{\nabla f_2(t,x)}^q  \bigg]^{\rfrac{1}{q}} \bigg[\int_{\Omega} \abs{\bs{f}_1(t,x)}^q d\Omega  \bigg]^{\rfrac{1}{q}} \bigg\}^p dt\\
	&\leq \int_{0}^{\infty} \norm{\nabla f_2 (t, \cdot)}^p_{\bs{L}^{\infty}(\Omega)} \norm{\bs{f}_1(t, \cdot)}^p_{\lso} dt\\
	\leq \sup_{0 \leq t \leq \infty} &\norm{\bs{f}_1(t, \cdot)}^p_{\bs{L}_{\sigma}^q(\Omega)} \int_{0}^{\infty} \norm{\nabla f_2 (t, \cdot)}^p_{\bs{L}^{\infty}(\Omega)} dt\\
	&= \norm{\bs{f}_1}^p_{L^{\infty}(0,\infty; \lso)}\norm{\nabla f_2}^p_{L^p(0,\infty;\bs{L}^{\infty}(\Omega))}. \label{8.22}
	\end{align}
	\noindent See also \cite[Eq (8.14)]{LPT.1}. The following embeddings hold true (see the stronger Eq \eqref{2.16}):
	\begin{enumerate}[(i)]
		\item \cite[Proposition 4.3, p 1406 with $\mu = 0, s = \infty, r = q$]{GGH:2012}  so that the required formula reduces to $1 \geq \rfrac{1}{p}$, as desired
		\begin{subequations}\label{8.23}
			\begin{align}
			\bs{f}_1 \in \xipqs \hookrightarrow \bs{f}_1 &\in L^{\infty}(0,\infty; \lso) \label{8.23a}\\
			\text{ so that, } \norm{\bs{f}_1}_{L^{\infty}(0,\infty; \lso)} &\leq C\norm{\bs{f}_1}_{\xipqs}; \label{8.23b}
			\end{align}
		\end{subequations}
		\item \cite[Theorem 2.4.4, p 74 requiring $C^1$-boundary]{SK:1989}
		\begin{equation}\label{8.24}
		W^{1,q}(\Omega) \subset L^{\infty}(\Omega) \text{ for q} > \text{dim } \Omega = d, \ d = 2,3,
		\end{equation}
	\end{enumerate}
	
	\noindent so that, with $p>1, q>d$:
	\begin{align}
	\norm{\nabla f_2}^p_{L^p(0,\infty; \bs{L}^{\infty}(\Omega))} &\leq C \norm{ \nabla f_2}^p_{L^p(0,\infty; \bs{W}^{1,q}(\Omega))} \leq C \norm{f_2}^p_{L^p(0,\infty; W^{2,q}(\Omega))} \label{8.25}\\
	&\leq C \norm{f_2}^p_{\xipq}. \label{8.26}
	\end{align}
	
	\noindent In going from (\ref{8.25}) to (\ref{8.26}) we have recalled the definition of $f_2 \in \xipq$ in \eqref{8.9b} or \eqref{2.15}, as $f_2$ was taken at the outset in $\ds \calD(\calB_q) \subset L^q(\Omega)$. Then the sought-after final estimate (\ref{8.16}) of the nonlinear term $\ds \calM_q[\bs{f}_1]f_2$ is obtained from substituting (\ref{8.23b}) and (\ref{8.26}) into the RHS of (\ref{8.22}).\\
	
	\noindent \textit{Step 3:} Substituting estimates (\ref{8.15}) and (\ref{8.16}) on the RHS of (\ref{8.14}), we finally obtain	
	\begin{equation}\label{8.27}
	\norm{\calF(\bs{b}_0, \bs{f}}_{\xipqs \times \xipq} \leq C \Big\{ \norm{\bs{b}_0}_{\VbqpO} + \norm{\bs{f}_1}_{\xipqs} \big( \norm{\bs{f}_1}_{\xipqs} + \norm{f_2}_{\xipq} \big) \Big\}.
	\end{equation}	
	\noindent See \cite[Eqt (8.20)]{LPT.1}.\\
	
	\noindent \textit{Step 4:} We now impose restrictions on the data on the RHS of (\ref{8.27}): $\bs{b}_0$ is in a ball of radius $r_1 > 0$ in $\ds \VbqpO = \Bto \times \BsO$ and $\bs{f} = \{ \bs{f}_1, f_2\}$ lies in a ball of radius $r > 0$ in $\ds \xipqs \times \xipq$. We further demand that the final result $\ds \calF(\bs{b}_0, \bs{f})$ shall lie in a ball of radius $r > 0$ in $\ds \xipqs \times \xipq$. Thus, we obtain from \eqref{8.27}	
	\begin{align}
	\norm{\calF(\bs{b}_0, \bs{f})}_{\xipqs \times \xipq} &\leq C \Big\{ \norm{\bs{b}_0}_{\VbqpO} + \norm{\bs{f}_1}_{\xipqs} \big( \norm{\bs{f}_1}_{\xipqs} + \norm{f_2}_{\xipq} \big) \Big\} \nonumber\\
	&\leq C(r_1 + r \cdot r) \leq r. \label{8.28}
	\end{align}	
	\noindent This implies
	\begin{equation}\label{8.29}
	Cr^2 - r + Cr_1 \leq 0 \quad \text{or} \quad \frac{1 - \sqrt{1-4C^2r_1}}{2C} \leq r \leq \frac{1 + \sqrt{1-4C^2r_1}}{2C}
	\end{equation}
	\noindent whereby
	\begin{equation}\label{8.30}
	\begin{Bmatrix}
	\text{ range of values of r }
	\end{Bmatrix}
	\longrightarrow \text{ interval } \Big[ 0, \frac{1}{C} \Big], \text{ as } r_1 \searrow 0,
	\end{equation}
	\noindent a constraint which is guaranteed by taking
	\begin{equation}\label{8.31}
	r_1 \leq \frac{1}{4C^2},\ C \text{ being the constant in } (\ref{8.28}) \ \left(\text{w.l.o.g. } C > \frac{1}{4} \right).
	\end{equation}
	\noindent We have thus established that by taking $r_1$ as in (\ref{8.31})  and subsequently $r$ as in (\ref{8.29}), then the map
	\begin{multline}\label{8.32}
	\calF(\bs{b}_0, \bs{f}) \text{ takes: }
	\begin{Bmatrix}
	\text{ ball in } \VbqpO \\
	\text{of radius } r_1
	\end{Bmatrix}
	\times
	\begin{Bmatrix}
	\text{ ball in } \xipqs \times \xipq \\
	\text{of radius } r
	\end{Bmatrix}
	\text{ into }\\
	\begin{Bmatrix}
	\text{ ball in } \xipqs \times \xipq \\
	\text{of radius } r
	\end{Bmatrix}, \ d < q, \ 1 < p < \frac{2q}{2q-1}.
	\end{multline}
	\noindent This establishes Theorem \ref{Thm-8.2}. \qedsymbol \\
	
	\noindent \underline{Proof of Theorem \ref{Thm-8.3}}
	
	\noindent \textit{Step 1:} For $\bs{f} = \{ \bs{f}_1, f_2 \}, \bs{g} = \{ \bs{g}_1, g_2 \}$ both in the ball of $\ds \xipqs \times \xipq$ of radius $r$ obtained in the proof of Theorem \ref{Thm-8.2}, we estimate from (\ref{8.10}):	
	\begin{align}
	&\norm{\calF (\bs{b}_0, \bs{f}) - \calF(\bs{b}_0,\bs{g})}_{\xipqs \times \xipq} = \nonumber\\ & \hspace{2.5cm} \norm{\int_{0}^t e^{\BA_{F,q}(t - \tau)} \bbm \calN_q\bs{f}_1(\tau) - \calN_q \bs{g}_1(\tau) \\[1mm] \calM_q[\bs{f}_1]f_2(\tau) - \calM_q[\bs{g}_1]g_2(\tau)\ebm d \tau}_{\xipqs \times \xipq} \nonumber\\
	&\leq \wti{m} \Big[ \norm{\calN_q\bs{f}_1 - \calN_q\bs{g}_1}_{\lplqs} + \norm{\calM_q[\bs{f}_1]f_2 - \calM_q[\bs{g}_1]g_2}_{L^p(0,\infty;L^q(\Omega))} \Big] \label{8.33}
	\end{align}
	\noindent after invoking the maximal regularity property \eqref{7.6} of Theorem \ref{Thm-7.1}, as in \eqref{8.14}.\\
	
	\noindent \textit{Step 2:} As to the first term of the RHS of (\ref{8.33}), we can invoke \cite[Eq (8.41)]{LPT.1} and obtain
	\begin{multline}\label{8.34}
	\norm{\calN_q\bs{f}_1 - \calN_q\bs{g}_1}_{\lplqs} \leq 2^{\rfrac{1}{p}}C^{\rfrac{1}{p}} \norm{\bs{f}_1 - \bs{g}_1}_{\xipqs} \left( \norm{\bs{f}_1}_{\xipqs} \right.\\ \left. + \norm{\bs{g}_1}_{\xipqs} \right).
	\end{multline}
	\noindent Regarding the second term on the RHS of (\ref{8.33}) involving $\calM_q$, we can track the proof of \cite[from (8.28) to (8.41)]{LPT.1} (which yielded estimate (\ref{8.34})) \textit{mutatis mutandis}. For the sake of clarity, we shall reproduce the computations in the present case, recalling from (\ref{1.22}) that $\ds \calM_q[\bs{f}_1]f_2 = \bs{f}_1 \cdot \nabla f_2, \calM_q[\bs{g}_1]g_2 = \bs{g}_1 \cdot \nabla g_2$. We shall obtain
	\begin{multline}\label{8.35}
	\norm{\calM_q[\bs{f}_1]f_2 - \calM_q[\bs{g}_1]g_2}^p_{\lplq} \leq C \Big\{ \norm{\bs{f}_1 - \bs{g}_1}^p_{\xipqs} \norm{f_2}^p_{\xipq} \\+ \norm{\bs{g}_1}^p_{\xipqs} \norm{f_2 - g_2}^p_{\xipq} \Big\}.
	\end{multline}
	\noindent In fact, adding and subtracting
	\begin{align}
	\calM_q[\bs{f}_1]f_2 - \calM_q[\bs{g}_1]g_2 &= \bs{f}_1 \cdot \nabla f_2 - \bs{g}_1 \cdot \nabla g_2 \nonumber\\
	&= \bs{f}_1 \cdot \nabla f_2 - \bs{g}_1 \cdot \nabla f_2 + \bs{g}_1 \cdot \nabla f_2 - \bs{g}_1 \cdot \nabla g_2 \nonumber\\
	&= (\bs{f}_1 - \bs{g}_1)\cdot \nabla f_2 + \bs{g}_1 \cdot \nabla (f_2 - g_2) = A + B. \label{8.36}
	\end{align}
	\noindent Thus, using $(*): \ \abs{A + B}^p \leq 2^p \big[ \abs{A}^p + \abs{B}^p  \big]$ \cite[p. 12]{TL:1980}, we estimate
	\begin{align}
	\norm{\calM_q[\bs{f}_1]f_2 - \calM_q[\bs{g}_1]g_2}_{\lplq}^p  & = \int_{0}^{\infty} \bigg\{ \bigg[ \int_{\Omega} \abs{\bs{f}_1 \cdot \nabla f_2 - \bs{g}_1 \cdot \nabla g_2}^q d \Omega \bigg]^{\rfrac{1}{q}}\bigg\}^p dt\\
	\text{(by \eqref{8.36})} \qquad &= \int_{0}^{\infty} \bigg[ \int_{\Omega} \abs{A+B}^q d \Omega \bigg]^{\rfrac{p}{q}}dt\\
	&\leq 2^p \int_{0}^{\infty} \bigg\{ \int_{\Omega} \big[\abs{A}^q + \abs{B}^q \big] d \Omega \bigg\}^{\rfrac{p}{q}}dt\\
	&= 2^p \int_{0}^{\infty} \bigg\{ \Big[ \int_{\Omega} \abs{A}^q d \Omega + \int_{\Omega} \abs{B}^q d \Omega   \Big]^{\rfrac{1}{q}} \bigg\}^p dt\\
	&= 2^p \int_{0}^{\infty} \bigg\{ \Big[ \norm{A}^q_{L^q(\Omega)} + \norm{B}^q_{L^q(\Omega)} \Big]^{\rfrac{1}{q}} \bigg\}^p dt\\
	\left(\text{by } (*) \mbox{ with } p \to \frac{1}{q} \right) \qquad
	&\leq 2^p \cdot 2^{\rfrac{p}{q}}\int_{0}^{\infty} \Big\{ \norm{A}_{L^q(\Omega)} + \norm{B}_{L^q(\Omega)} \Big\}^p dt\\
	(\text{by } (*) ) \qquad
	&\leq 2^{p+p+\rfrac{p}{q}}\int_{0}^{\infty} \Big[ \norm{A}^p_{L^q(\Omega)} + \norm{B}^p_{L^q(\Omega)} \Big] dt \label{8.43}\\
	&= 2^{p+p+\rfrac{p}{q}}\int_{0}^{\infty} \Big[ \norm{(\bs{f}_1 - \bs{g}_1)\cdot \nabla f_2}^p_{L^q(\Omega)}\nonumber \\ &\hspace{1cm}+\norm{\bs{g}_1 \cdot \nabla(f_2 - g_2)}^p_{L^q(\Omega)} \Big] dt \label{8.44}\\
	&\leq  2^{p+p+\rfrac{p}{q}}\int_{0}^{\infty} \Big\{ \norm{\bs{f}_1 - \bs{g}_1}^p_{\lso} \norm{\nabla f_2}^p_{L^q(\Omega)}\nonumber \\ &\hspace{1cm}+\norm{\bs{g}_1}^p_{\lso} \norm{\nabla(f_2 - g_2)}^p_{L^q(\Omega)} \Big\} dt, \label{8.45}
	\end{align}
	\noindent recalling the definitions of $A$ and $B$ in (\ref{8.36}) in passing from (\ref{8.43}) to (\ref{8.44}). \cite[This is the counterpart Eq (8.36)]{LPT.1}. Next we proceed by majorizing the $L^q(\Omega)$-norm for $\nabla f_2$ and $\nabla (f_2 - g_2)$ by the $L^{\infty}(\Omega)$-norm. We obtain
	\begin{align}
	\norm{\calM_q[\bs{f}_1]f_2 - \calM_q[\bs{g}_1]g_2}_{\lplq}^p &\leq \text{RHS of (\ref{8.45})} \nonumber\\
	&\hspace{-3.5cm} \leq  c 2^{p+p+\rfrac{p}{q}} \bigg\{ \sup_{0 \leq t \leq \infty} \norm{\bs{f}_1 - \bs{g}_1}^p_{\lso} \int_{0}^{\infty} \norm{\nabla f_2}^p_{L^{\infty}(\Omega)}\nonumber \\ & \hspace{-1.5cm} + \sup_{0 \leq t \leq \infty} \norm{\bs{g}_1}^p_{\lso} \int_{0}^{\infty} \norm{\nabla(f_2 - g_2)}^p_{L^{\infty}(\Omega)} dt \bigg\} \label{8.46}\\
	&\hspace{-3.5cm} \leq c 2^{p+p+\rfrac{p}{q}} \Big\{\norm{\bs{f}_1 - \bs{g}_1}^p_{L^{\infty}(0,\infty; \lso)} \norm{\nabla f_2}^p_{L^p(0,\infty; L^{\infty}(\Omega))}\nonumber \\
	&\hspace{-1.85cm} + \norm{\bs{g}_1}^p_{L^{\infty}(0,\infty; \lso)} \norm{\nabla(f_2 - g_2)}^p_{L^p(0,\infty; L^{\infty}(\Omega))} \Big\}\label{8.47}\\
	& \hspace{-5.5cm} \text{ by (\ref{8.23b}) and (\ref{8.26})}  \nonumber \\
	& \hspace{-3.5cm}\leq C \Big\{ \norm{\bs{f}_1 - \bs{g}_1}_{\xipqs}^p \norm{f_2}_{\xipq}^p + \norm{\bs{g}_1}_{\xipqs}^p  \norm{f_2 - g_2}_{\xipq}^p \Big\},  \label{8.48}
	\end{align}
	\noindent counterpart of \cite[Eq (8.39)]{LPT.1}.\\
	
	\noindent \textit{Step 3:} We substitute estimate (\ref{8.34}) and estimate (\ref{8.48}) on the RHS of (\ref{8.33}) and obtain via $(*)$
	
	\begin{align}
	\norm{\calF(\bs{b}_0,\bs{f}) - \calF(\bs{b}_0,\bs{g})}^p_{\xipqs \times \xipq} &\leq 2^p \wti{m}^p  \Big\{ \norm{\calN_q \bs{f}_1 - \calN_q \bs{g}_1}^p_{\lplq} \nonumber
	\\&+ \norm{\calM_q [\bs{f}_1] f_2 - \calM_q[\bs{g}_1] g_2}^p_{L^p(0,\infty;L^q(\Omega))} \Big\} \label{8.49}\\
	&\hspace{-2.5cm} \leq C_p \wti{m}^p  \Big\{ \norm{\bs{f}_1 - \bs{g}_1}^p_{\xipqs} \Big( \norm{\bs{f}_1}_{\xipqs} + \norm{\bs{g}_1}_{\xipqs} \Big)^p \nonumber
	\\&\hspace{-2.2cm} + \norm{\bs{f}_1 - \bs{g}_1}^p_{\xipqs} \norm{f_2}^p_{\xipq} + \norm{\bs{g}_1}^p_{\xipqs} \norm{f_2 - g_2}^p_{\xipq} \Big\}. \label{8.50}
	\end{align}
	\noindent This is the counterpart of \cite[Eq (8.42)]{LPT.1}.\\
	
	\noindent \textit{Step 4:} Next pick the points $\bs{f} = \{ \bs{f}_1, f_2 \}$ and $\bs{g} = \{ \bs{g}_1, g_2 \}$ in a ball of $\ds \xipqs \times \xipq$ of radius $R$:
	\begin{align}
	\norm{\bs{f}}_{\xipqs \times \xipq} &= \norm{\bs{f}_1}_{\xipqs} + \norm{f_2}_{\xipq} < R \label{8.51}\\
	\norm{\bs{g}}_{\xipqs \times \xipq} &= \norm{\bs{g}_1}_{\xipqs} + \norm{g_2}_{\xipq} < R. \label{8.52}
	\end{align}
	\noindent Then (\ref{8.51}), \eqref{8.52} used (\ref{8.50}) implies
	\begin{align}
	\norm{\calF(\bs{b}_0,\bs{f}) - \calF(\bs{b}_0,\bs{g})}^p_{\xipqs \times \xipq} &\leq C_p \Big\{ \norm{\bs{f}_1 - \bs{g}_1}^p_{\xipqs} \big[ (2R)^p + R^p \big] \nonumber \\
	&\hspace{2cm}+ \norm{f_2 - g_2}^p_{\xipq} R^p \Big\}\\
	&\leq C_p \Big\{ \norm{\bs{f}_1 - \bs{g}_1}^p_{\xipqs} \big[ (2^p + 1)R^p \big] \nonumber \\
	&\hspace{2cm}+ \norm{f_2 - g_2}^p_{\xipq} R^p \Big\} \nonumber\\
	\big( K_p = C_p (2^p + 1) \big) \quad &\leq K_p R^p \Big\{ \norm{\bs{f}_1 - \bs{g}_1}^p_{\xipqs} + \norm{f_2 - g_2}^p_{\xipq} \Big\}\nonumber\\
	[a>0, b>0, \ a^p + b^p \leq (a + b)^p] \quad &\leq K_p R^p \Big\{ \norm{\bs{f}_1 - \bs{g}_1}_{\xipqs} + \norm{f_2 - g_2}_{\xipq} \Big\}^p \label{8.54}
	\end{align}
	using $K_p = C_p(2^p + 1)$ and $a^p + b^p \leq (a+b)^p$ for $a>0, b>0$.
	Finally,
	\begin{align}
	\norm{\calF(\bs{b}_0,\bs{f}) - \calF(\bs{b}_0,\bs{g})}_{\xipqs \times \xipq} &\leq  K_p^{\rfrac{1}{p}} R \Big\{ \norm{\bs{f}_1 - \bs{g}_1}_{\xipqs} + \norm{f_2 - g_2}_{\xipq} \Big\} \nonumber\\
	&= \rho_0 \norm{\bs{f} - \bs{g}}_{\xipqs \times \xipq} \label{8.55}
	\end{align}
	\noindent and $\ds \calF(\bs{b}_0,\bs{f})$ is a contraction on the space $\xipqs \times \xipq$ as soon as
	\begin{equation}\label{8.56}
	\rho_0 = K_p^{\rfrac{1}{p}}R < 1 \quad \text{or} \quad R < \frac{1}{K_p^{\rfrac{1}{p}}}.
	\end{equation}
	\noindent In this case, the map $\ds \calF(\bs{b}_0, \bs{f})$ defined in (\ref{8.10}) has a fixed point $\ds \bbm \bs{z} \\ h \ebm$ in $\xipqs \times \xipq$:
	\begin{equation}\label{8.57}
	\calF \bpm \bs{b}_0, \bbm \bs{z} \\ h \ebm \epm = \bbm \bs{z} \\ h \ebm, \ \text{or} \ \bbm \bs{z} \\ h \ebm (t) =  e^{\BA_{F,q}t} \bbm \bs{z}_0 \\ h_0 \ebm - \int_{0}^{t} e^{\BA_{F,q}(t - \tau)} \bbm \calN_q \bs{z}(\tau) \\[1mm] \calM_q[\bs{z}]h(\tau) \ebm d \tau,
	\end{equation}
	\noindent and such point $\ds \bbm \bs{z} \\ h \ebm$ is the unique solution of the translated non-linear system (\ref{8.2}), or (\ref{8.3}), with finite dimensional control
	\begin{equation*}
	\bbm P_qm(\bs{u}) \\[2mm] v \ebm = \bbm \ds P_q \Bigg( m \bigg( \sum_{k = 1}^{K} \bigg(P_N \bbm \bs{z} \\ h \ebm, \bs{p}_k \bigg) \bs{u}_k \bigg) \Bigg) \\[3mm] \ds \sum_{k = 1}^{K} \bigg(P_N \bbm \bs{z} \\ h \ebm, \bs{p}_k \bigg) f_k \ebm
	\end{equation*}
	
	\noindent in feedback form, as described by Eq \eqref{8.1}. Theorem \ref{Thm-8.3} and hence Theorem \ref{Thm-8.1} are proved. \qedsymbol
	
	\begin{rmk} \label{R5.1}
		Recall from \eqref{2.1} and the statement of Theorem \ref{Thm-2.1}, that $\ds \bs{u}_k \in (\widehat{\bs{W}}^q_{\sigma})^u_N \subset  \widehat{\bs{L}}^{q}_{\sigma}(\Omega)$ as established in Appendix \ref{app-C}. This means that the feedback control acting on the fluid variable $\bs{u}$ is of reduced dimension $(d-1)$ to include the $d$\textsuperscript{th} component. \qedsymbol
	\end{rmk}
	
	\section{Proof of Theorem \ref{Thm-2.3}: local exponential decay of the non-linear $ [\bs{z},h]$ translated dynamics \eqref{2.11} = \eqref{2.12} with finite dimensional localized feedback controls and $(d-1)$ dimensional $\bs{u}_K$.}
	\label{Sec-9}
	
	\noindent In this section we return to the feedback problem \eqref{2.12}, rewritten equivalently as in \eqref{2.14}
	\begin{equation}\label{9.1}
	\bbm \bs{z} \\ h \ebm (t) =  e^{\BA_{F,q}t} \bbm \bs{z}_0 \\ h_0 \ebm - \int_{0}^{t} e^{\BA_{F,q}(t - \tau)} \bbm \calN_q \bs{z}(\tau) \\[1mm] \calM_q[\bs{z}]h(\tau) \ebm d \tau,
	\end{equation}
	\noindent For $\bs{b}_0 = [\bs{z}_0, h_0]$ in a small ball of $\ds \VbqpO = \Bto \times \BsO$, Theorem \ref{Thm-8.1}(= Theorem \ref{Thm-2.2}) provides a unique solution $\{\bs{z}, h\}$ in a small ball of $\ds \xipqs \times \xipq$. We recall from \eqref{6.7} of Theorem \ref{Thm-6.2}	
	\begin{equation}\label{9.2}
	\norm{e^{\BA_{F,q}t} \bbm \bs{z}_0 \\ h_0 \ebm}_{\VbqpO} \leq M_{\gamma_0} e^{-\gamma_0 t} \norm{\bbm \bs{z}_0 \\ h_0 \ebm}_{\VbqpO}, t \geq 0
	\end{equation}
	\noindent $\ds \VbqpO = \Bto \times \BsO$. Our goal is to show that for $\ds [\bs{z}_0, h_0]$ in a small ball of $\ds \VbqpO$, problem (\ref{9.1}) satisfies the exponential decay
	\begin{multline}\label{9.3}
	\norm{\bbm \bs{z} \\ h \ebm(t)}_{\VbqpO} \leq C_a e^{-a t} \norm{\bbm \bs{z}_0 \\ h_0 \ebm}_{\VbqpO}, t \geq 0,\\ \text{for some constants } a > 0,\ C = C_a \geq 1.
	\end{multline}
	
	\noindent \textit{Step 1:} Starting from (\ref{9.1}) and using the decay (\ref{9.2}), we estimate
	\begin{align}
	\norm{\bbm \bs{z}(t) \\ h(t) \ebm}_{\VbqpO} &\leq  M_{\gamma_0}e^{-\gamma_0 t} \norm{\bbm \bs{z}_0 \\ h_0 \ebm}_{\VbqpO} + \nonumber \\
	\sup_{0 \leq t \leq \infty} &\norm{\int_{0}^{t} e^{\BA_{F,q}(t - \tau)} \bbm \calN_q \bs{z}(\tau) \\[2mm] \calM_q[\bs{z}]h(\tau) \ebm d \tau}_{\VbqpO} \label{9.4}\\
	\leq M_{\gamma_0}e^{-\gamma_0 t} \norm{\bbm \bs{z}_0 \\ h_0 \ebm}_{\VbqpO} &+ C_1 \norm{\int_{0}^{t} e^{\BA_{F,q}(t - \tau)} \bbm \calN_q \bs{z}(\tau) \\[2mm] \calM_q[\bs{z}]h(\tau) \ebm d \tau}_{\xipqs \times \xipq} \label{9.5}\\
	\leq M_{\gamma_0}e^{-\gamma_0 t} \norm{\bbm \bs{z}_0 \\ h_0 \ebm}_{\VbqpO} &+ C_2 \bigg[ \norm{\calN_q \bs{z}}_{\lplqs} + \nonumber \\
	&\hspace{2cm}\norm{\calM_q[\bs{z}]h}_{\lplq} \bigg] \label{9.6}\\
	\leq M_{\gamma_0}e^{-\gamma_0 t} \norm{\bbm \bs{z}_0 \\ h_0 \ebm}_{\VbqpO} &+ C_3 \bigg[ \norm{\bs{z}}^2_{\xipqs} + \norm{\bs{z}}_{\xipqs} \norm{h}_{\xipq} \bigg]. \label{9.7}
	\end{align}
	\noindent In going from (\ref{9.4}) to \eqref{9.5} we have recalled the embedding $\ds \xipqs \times \xipq \hookrightarrow L^{\infty} \left( 0, \infty; \VbqpO \right)$ from \eqref{2.16}-\eqref{2.18}. Next, in going from \eqref{9.5} to \eqref{9.6} we have used the maximal regularity property \eqref{7.6} of Theorem \ref{Thm-7.1}. Finally, to go from \eqref{9.6} to \eqref{9.7}, we have invoked \eqref{8.15} for $\ds \calN_q \bs{z}$ and \eqref{8.16} for $\ds \calM_q [\bs{z}]h$. Thus, the conclusion of \textit{Step 1} is
	\begin{align}
	\norm{\bbm \bs{z}(t) \\ h(t) \ebm}_{\VbqpO} &\leq M_{\gamma_0}e^{-\gamma_0 t} \norm{\bbm \bs{z}_0 \\ h_0 \ebm}_{\VbqpO} + C_3 \norm{\bs{z}}_{\xipqs} \bigg[ \norm{\bs{z}}_{\xipqs} + \norm{h}_{\xipq} \bigg] \label{9.8}\\
	&= M_{\gamma_0}e^{-\gamma_0 t} \norm{\bbm \bs{z}_0 \\ h_0 \ebm}_{\VbqpO} + C_3 \norm{\bs{z}}_{\xipqs} \norm{[\bs{z}, h]}_{\xipqs \times \xipq}. \label{9.9}
	\end{align}	
	\noindent \textit{Step 2:} We now return to (\ref{9.1}) and take the $\ds \xipqs \times \xipq$ norm across:
	\begin{align}
	\norm{\bbm \bs{z} \\ h \ebm(t)}_{\xipqs \times \xipq} &\leq  \norm{e^{\BA_{F,q}t} \bbm \bs{z}_0 \\ h_0 \ebm}_{\xipqs \times \xipq} + \nonumber \\
	&\hspace{1.5cm} \norm{\int_{0}^{t} e^{\BA_{F,q}(t - \tau)} \bbm \calN_q \bs{z}(\tau) \\[1mm] \calM_q[\bs{z}]h(\tau) \ebm d \tau}_{\xipqs \times \xipq} \label{9.10}\\
	(\text{by \eqref{7.8}} \ &\leq M_1 \norm{\bbm \bs{z}_0 \\ h_0 \ebm}_{\VbqpO} + C_3 \norm{\bs{z}}_{\xipqs} \norm{[\bs{z}, h]}_{\xipqs \times \xipq} \label{9.11}
	\end{align}
	\noindent by invoking \eqref{7.8} of Theorem \ref{Thm-7.2} on the first semigroup term in (\ref{9.10}) and the estimate from (\ref{9.5}) to (\ref{9.9}) on the second integral term in (\ref{9.10}). Thus (\ref{9.11}) is established and implies
	\begin{equation}\label{9.12}
	\norm{\bbm \bs{z} \\ h \ebm}_{\xipqs \times \xipq}  \Big[ 1 - C_3 \norm{\bs{z}}_{\xipqs} \Big]\leq M_1 \norm{\bbm \bs{z}_0 \\ h_0 \ebm}_{\VbqpO}.
	\end{equation}
	\noindent [This is the counterpart of \cite[Eq (9.7)]{LPT.1}].\\
	
	\noindent \textit{Step 3:} The well-posedness Theorem \ref{Thm-8.1} says that
	\begin{equation}\label{9.13}
	\begin{Bmatrix}
	\text{ If } \norm{\bbm \bs{z}_0 \\ h_0 \ebm}_{\VbqpO} \leq r_1 \\[4mm]
	\text{for } r_1 \text{ sufficiently small}
	\end{Bmatrix}
	\implies
	\begin{Bmatrix}
	\text{ The solution } [\bs{z},h] \text{ satisfies} \\[3mm]
	\norm{\bbm \bs{z} \\ h \ebm}_{\xipqs \times \xipq} \leq r
	\end{Bmatrix}
	\end{equation}	
	\noindent where the constant $r$ satisfies the constraint (\ref{8.29}) in terms of the constant $r_1$ and some constant $C$ in (\ref{8.27}) or (\ref{8.28}) or (\ref{8.15}) or (\ref{8.16}). We seek to guarantee that we have
	\begin{equation}\label{9.14}
	\norm{\bs{z}}_{\xipqs} \leq \frac{1}{2C_3} < \frac{1}{2C},  \mbox{ hence }
	\frac{1}{2} \leq \Big[ 1 - C_3 \norm{\bs{z}}_{\xipqs} \Big],
	\end{equation}
	\noindent where we can always take the constant $C_3$ in \eqref{9.14} greater than the constant $C$ in \eqref{8.29}, \eqref{8.30}, \eqref{8.31}.Then \eqref{9.14} can be achieved by choosing $r_1 > 0$ sufficiently small. In fact, as $r_1 \searrow 0$. Eq. \eqref{8.30} shows that the interval $r_{\min} \leq r \leq r_{\max}$ of corresponding values of $r$ tends to the interval $[0,\frac{1}{C}]$. Thus, \eqref{9.14} can be achieved as $r_{\min} \searrow 0$: $0 < r_{\min} < r < \frac{1}{2C}$. Next (\ref{9.14}) implies that (\ref{9.12}) becomes
	\begin{equation}\label{9.15}
	\norm{\bbm \bs{z} \\ h \ebm}_{\xipqs \times \xipq} \leq 2M_1 \norm{\bbm \bs{z}_0 \\ h_0 \ebm}_{\VbqpO} \leq 2M_1r_1.
	\end{equation}
	\noindent by \eqref{9.13}. Substituting (\ref{9.15}) in estimate (\ref{9.9}) then yields
	\begin{align}
	\norm{\bbm \bs{z}(t) \\ h(t) \ebm}_{\VbqpO} &\leq M_{\gamma_0} e^{-\gamma_0 t}  \norm{\bbm \bs{z}_0 \\ h_0 \ebm}_{\VbqpO} + 2M_1 C_3 \norm{\bs{z}}_{\xipqs} \norm{\bbm \bs{z}_0 \\ h_0 \ebm}_{\VbqpO} \nonumber\\
	&\leq \what{M} \Big[ e^{-\gamma_0 t} + 4 \what{M} C_3 r_1 \Big] \norm{\bbm \bs{z}_0 \\ h_0 \ebm}_{\VbqpO} \label{9.16}
	\end{align}
	\noindent again by (\ref{9.15}) with $\ds \what{M} = \max \{ M_{\gamma_0}, M_1 \}$. This is the counterpart of \cite[Eq (9.16)]{LPT.1}.\\
	
	\noindent \textit{Step 4:} We now take $T$ sufficiently large and $r_1 > 0$ sufficiently small so that
	\begin{equation}\label{9.17}
	\beta = \what{M} \left[e^{-\gamma_0 T} + 4 \what{M} C_3 r_1 \right] < 1.
	\end{equation}
	\noindent Then (\ref{9.16}) implies
	\begin{equation}\label{9.18}
	\norm{\bbm \bs{z}(T) \\ h(T) \ebm}_{\VbqpO} \leq \beta \norm{\bbm {\bs{z}}_0 \\ h_0 \ebm}_{\VbqpO}
	\end{equation}
	\noindent $\ds \VbqpO = \Bto \times \BsO$, and hence
	\begin{equation}\label{9.19}
	\norm{\bbm \bs{z}(nT) \\ h(nT) \ebm}_{\VbqpO} \leq \beta \norm{\bbm \bs{z}((n-1)T) \\ h((n-1)T) \ebm}_{\VbqpO} \leq \beta^n \norm{\bbm \bs{z}_0 \\ h_0 \ebm}_{\VbqpO}.
	\end{equation}
	\noindent Since $\beta < 1$, the semigroup property of the evolution implies \cite{Bal:1981} that there are constants $M_{\wti{\gamma}} \geq 1, \wti{\gamma} >0$ such that
	\begin{equation}\label{9.20}
	\norm{\bbm \bs{z}(t) \\ h(t) \ebm}_{\VbqpO} \leq M_{\wti{\gamma}} e^{-\wti{\gamma}t} \norm{\bbm \bs{z}_0 \\ h_0 \ebm}_{\VbqpO}
	\end{equation}
	\noindent with $\ds \norm{[\bs{z}_0, h_0]}_{\VbqpO}  \leq r_1 = $ small. This proves \eqref{2.18}, i.e Theorem \ref{Thm-2.3}. \qedsymbol
	
	\begin{rmk}\label{I-Rmk-9.1}
		The above computations - (\ref{9.17}) through (\ref{9.19}) - can be used to support qualitatively the intuitive expectation that ``the larger the decay rate $\gamma_0$ in (\ref{6.7}) = (\ref{9.2}) of the linearized feedback $\bs{w}$-dynamics (\ref{5.4}), the larger the decay rate $\wti{\gamma}$ in (\ref{9.20}) of the nonlinear feedback $\{\bs{z},h\}$-dynamics \eqref{9.1} or \eqref{2.13}; hence the larger the rate $\wti{\gamma}$ in (\ref{9.20}) of the original $\{\bs{y},\theta\}$-dynamics in \eqref{1.1}".\\
		
		\noindent The following considerations are somewhat qualitative. Let $S(t)$ denote the non-linear semigroup in the space $\VbqpO \equiv  \Bto \times \BsO$, with infinitesimal generator $\ds \big[ \BA_{_{F,q}} - \calN_q \big]$ describing the feedback $\{\bs{z},h\}$-dynamics, \eqref{9.1} or \eqref{2.13} as guaranteed by the well posedness Theorem \ref{Thm-2.2} = Theorem \ref{Thm-8.1}. Thus, $\bbm \bs{z}(t) \\ h(t) \ebm = S(t)\bbm \bs{z}_0 \\ h_0 \ebm$ on $\VbqpO$. By (\ref{9.17}), we can rewrite (\ref{9.18}) as:
		\begin{equation}\label{9.21}
		\norm{S(T)}_{\calL \big(\VbqpO \big)} \leq \beta < 1.
		\end{equation}
		\noindent It follows from \cite[p 178]{Bal:1981} via the semigroup property that
		\begin{equation}\label{9.22}
		- \wti{\gamma} \ \ \text{is just below} \ \ \frac{\ln \beta}{T} < 0.
		\end{equation}
		\noindent Pick $r_1 > 0$ in (\ref{9.17}) so small that $4\widehat{M}^2 C_3 r_1$ is negligible, so that $\beta$ is just above $\ds \widehat{M} e^{- \gamma_0 T}$, so $\ln \beta$ is just above $\ds \big[ \ln \widehat{M} - \gamma_0 T \big]$, hence
		\begin{equation}\label{9.23}
		\frac{\ln \beta}{T} \text{ is just above } \Big[ (-\gamma_0) + \frac{\ln \widehat{M}}{T}\Big].
		\end{equation}
		\noindent Hence, by (\ref{9.22}), (\ref{9.23}),
		\begin{equation}\label{9.24}
		\wti{\gamma} \sim \gamma_0 - \frac{\ln \widehat{M}}{T}
		\end{equation}
		\noindent and the larger $\gamma_0$, the larger is $\wti{\gamma}$, as desired.
	\end{rmk}
	
	\begin{appendices}
		
		\section{Some auxiliary results for the Stokes and Oseen operators: analytic semigroup generation, maximal regularity, domains of fractional powers.}
		\label{app-A}

		\renewcommand{\theequation}{{\rm A}.\arabic{equation}}
		\setcounter{equation}{0}		
		\setcounter{theorem}{0}
		\renewcommand{\thetheorem}{{\sc A}.\arabic{theorem}}
		
		In this section we collect some known results used in the paper. As a prerequisite of the present Appendix \ref{app-A}, we make reference to the paragraph  \textbf{Definition of Besov spaces $\bs{B}^s_{q,p}$ on domains of class $C^1$ as real interpolation of Sobolev spaces}, Eqts (\ref{1.7})-(\ref{1.11}) and Remark \ref{Rmk-1.2}.

		\begin{enumerate}[(a)]
			\item \textbf{The Stokes and Oseen operators generate a strongly continuous analytic semigroup on $\lso$, $1 < q < \infty$}.\\
			
			\begin{thm}\label{A-Thm-1.2}
				Let $d \geq 2, 1 < q < \infty$ and let $\Omega$ be a bounded domain in $\mathbb{R}^d$ of class $C^3$. Then
				\begin{enumerate}[(i)]
					\item the Stokes operator $-A_q = P_q \Delta$ in \eqref{1.14}, repeated here as
					\begin{equation}\label{A-1.17}
					-A_q \psi  = P_q \Delta \psi , \quad
					\psi \in \mathcal{D}(A_q) = \bs{W}^{2,q}(\Omega) \cap \bs{W}^{1,q}_0(\Omega) \cap \lso
					\end{equation}
					generates a s.c. analytic semigroup $e^{-A_qt}$ on $\lso$. See \cite{Gi:1981} and the review paper \cite[Theorem 2.8.5 p 17]{HS:2016}.			
					\item The Oseen operator $\calA_q$ in \eqref{1.17} \label{A-Thm-1.2(iii)}
					\begin{equation}\label{A-1.18}
					\calA_q  = - (\nu A_q + A_{o,q}), \quad \calD(\calA_q) = \calD(A_q) \subset \lso
					\end{equation}
					generates a s.c. analytic semigroup $e^{\calA_qt}$ on $\lso$. This follows as $A_{o,q}$ is relatively bounded with respect to $A^{\rfrac{1}{2}}_q$, see \eqref{1.15}: thus a standard theorem on perturbation of an analytic semigroup generator applies \cite[Corollary 2.4, p 81]{P:1983}.
					
					\item \begin{subequations}\label{A-1.19}
						\begin{align}
						0 \in \rho (A_q) &= \text{ the resolvent set of the Stokes operator } A_q\\
						\begin{picture}(0,0)
						\put(-40,10){ $\left\{\rule{0pt}{18pt}\right.$}\end{picture}
						A_q^{-1} &: \lso \longrightarrow \lso \text{ is compact}.
						\end{align}
					\end{subequations}			
					\item The s.c. analytic Stokes semigroup $e^{-A_qt}$ is uniformly stable on $\lso$: there exist constants $M \geq 1, \delta > 0$ (possibly depending on $q$) such that
					\begin{equation}\label{A-1.20}
					\norm{e^{-A_qt}}_{\calL(\lso)} \leq M e^{-\delta t}, \ t > 0.
					\end{equation}
					
				\end{enumerate}
			\end{thm}
			\item \textbf{Domains of fractional powers, $\calD(A_q^{\alpha}), 0 < \alpha < 1$ of the Stokes operator $A_q$ on $\lso, 1 < q < \infty$}.
			We elaborate on \eqref{1.16}
			\begin{thm}\label{A-Thm-1.3}
				For the domains of fractional powers $\calD(A_q^{\alpha}), 0 < \alpha < 1$, of the Stokes operator $A_q$ in \eqref{A-1.17} = \eqref{1.14}, the following complex interpolation relation holds true \cite{Gi:1985} and \cite[Theorem 2.8.5, p 18]{HS:2016}
				\begin{equation}\label{A-1.21}
				[ \calD(A_q), \lso ]_{1-\alpha} = \calD(A_q^{\alpha}), \ 0 < \alpha < 1, \  1 < q < \infty;
				\end{equation}
				in particular
				\begin{equation}\label{A-1.22}
				[ \calD(A_q), \lso ]_{\frac{1}{2}} = \calD(A_q^{\rfrac{1}{2}}) \equiv \bs{W}_0^{1,q}(\Omega) \cap \lso.
				\end{equation}
				Thus, on the space $\calD(A_q^{\rfrac{1}{2}})$, the norms
				\begin{equation}\label{A-1.23}
				\norm{\nabla \ \cdot \ }_{L^q(\Omega)} \text{ and } \norm{ \ }_{L^q(\Omega)}
				\end{equation}
				are related via Poincar\'{e} inequality.
			\end{thm}
			
			\item \textbf{The Stokes operator $-A_q$ and the Oseen operator $\calA_q, 1 < q < \infty$ generate s.c. analytic semigroups on the Besov space, from \eqref{1.11}} \label{A-Sec-1.10d}
			\begin{subequations}\label{A-1.24}
				\begin{align}
				\Big( \lso,\mathcal{D}(A_q) \Big)_{1-\frac{1}{p},p} &= \Big\{ \bs{g} \in \Bso : \text{ div } \bs{g} = 0, \ \bs{g}|_{\Gamma} = 0 \Big\} \nonumber\\
				& \hspace{2cm} \text{if } \frac{1}{q} < 2 - \frac{2}{p} < 2; \label{A-1.24a}\\
				\Big( \lso,\mathcal{D}(A_q) \Big)_{1-\frac{1}{p},p} &= \Big\{ \bs{g} \in \Bso : \text{ div } \bs{g} = 0, \ \bs{g}\cdot \nu|_{\Gamma} = 0 \Big\} \nonumber\\
				&\equiv \Bto, \ \text{ if } 0 < 2 - \frac{2}{p} < \frac{1}{q}.  \label{A-1.24b}
				\end{align}	
			\end{subequations}
			Theorem \ref{A-Thm-1.2} states that the Stokes operator $-A_q$ generates a s.c. analytic semigroup on the space $\lso, \ 1 < q < \infty$, hence on the space $\calD(A_q)$ in \eqref{1.14} = (\ref{A-1.17}), with norm $\ds \norm{ \ \cdot \ }_{\calD(A_q)} = \norm{ A_q \ \cdot \ }_{\lso}$ as $0 \in \rho(A_q)$. Then, one obtains that the Stokes operator $-A_q$ generates a s.c. analytic semigroup on the real interpolation spaces in (\ref{A-1.24}). Next, the Oseen operator $\calA_q = -(\nu A_q + A_{o,q})$ in \eqref{A-1.18} = \eqref{1.17} likewise generates a s.c. analytic semigroup $\ds e^{\calA_q t}$ on $\ds \lso$ since $A_{o,q}$ is relatively bounded w.r.t. $A_q^{\rfrac{1}{2}},$ as $A_{o,q}A_q^{-\rfrac{1}{2}}$ is bounded on $\lso$. Moreover $\calA_q$ generates a s.c. analytic semigroup on $\ds \calD(\calA_q) = \calD(A_q)$ (equivalent norms). Hence $\calA_q$ generates a s.c. analytic semigroup on the real interpolation space of (\ref{A-1.24}). Here below, however, we shall formally state the result only in the case $\ds 2-\rfrac{2}{p} < \rfrac{1}{q}$. i.e. $\ds  1 < p < \rfrac{2q}{2q-1}$, in the space $\ds \Bto$, as this does not contain B.C., Remark \ref{Rmk-1.2}. The objective of the present paper is precisely to obtain stabilization results on spaces that do not recognize B.C.
			
			\begin{thm}\label{A-Thm-1.4}
				Let $1 < q < \infty, 1 < p < \rfrac{2q}{2q-1}$
				\begin{enumerate}[(i)]
					\item The Stokes operator $-A_q$ in \eqref{A-1.17} = \eqref{1.14} generates a s.c. analytic semigroup $e^{-A_qt}$ on the space $\Bt(\Omega)$ defined in \eqref{1.11} = \eqref{A-1.24b} which moreover is uniformly stable, as in \eqref{A-1.20},
					\begin{equation}\label{A-1.25}
					\norm{e^{-A_qt}}_{\calL \big(\Bto\big)} \leq M e^{-\delta t}, \quad t > 0.
					\end{equation}
					\item The Oseen operator $\calA_q$ in \eqref{A-1.18} = \eqref{1.17} generates a s.c. analytic semigroup $e^{\calA_qt}$ on the space $\Bto$ in \eqref{1.11} = \eqref{A-1.24}.
				\end{enumerate}
			\end{thm}
			\item \textbf{Space of maximal $L^p$ regularity on $\lso$ of the Stokes operator $-A_q, \ 1 < p < \infty, \ 1 < q < \infty $ up to $T = \infty$.}
			We return to the dynamic Stokes problem in $\{\boldsymbol{\varphi}(t,x), \pi(t,x) \}$
			\begin{subequations}\label{A-1.26}
				\begin{align}
				\boldsymbol{\varphi}_t - \Delta \boldsymbol{\varphi} + \nabla \pi &= \bs{F} &\text{ in } (0, T] \times \Omega \equiv Q   \label{N!-13a}\\		
				div \ \boldsymbol{\varphi} &\equiv 0 &\text{ in } Q\\
				\begin{picture}(0,0)
				\put(-80,5){ $\left\{\rule{0pt}{35pt}\right.$}\end{picture}
				\left.\boldsymbol{\varphi} \right \rvert_{\Sigma} &\equiv 0 &\text{ in } (0, T] \times \Gamma \equiv \Sigma\\
				\left. \boldsymbol{\varphi} \right \rvert_{t = 0} &= \boldsymbol{\varphi}_0 &\text{ in } \Omega,
				\end{align}
			\end{subequations}
			
			rewritten in abstract form, after applying the Helmholtz projection $P_q$ to (\ref{N!-13a}) and recalling $A_q$ in \eqref{A-1.17} = \eqref{1.14} as
			\begin{equation}\label{A-1.27}
			\boldsymbol{\varphi}' + A_q \boldsymbol{\varphi} = \Fs \equiv P_q \bs{F}, \quad \boldsymbol{\varphi}_0 \in \lqaq.
			\end{equation}
			
			Next, we introduce the space of maximal regularity for $\{\boldsymbol{\varphi}, \boldsymbol{\varphi}'\}$ as \cite[p 2; Theorem 2.8.5.iii, p 17]{HS:2016}, \cite[p 1404-5]{GGH:2012}, with $T$ up to $\infty$:
			\begin{equation}\label{A-1.28}
			\wti{\bs{X}}^T_{p,q, \sigma} = L^p(0,T;\calD(A_q)) \cap {W}^{1,p}(0,T;\lso)
			\end{equation}
			(recall \eqref{A-1.17} = \eqref{1.14} for $\calD(A_q)$) and the corresponding space for the pressure as
			\begin{equation}\label{A-1.29}
			\wti{Y}^T_{p,q} = L^p(0,T;\widehat{W}^{1,q}(\Omega)), \quad \widehat{W}^{1,q}(\Omega) = W^{1,q}(\Omega) / \mathbb{R}.
			\end{equation}
			The following embedding, also called trace theorem, holds true \cite[Theorem 4.10.2, p 180, BUC for $T=\infty$]{HA:2000}, \cite{PS:2016}.
			\begin{multline}\label{A-1.30}
			\xttpqs \subset \xbttpq \equiv L^p(0,T; \bs{W}^{2,q}(\Omega)) \cap W^{1,p}(0,T; \bs{L}^q(\Omega)) \\ \hookrightarrow C \Big([0,T]; \Bso \Big).
			\end{multline}
			For a function $\bs{g}$ such that $div \ \bs{g} \equiv 0, \ \bs{g}|_{\Gamma} = 0$ we have $\bs{g} \in \xbttpq \iff \bs{g} \in \xttpqs$.\\
			The solution of Eq \eqref{A-1.27} is
			\begin{equation}\label{A-1.31}
			\boldsymbol{\varphi}(t) = e^{-A_qt} \boldsymbol{\varphi}_0 + \int_{0}^{t} e^{-A_q(t-s)} \Fs(\tau) d \tau.
			\end{equation}
			The following is the celebrated result on maximal regularity on $\lso$ of the Stokes problem due originally to Solonnikov \cite{VAS:1977} reported in \cite[Theorem 2.8.5.(iii) and Theorem 2.10.1 p24 for $\boldsymbol{\varphi}_0 = 0$]{HS:2016}, \cite{S:2006}, \cite[Proposition 4.1 , p 1405]{GGH:2012}.
			\begin{thm}\label{A-Thm-1.5}
				Let $1 < p,q < \infty, T \leq \infty$. With reference to problem \eqref{A-1.26} = \eqref{A-1.27}, assume
				\begin{equation}\label{A-1.32}
				\Fs \in L^p(0,T;\lso), \ \boldsymbol{\varphi}_0 \in \Big( \lso, \calD(A_q)\Big)_{1-\frac{1}{p},p}.
				\end{equation}
				Then there exists a unique solution $\boldsymbol{\varphi} \in \xttpqs, \pi \in \yttpq$ to the dynamic Stokes problem \eqref{A-1.26} or \eqref{A-1.27}, continuously on the data: there exist constants $C_0, C_1$ independent of $T, \Fs, \boldsymbol{\varphi}_0$ such that via (\ref{A-1.30})
				\begin{equation}\label{A-1.33}
				\begin{aligned}
				C_0 \norm{\boldsymbol{\varphi}}_{C \big([0,T]; \Bso \big)} &\leq \norm{\boldsymbol{\varphi}}_{\xttpqs} +  \norm{\pi}_{\yttpq}\\
				\equiv \norm{\boldsymbol{\varphi}'}_{L^p(0,T;\lso)} &+ \norm{A_q \boldsymbol{\varphi}}_{L^p(0,T;\lso)} +  \norm{\pi}_{\yttpq}\\
				\leq C_1 \bigg \{ \norm{\Fs}_{L^p(0,T;\lso)}  &+ \norm{\boldsymbol{\varphi}_0}_{\big( \lso, \calD(A_q)\big)_{1-\frac{1}{p},p}} \bigg \}.
				\end{aligned}
				\end{equation}
				In particular,
				\begin{enumerate}[(i)]
					\item With reference to the variation of parameters formula \eqref{A-1.31} of problem \eqref{A-1.27} arising from the Stokes problem \eqref{A-1.26}, we have recalling \eqref{A-1.28}: the map
					\begin{align}
					\Fs &\longrightarrow \int_{0}^{t} e^{-A_q(t-\tau)}\Fs(\tau) d\tau \ : \text{continuous} \label{A-1.34}\\
					L^p(0,T;\lso) &\longrightarrow \xttpqs \equiv L^p(0,T; \calD(A_q)) \cap W^{1,p}(0,T; \lso) \label{A-1.35}.				
					\end{align}			
					\item The s.c. analytic semigroup $e^{-A_q t}$ generated by the Stokes operator $-A_q$ (see \eqref{A-1.17}= \eqref{1.14}) on the space $\ds \Big( \lso, \calD(A_q)\Big)_{1-\frac{1}{p},p}$ (see statement below \eqref{A-1.24}) satisfies
					\begin{subequations}\label{A-1.36}
						\begin{multline}
						e^{-A_q t}: \ \text{continuous} \quad \Big( \lso, \calD(A_q)\Big)_{1-\frac{1}{p},p} \longrightarrow \xttpqs \equiv \\ L^p(0,T; \calD(A_q)) \cap W^{1,p}(0,T; \lso) \label{A-1.36a}.
						\end{multline}
						In particular via \eqref{A-1.24b}, for future use, for $1 < q < \infty, 1 < p < \frac{2q}{2q - 1}$, the s.c. analytic semigroup $\ds e^{-A_q t}$ on the space $\ds \Bto$, satisfies
						\begin{equation}
						e^{-A_q t}: \ \text{continuous} \quad \Bto \longrightarrow \xttpqs. \label{A-1.36b}
						\end{equation}
					\end{subequations}				
					\item Moreover, for future use, for $1 < q < \infty, 1 < p < \frac{2q}{2q - 1}$, then \eqref{A-1.33} specializes to
					\begin{equation}\label{A-1.37}
					\norm{\boldsymbol{\varphi}}_{\xttpqs} + \norm{\pi}_{\yttpq} \leq C \bigg \{ \norm{\Fs}_{L^p(0,T;\lso)} + \norm{\boldsymbol{\varphi}_0}_{\Bto} \bigg \}.
					\end{equation}
				\end{enumerate}		
			\end{thm}
			
			\item \textbf{Maximal $L^p$ regularity on $\lso$ of the Oseen operator $\calA_q, \ 1 < p < \infty, \ 1 < q < \infty$, up to $T < \infty$.} We next transfer the maximal regularity of the Stokes operator $(-A_q)$ on $\lso$-asserted in Theorem \ref{A-Thm-1.5} into the maximal regularity of the Oseen operator $\calA_q = -\nu A_q - A_{o,q}$ in (\ref{A-1.18}) exactly on the same space $\xttpqs$ defined in (\ref{A-1.28}), however only up to $T < \infty$.
			
			\noindent Thus, consider the dynamic Oseen problem in $\{ \boldsymbol{\varphi}(t,x), \pi(t,x) \}$ with equilibrium solution $\bs{y}_e$, see (\ref{1.2}):		
			\begin{subequations}\label{A-1.38}
				\begin{empheq}[left=\empheqlbrace]{align}
				\boldsymbol{\varphi}_t - \Delta \boldsymbol{\varphi} + L_e(\boldsymbol{\varphi}) + \nabla \pi &= \bs{F} &\text{ in } (0, T] \times \Omega \equiv Q \label{A-1.38a}\\		
				div \ \boldsymbol{\varphi} &\equiv 0 &\text{ in } Q\\
				\left. \boldsymbol{\varphi} \right \rvert_{\Sigma} &\equiv 0 &\text{ in } (0, T] \times \Gamma \equiv \Sigma\\
				\left. \boldsymbol{\varphi}\right \rvert_{t = 0} &= \boldsymbol{\varphi}_0 &\text{ in } \Omega,
				\end{empheq}
			\end{subequations}
			\begin{equation}
			L_e(\boldsymbol{\varphi}) = (\bs{y}_e . \nabla) \boldsymbol{\varphi} + (\boldsymbol{\varphi}. \nabla) \bs{y}_e \hspace{3cm} \label{A-1.39}
			\end{equation}
			rewritten in abstract form, after applying the Helmholtz projector $P_q$ to (\ref{A-1.38a}) and recalling $\calA_q$ in (\ref{A-1.18}), as
			\begin{equation}\label{A-1.40}
			\boldsymbol{\varphi}_t = \calA_q \boldsymbol{\varphi} + P_q \bs{F} = - \nu A_q \boldsymbol{\varphi} - A_{o,q}\boldsymbol{\varphi}+ \Fbs, \quad \boldsymbol{\varphi}_0 \in \big( \lso, \calD(A_q)\big)_{1-\frac{1}{p},p}
			\end{equation}
			whose solution is via \eqref{1.18} = \eqref{A-1.18}
			\begin{equation}\label{A-1.41}
			\boldsymbol{\varphi}(t) = e^{\calA_qt} \boldsymbol{\varphi}_0 + \int_{0}^{t} e^{\calA_q(t-\tau)} \Fbs(\tau) d \tau,
			\end{equation}
			\begin{equation}\label{A-1.42}
			\boldsymbol{\varphi}(t) = e^{-\nu A_qt}\boldsymbol{\varphi}_0 + \int_{0}^{t} e^{-\nu A_q(t-\tau)} \Fbs(\tau) d \tau - \int_{0}^{t} e^{- \nu A_q(t-\tau)} A_{o,q} \boldsymbol{\varphi}(\tau) d \tau.
			\end{equation}
			
			\begin{thm}\label{A-Thm-1.6}
				Let $1 < p,q < \infty, \ 0 < T < \infty$. Assume (as in \eqref{A-1.32})
				\begin{equation}\label{A-1.43}
				\Fbs \in L^p \big( 0, T; \bs{L}^q_{\sigma} (\Omega) \big), \quad \boldsymbol{\varphi}_0 \in \lqaq
				\end{equation}
				where $\calD(A_q) = \calD(\calA_q)$, see \eqref{A-1.18} = \eqref{1.18}. Then there exists a unique solution $\boldsymbol{\varphi} \in \xttpqs, \ \pi \in \yttpq$ of the dynamic Oseen problem \eqref{A-1.38}, continuously on the data: that is, there exist constants $C_0, C_1$ independent of $\Fbs, \boldsymbol{\varphi}_0$ such that
				\begin{align}
				C_0 \norm{\boldsymbol{\varphi}}_{C \big([0,T]; \Bso \big)} &\leq \norm{\boldsymbol{\varphi}}_{\xttpqs} + \norm{\pi}_{\yttpq}\nonumber \\ &\equiv \norm{\boldsymbol{\varphi}'}_{L^p(0,T;\bs{L}^q(\Omega))} + \norm{A_q \boldsymbol{\varphi}}_{L^p(0,T;\bs{L}^q(\Omega))} + \norm{\pi}_{\yttpq}\\
				&\leq C_T \left \{ \norm{\Fbs}_{L^p(0,T;\lso)}  + \norm{\boldsymbol{\varphi}_0}_{\lqaq} \right \}
				\end{align}
				where $T < \infty$. Equivalently, for $1 < p, q < \infty$
				\begin{enumerate}[i.]
					\item The map
					\begin{equation}
					\begin{aligned}
					\Fbs \longrightarrow \int_{0}^{t} e^{\calA_q(t-\tau)}\Fbs(\tau) d\tau \ : \text{continuous}&\\
					L^p(0,T;\lso) &\longrightarrow L^p \big(0,T;\calD(\calA_q) = \calD(A_q) \big)\label{A-1.46}
					\end{aligned}			
					\end{equation}
					where then automatically, see (\ref{A-1.40})
					\begin{equation}
					L^p(0,T;\lso) \longrightarrow W^{1,p}(0,T;\lso) \label{A-1.47}
					\end{equation}
					and ultimately via \eqref{A-1.28}
					\begin{equation}
					L^p(0,T;\lso) \longrightarrow \xttpqs \equiv L^p \big(0,T;\calD(A_q) \big) \cap W^{1,p}(0,T;\lso). \label{A-1.48}
					\end{equation}
					\item The s.c. analytic semigroup $e^{\calA_q t}$ generated by the Oseen operator $\calA_q$ (see \eqref{A-1.18} = \eqref{1.18}) on the space $\ds \lqaq $ satisfies for $1 < p, q < \infty$
					\begin{equation}
					e^{\calA_q t}: \ \text{continuous} \quad \lqaq \longrightarrow L^p \big(0,T;\calD(\calA_q) = \calD(A_q)  \big) \label{A-1.49}
					\end{equation}
					and hence automatically by (\ref{A-1.28})
					\begin{equation}
					e^{ \calA_q t}: \ \text{continuous} \quad \lqaq \longrightarrow \xttpqs. \label{A-1.50}
					\end{equation}
					In particular, for future use, for $1 < q < \infty, 1 < p < \frac{2q}{2q - 1}$, we have that the s.c. analytic semigroup $\ds e^{\calA_q t}$ on the space $\ds \Bto$, satisfies
					\begin{equation}
					e^{\calA_q t}: \ \text{continuous} \quad \Bto \longrightarrow L^p \big(0,T;\calD(\calA_q) = \calD(A_q)  \big), \ T < \infty. \label{A-1.51}
					\end{equation}
					and hence automatically
					\begin{equation}
					e^{ \calA_q t}: \ \text{continuous} \quad \Bto \longrightarrow \xttpqs , \ T < \infty. \label{A-1.52}
					\end{equation}
				\end{enumerate}
			\end{thm}
			\noindent A proof is given in \cite[Appendix B]{LPT.1}.
		\end{enumerate}
		
		\section{Material in support of the proof of Theorem \ref{Thm-4.1}: the required UCP and $\ds D^*\calB_q^*f$ in \eqref{4.18}.}\label{app-B}
		\setcounter{equation}{0}
		\setcounter{theorem}{0}
		\renewcommand{\theequation}{{\rm B}.\arabic{equation}}
		\renewcommand{\thetheorem}{{\c B}.\arabic{theorem}}
		
		
	\subsection{The required UCP.}\label{Sec-B.1}
	\nin We return to the operator $\ds \BA_q$ in \eqref{1.34}:
		\begin{multline}\label{B.1}
		\BA_q = \bbm \calA_q & -\calC_{\gamma} \\ -\calC_{\theta_e} & -\calB_q \ebm : \bs{W}^q_{\sigma}(\Omega) = \lso \times \lqo \supset \calD(\BA_q) = \calD(\calA_q) \times \calD(\calB_q) \\ = (\bs{W}^{2,q}(\Omega) \cap \bs{W}^{1,q}_{0}(\Omega) \cap \lso) \times (W^{2,q}(\Omega) \cap W^{1,q}_{0}(\Omega)) \longrightarrow \bs{W}^q_{\sigma}(\Omega).
		\end{multline}
		\nin With $ \boldsymbol{\Phi} = [\boldsymbol{\varphi}, \psi ],$ the PDE-version of $ \BA_q  \boldsymbol{\Phi} = \lambda \boldsymbol{\Phi} $ is
		\begin{subequations}\label{B.2}
			\begin{empheq}[left=\empheqlbrace]{align}
			- \nu \Delta \boldsymbol{\varphi} + L_e(\boldsymbol{\varphi}) + \nabla \pi - \gamma \psi \bs{e}_d&= \lambda \boldsymbol{\varphi}   &\text{in } \Omega  \label{B.2a} \\
			- \kappa \Delta \psi - \bs{y}_e \cdot \nabla \psi + \boldsymbol{\varphi} \cdot \nabla \theta_e &= \lambda \psi &\text{in } \Omega \label{B.2b}\\
			\text{div } \boldsymbol{\varphi} &= 0   &\text{in } \Omega  \label{B.2c} \\
			\boldsymbol{\varphi} = 0, \ \psi &= 0 &\text{on } \Gamma.    \label{B.2d}
			\end{empheq}		
		\end{subequations}\\
		\nin Several UCP for over-determined versions of the eigenproblem \eqref{B.2} are given in \cite{TW.1}. However, establishing in Theorem \ref{Thm-4.1}	controllability of the finite dimensional projected problem \eqref{4.23} or \eqref{4.29} via verification of the Kalman rank condition \eqref{4.31} involves the following UCP for the adjoint problem. First, recall from \eqref{4.4} that the adjoint $\ds \BA_q^*$ or $\BA_q$ in \eqref{B.1} is		
		\begin{multline}\label{B.3}
		\BA_q^* = \bbm \calA_q^* & -\calC_{\theta_e}^* \\ -\calC_{\gamma}^* & -\calB_q^* \ebm : \bs{W}^{q'}_{\sigma}(\Omega) = \lsoo \times \lqoo \supset \calD(\BA_q^*) = \calD(\calA_q^*) \times \calD(\calB_q^*) \\ = (\bs{W}^{2,q'}(\Omega) \cap \bs{W}^{1,q'}_{0}(\Omega) \cap \lsoo) \times (W^{2,q'}(\Omega) \cap \bs{W}^{1,q'}_{0}(\Omega)) \longrightarrow \bs{W}^{q'}_{\sigma}(\Omega).
		\end{multline}
		
		\nin With $ \boldsymbol{\Phi}^* = [\boldsymbol{\varphi}^*, \psi^* ],  $ the PDE-version of $ \BA_q^*  \boldsymbol{\Phi}^* = \lambda \boldsymbol{\Phi}^* $ is
		
		\begin{subequations}\label{B.4}
			\begin{empheq}[left=\empheqlbrace]{align}
			-\nu \Delta \boldsymbol{\varphi^*} + L^*_e(\boldsymbol{\varphi^*})+ \psi^* \nabla \theta_e + \nabla \pi &= \lambda \boldsymbol{\varphi^*}   &\text{in } \Omega      \label{B.4a} \\
			-\kappa \Delta \psi^* - \bs{y}_e \cdot \nabla \psi^* - \gamma  \boldsymbol{\varphi^*}  \cdot\bs{e}_d &= \lambda \psi^* &\text{in } \Omega \label{B.4b}\\
			\text{div } \boldsymbol{\varphi^*} &= 0   &\text{in } \Omega  \label{B.4c} \\
			\boldsymbol{\varphi^*} = 0, \ \psi^* &= 0 &\text{on } \Gamma.    \label{B.4e}
			\end{empheq}		
		\end{subequations}
		
		\nin The UCP invoked in the proof of Theorem \ref{Thm-4.1} is
		\begin{thm}\label{Thm-B.1}{\cite[Theorem 5]{TW.1}}
			Let $\ds \{ \bs{\varphi}, \psi, \pi \} \in \left[ \bs{W}^{2,q}(\Omega) \cap \lso \right] \times W^{2,q}(\Omega), \ \pi \in W^{1,q}(\Omega),$ be a solution of the following dual problem
			\begin{subequations}
				\begin{empheq}[left=\empheqlbrace]{align}
				-\nu \Delta \boldsymbol{\varphi^*} + L^*_e(\boldsymbol{\varphi^*})+ \psi^* \nabla \theta_e + \nabla \pi &= \lambda \boldsymbol{\varphi^*}   &\text{in } \Omega      \label{B.5a} \\
				-\kappa \Delta \psi^* - \bs{y}_e \cdot \nabla \psi^* - \gamma  \boldsymbol{\varphi^*}  \cdot\bs{e}_d &= \lambda \psi^* &\text{in } \Omega \label{B.5b}\\
				\text{div } \boldsymbol{\varphi^*} &= 0   &\text{in } \Omega  \label{B.5c} \\
				\left\{\varphi^{*(1)}, \dots, \varphi^{*(d-1)} \right\} = 0, \ \psi^* &= 0 &\text{on } \omega \label{B.5d}
				\end{empheq}		
			\end{subequations}
			\nin with over-determination in \eqref{B.5d} Then 
			\begin{equation}
				\bs{\varphi}^* = 0, \ \psi^* = 0, \ \pi = const \ \text{in } \Omega. \ \qed
			\end{equation}
		\end{thm}
		
		\begin{proof}[Proof of $\ds D^* \calB^* f = -\kappa\frac{\partial f}{\partial \nu} \bigg|_{\Gamma}, \ f \in \calD(\calB_q^*)$, in \eqref{4.18} essentially by Green's formula]\  
			\begin{enumerate}[1.]
				\item Let $1 < q < \infty$ and define
				\begin{equation}\label{B.7}
					B_{o,q}h = \bs{y}_e \cdot h, \ L^q(\Omega) \supset \calD(B_{o,q}) = W^{1,q}_0(\Omega) \longrightarrow L^q(\Omega)
				\end{equation}
				\nin Then 
				\begin{equation}\label{B.8}
					B_{o,q}^*h = -\bs{y}_e \cdot f, \ L^{q'}(\Omega) \supset \calD(B_{o,q}^*) = W^{1,q'}_0(\Omega) \longrightarrow L^{q'}(\Omega)
				\end{equation}
				\nin In fact recalling $\bs{y}_e|_{\Gamma} \equiv 0$ and div $\bs{y}_e \equiv 0$ in $\Omega$ from (\hyperref[1.2c]{1.2c-d}), we compute in the duality pairing $L^q(\Omega), \ L^{q'}(\Omega)$, with $h \in \calD(B_{o,q}), \ f \in \calD(B^*_{o,q})$ by \eqref{B.7}:
				\begin{align}
					\ip{B_{o,q}h}{f} &= \int_{\Omega} f \bs{y}_e \cdot \nabla h d\Omega = \int_{\Gamma}\cancel{hf\bs{y}_e\cdot \nu d \Gamma} - \int_{\Omega} h \text{ div}(f\bs{y}_e) d\Omega \nonumber\\
					&= - \int_{\Omega} h \bs{y}_e \cdot \nabla f d\Omega = \ip{h}{B^*_{o,q}f}_{\Omega} \label{B.9}
				\end{align}
				\item We return to the operator $\calB_q$ in \eqref{1.18}, with $h \in \calD(\calB_q)$
				\begin{equation}\label{B.10}
					\calB_q h = -\kappa \Delta h + B_{o,q}h, \ \text{so that } \calB_q^* h = -\kappa \Delta f + B_{o,q}^*f, \ f \in \calD(\calB_q^*)
				\end{equation} 
				\nin We shall show that
				\begin{equation}\label{B.11}
					\ds D^* \calB^* f = -\kappa \frac{\partial f}{\partial \nu} \bigg|_{\Gamma}, \ f \in \calD(\calB_q^*) = W^{2,q'}(\Omega) \cap W^{1,q'}_0(\Omega)
				\end{equation}
				\nin In fact, by \eqref{B.8} and \eqref{B.10} we compute with $v \in L^2(\Gamma)$ recalling the definition of $D$ in \eqref{1.23a}
				\begin{align}
					\ip{D^* \calB^*_q f}{v}_{\Gamma} &= \ip{\calB^*_q f}{Dv}_{\Omega} = \ip{\Delta f}{-\kappa Dv}_{\Omega} + \ip{f}{B^*_{o,q}Dv}_{\Omega}\\
					&= \ip{f}{\cancel{(-\kappa \Delta - \bs{y}_e \cdot \nabla)} Dv}_{\Omega} + \int_{\Gamma} \frac{\partial f}{\partial \nu} (-\kappa Dv) d \Gamma + \int_{\Gamma}\cancel{f \left(\kappa \frac{\partial Dv}{\partial \nu}\right)} d \Gamma\\
					&= \ip{-\kappa \frac{\partial f}{\partial \nu}}{v}_{\Gamma}, \quad \text{for all } v \in L^q(\Gamma).
				\end{align}
			\end{enumerate}
		\nin and \eqref{B.11} is established.
		\end{proof}

		\section{Validation of the Kalman controllability conditions \eqref{4.31} with fluid vectors $\ds \{ \bs{u}_1, \dots, \bs{u}_{\ell_i}\}, \ \ell_i \leq K, \ i = 1, \dots, M$ having only $(d-1)$ components.}\label{app-C}
		\setcounter{equation}{0}
		\setcounter{theorem}{0}
		\renewcommand{\theequation}{{\rm C}.\arabic{equation}}
		\renewcommand{\thetheorem}{{\c C}.\arabic{theorem}}
		
		\nin For the sake of the clarity, we shall consider separately the cases $d = 2$ and $d = 3$.\\
		
		\nin \uline{Case $d = 2$.} Express the $2$-dimensional vectors $\bs{u}_i$ and $\bs{\varphi}_{ij}^*$ of Section \ref{Sec-4} in terms of their two components.
		\begin{equation}\label{C.1}
			\bs{u}_i = \left[u^{(1)}_i, u^{(2)}_i \right], \ \bs{\varphi}^*_{ij} = \left[ \varphi_{ij}^{*(1)}, \varphi_{ij}^{*(2)} \right], \ i = 1, \dots, M, \ j = 1, \dots, \ell_i
		\end{equation}
		\nin Here we shall recall the matrix $\bs{W}_i$ in \eqref{4.24} but we shall replace the matrix $U_i$ in \eqref{4.25} with the following matrix $U_i^{(2)}$: 
		\begin{equation}
			U^{(2)}_i = \bbm \ip{u^{(2)}_1}{\varphi_{i1}^{*(2)}}_{\omega} & \dots &  \ip{u^{(2)}_{\ell_i}}{\varphi_{i1}^{*(2)}}_{\omega} \\[2mm] 
			\ip{u^{(2)}_2}{\varphi_{i2}^{*(2)}}_{\omega} & \dots &  \ip{u^{(2)}_{\ell_i}}{\varphi_{i2}^{*(2)}}_{\omega}\\
			\vdots & & \vdots \\
			\ip{u^{(2)}_1}{\varphi_{i \ell_i}^{*(2)}}_{\omega} & \dots &  \ip{u^{(2)}_{\ell_i}}{\varphi_{i \ell_i}^{*(2)}}_{\omega}
			\ebm: \ell_i \times K
		\end{equation}
		\nin where now the duality pairing $\ip{ \ }{\ }_{\omega}$ involves two scalar functions. Next, for $d = 2$, we shall establish the Kalman algebraic rank condition for the finite dimensional, unstable, feedback problem \eqref{4.23}, by employing not the full strength of the $2$-dimensional vectors $\bs{u}_i$ and $\bs{\varphi}_{ij}^*$ as in \eqref{4.31} involving the matrix $U_i$ in \eqref{4.25}, but instead replacing $U_i$ with $U_i^{(2)}$. This way, only the second (scalar) components $u^{(2)}_i$ and $\varphi_{ij}^{*(2)}$ of the $2$-dimensional vectors $\bs{u}_i$ and $\bs{\varphi}_{ij}^*$ in \eqref{C.1} are needed. Recall that we are dealing with a pair $\{ \omega, \wti{\Gamma} \}$ as in Fig 1. The counterpart of Theorem \ref{Thm-4.1} is now
		
		\begin{thm}\label{Thm-C.1}
			Let $d = 2$. It is possible to select boundary vectors $f_1, \dots, f_K$ in $\calF \subset W^{2 - \rfrac{1}{q},q}(\wti{\Gamma})$ with support on $\wti{\Gamma}$, and scalar second components $\{ u^{(2)}_1, \dots, u^{(2)}_{\ell_i} \}$ as in \eqref{C.1}, such that
			\begin{equation}\label{C.3}
				\text{rank } \left[ W_i, \ U_i^{(2)} \right] = \ell_i, \ i, \dots, M.
			\end{equation}
		\end{thm}
	\begin{proof}
		We shall appropriately modify the proof of Theorem \ref{Thm-4.1}. As in this proof, the crux is to establish that we cannot have simultaneously
	\begin{equation}\label{C.4}
	\partial_{\nu} \psi_{i \ell_i}^* = \sum_{j = 1}^{\ell_i - 1} \alpha_j \partial_{\nu} \psi_{ij}^* \text{ in } L^q(\wti{\Gamma}) \text{ and } \varphi^{*(2)}_{i \ell_i} = \sum_{j = 1}^{\ell_i - 1} \alpha_j \varphi^{*(2)}_{ij}  \text{ in } L^q_{\sigma}(\omega)
	\end{equation}
	\nin with the same constant $\alpha_1, \dots, \alpha_{\ell_i-1}$ in both expression.\\
	
	\nin \uline{Claim: Statement \eqref{C.4} is false.} By contradiction, suppose that both linear combinations in \eqref{C.4} hold true. Next, as in \eqref{4.34} define the $(d+1)$-vector $\ds \bs{\Phi}^* \equiv \{ \bs{\varphi}^*, \psi^* \}$ by 
	\begin{equation}\label{C.5}
	\bs{\Phi}^* \equiv \bbm \bs{\varphi}^* \\ \psi^* \ebm = \sum_{j = 1}^{\ell_i - 1} \bbm \alpha_j \bs{\varphi}_{ij}^* \\ \alpha_j \psi_{ij}^* \ebm - \bbm \bs{\varphi}_{i \ell_i}^* \\ \psi_{i \ell_i}^*\ebm = \sum_{j = 1}^{\ell_i - 1} \alpha_j \bs{\Phi}^*_{ij} - \bs{\Phi}^*_{i \ell_i}, \ i = 1, \dots, M; \ q \geq 2,
	\end{equation}
	\nin in $\Wqs(\Omega) \equiv \lso \times L^q(\Omega)$, with $\ds \bs{\Phi}^*_{ij} \equiv \{ \bs{\varphi}^*_{ij}, \psi^*_{ij} \}$ eigenvector of $\ds \BA_{q,N}^*$ or $\ds \BA_q^*$, as in \eqref{4.3}. Then, in view of \eqref{C.4}, we now obtain
	\begin{equation}\label{C.6}
	\varphi^{*(2)} \equiv 0 \text{ in } \omega, \quad \partial_{\nu} \psi^*|_{\wti{\Gamma}} \equiv 0.
	\end{equation}
	\nin (in place of \eqref{4.35}). As in the proof of Theorem \ref{Thm-4.1}, since the $\ds \Phi^*_{ij}$ are eigenvectors of $\ds \BA_{q,N}^*$, or $\ds \BA_q^*$, so is the vector $\ds \bs{\Phi}^* \equiv \{ \bs{\varphi}^*, \psi^* \}$ defined in \eqref{C.5}. Thus $\ds \bs{\Phi}^* \equiv \{ \bs{\varphi}^*, \psi^* \}$ satisfies the PDE version (\hyperref[4.36a]{4.36a-d}) of the eigenproblem for the dual $\ds \BA_q^*$, which is now augmented with the over-determined conditions in \eqref{C.6}. The fact that now $\ds \bs{\varphi}^*|_{\Gamma} = 0$ by \eqref{4.36d} on the entire boundary $\Gamma$, permits a-fortiori the argument of \cite[Theorems 6, 7]{TW.1} to hold true. Namely, invoke $\ds \varphi^{*(2)} \equiv 0$ in $\omega$ in \eqref{C.6} to obtain, as in the proof of Theorem \ref{Thm-4.1}, the over-determined problem \eqref{4.37} for $\psi^*$ on $\omega$, with the conclusion that
	\begin{equation}\label{C.7}
		\psi^* \equiv 0 \text{ in } \omega, \text{ as in } \eqref{4.38}
	\end{equation}
	\nin Next, the divergence condition \eqref{4.36c} yields in view of \eqref{C.6}: $div \ \bs{\varphi}^* = \varphi^{*(1)}_{x_1} + \varphi^{*(2)}_{x_2} = \varphi^{*(1)}_{x_1} = 0$ in $\omega$. Hence $\varphi^{*(1)}(x_1, x_2) \equiv c(x_2)$ in $\omega$, where $c(x_2)$ is a function constant w.r.t. $x_1$ and depending only on $x_2$ in $\omega$. Next, let $\ds P = \{ x_1(P), x_2(P) \}$ be an arbitrary point of $\omega$. Consider the line $\ell$ passing through the point $P$ and parallel to the $x_1$-axis. On such a line $\ell$, the value $\ds \varphi^{*(1)}(x_1, x_2(P)) \equiv c_2(x_2(P))$ is constant w.r.t. $x_1$, as long as $\ell$ intersect $\omega$. By definition of the small set $\omega$ supported by $\wti{\Gamma}$, there is a non-empty open subset $\wti{\omega} \subset \omega$, where this happens: for all points $P$ in $\wti{\omega}$, the line $\ell$ remains in $\wti{\omega}$ and hits the boundary $\wti{\Gamma}$, where condition \eqref{4.36d} applies for $\ds \varphi^{*(1)}|_{\Gamma} = 0$. Thus $\varphi^{*(1)} \equiv 0$ in $\wti{\omega} \subset \omega$. Recalling \eqref{C.6}, we finally have
	\begin{equation}\label{C.8}
		\bs{\varphi}^* = \{ \varphi^{*(1)}, \varphi^{*(2)} \} \equiv 0 \text{ in } \wti{\omega} \subset \omega, \text{ along with } \psi^* \equiv 0 \text{ in } \wti{\omega} \text{ by } \eqref{C.7}.
	\end{equation}
	\nin We can then apply \cite[Theorem 5 for $\wti{\omega}$]{TW.1} and conclude that
	\begin{equation}\label{C.9}
		\bs{\varphi}^* = \{ \varphi^{*(1)}, \varphi^{*(2)} \} \equiv 0 \text{ in } \Omega, \ \psi^* \equiv 0 \text{ in } \Omega, \ p \equiv \text{ const in } \Omega.   
	\end{equation}
	\nin The rest of the proof proceeds as in Theorem \ref{Thm-4.1} following \eqref{4.40}. Theorem \ref{Thm-C.1} is established.
	\end{proof}
	\nin \uline{Case $d = 3$.} Now we express the $3$-dimensional vectors $\bs{u}_i$ and $\bs{\varphi}^*_{ij}$ of Section \ref{Sec-4} in terms of their components as
	\begin{equation}\label{C.10}
	\bs{u}_i = \left[u^{(1)}_i, u^{(2)}_i, u^{(3)}_i \right], \ \bs{\varphi}^*_{ij} = \left[ \varphi_{ij}^{*(1)}, \varphi_{ij}^{*(2)}, \varphi_{ij}^{*(3)} \right], \ i = 1, \dots, M, \ j = 1, \dots, \ell_i.
	\end{equation}
	We now distinguish two sub-cases.\\
	
	\nin \uline{Sub-case $\{1,3\}$.} Here, we extract only the first and third components, while we omit the second component. Accordingly, we introduce the following $2$-dimension vectors $\bs{u}^{(1,3)}_i$ and $\bs{\varphi}^{*(1,3)}$ and the corresponding matrix $\ds U^{(1,3)}_i$
	\begin{eqnarray}\label{C.11}
		\bs{u}_i^{(1,3)} = \left[u^{(1)}_i, u^{(3)}_i \right], \ \bs{\varphi}^{*(1,3)}_{ij} = \left[ \varphi_{ij}^{*(1)}, \varphi_{ij}^{*(3)} \right] \hspace{1.5cm} \nonumber\\
		U^{(1,3)}_i = \bbm \ip{\bbm u^{(1)}_1 \\[1mm] u^{(3)}_1 \ebm}{\bbm \varphi_{i1}^{*(1)} \\[1mm] \varphi_{i1}^{*(3)} \ebm}_{\omega} & \dots & \ip{\bbm u^{(1)}_{\ell_i} \\[1mm] u^{(3)}_{\ell_i} \ebm}{\bbm \varphi_{i1}^{*(1)} \\[1mm] \varphi_{i1}^{*(3)} \ebm}_{\omega} \\ 
		\dots & \dots & \dots	\\	
		\ip{\bbm u^{(1)}_1 \\[1mm] u^{(3)}_1 \ebm}{\bbm \varphi_{i \ell_i}^{*(1)} \\[1mm] \varphi_{i \ell_i}^{*(3)} \ebm}_{\omega} & \dots & \ip{\bbm u^{(1)}_{\ell_i} \\[1mm] u^{(3)}_{\ell_i} \ebm}{\bbm \varphi_{i \ell_i}^{*(1)} \\[1mm] \varphi_{i \ell_i}^{*(3)} \ebm}_{\omega}
		\ebm 		
	\end{eqnarray} 
	\nin \uline{Sub-case $\{2,3\}$.} Here, we extract only the second and third components, while we omit the first component. Accordingly, we introduce the following $2$-dimension vectors $\bs{u}^{(2,3)}_i$ and $\bs{\varphi}^{*(2,3)}$ and the corresponding matrix $\ds U^{(2,3)}_i$
	\begin{eqnarray}\label{C.12}
	\bs{u}_i^{(2,3)} = \left[u^{(2)}_i, u^{(3)}_i \right], \ \bs{\varphi}^{*(2,3)}_{ij} = \left[ \varphi_{ij}^{*(2)}, \varphi_{ij}^{*(3)} \right] \hspace{1.5cm} \nonumber\\
	U^{(2,3)}_i = \bbm \ip{\bbm u^{(2)}_1 \\[1mm] u^{(3)}_1 \ebm}{\bbm \varphi_{i1}^{*(2)} \\[1mm] \varphi_{i1}^{*(3)} \ebm}_{\omega} & \dots & \ip{\bbm u^{(2)}_{\ell_i} \\[1mm] u^{(3)}_{\ell_i} \ebm}{\bbm \varphi_{i1}^{*(2)} \\[1mm] \varphi_{i1}^{*(3)} \ebm}_{\omega} \\ 
	\dots & \dots & \dots\\	
	\ip{\bbm u^{(2)}_1 \\[1mm] u^{(3)}_1 \ebm}{\bbm \varphi_{i \ell_i}^{*(2)} \\[1mm] \varphi_{i \ell_i}^{*(3)} \ebm}_{\omega} & \dots & \ip{\bbm u^{(2)}_{\ell_i} \\[1mm] u^{(3)}_{\ell_i} \ebm}{\bbm \varphi_{i \ell_i}^{*(2)} \\[1mm] \varphi_{i \ell_i}^{*(3)} \ebm}_{\omega}
	\ebm 		
	\end{eqnarray} 
	\nin Next, for $d = 3$, we shall establish the Kalman algebraic rank condition for the finite dimensional, unstable, feedback problem \eqref{4.23}, by replacing the matrix $U_i$ in \eqref{4.25}, with either the matrix $U^{(1,3)}_i$ in \eqref{C.11} (Sub-case $\{1,3\}$); or else with the matrix $U^{2,3}_i$ in \eqref{C.12} (Sub-case $\{2, 3\}$). This way, the third (scalar) components $u^{(3)}_i$ and $\varphi^{*(3)}_{ij}$ of the $3$-dimensional vectors $\bs{u}_i$ and $\bs{\varphi}^*_{ij}$ in \eqref{C.10} are needed in both sub-cases. The counterpart of Theorem \ref{Thm-4.1} (or Theorem \ref{Thm-C.1}) is now
	
	\begin{thm}\label{Thm-C.2}
		Let $d = 3$. It is possible to select boundary vectors $f_1, \dots, f_K$ in $\calF \subset W^{2 - \rfrac{1}{q},q}(\wti{\Gamma})$ with support on $\wti{\Gamma}$, and 
		\begin{enumerate}[(i)]
			\item either $2$-dimensional vectors $\bs{u}^{(1,3)}_i = \left\{ u^{(1)}_i, u^{(3)}_i \right\}$ as in \eqref{C.11}
			\item or else $2$-dimensional vectors $\bs{u}^{(2,3)}_i = \left\{ u^{(2)}_i, u^{(3)}_i \right\}$ as in \eqref{C.12},
		\end{enumerate}such that
		\begin{equation}\label{C.13}
		\text{rank } \left[ W_i, \ U_i^{(1,3)} \right] = \ell_i, \ i, \dots, M, \text{ in case } (i)
		\end{equation}
		\nin or
		\begin{equation}\label{C.14}
		\text{rank } \left[ W_i, \ U_i^{(2,3)} \right] = \ell_i, \ i, \dots, M, \text{ in case } (ii)
		\end{equation}
		\nin respectively.
	\end{thm}
	\begin{proof}
		We shall only give a proof in the sub-case $\{1,3\}$, as the proof of the sub-case $\{2,3\}$ is the same \textit{mutatis mutandis}. As in the proof of Theorem \ref{Thm-4.1} (or Theorem \ref{Thm-C.1}), the crux is to establish that we cannot have simultaneously
		\begin{equation}\label{C.15}
		\partial_{\nu} \psi_{i \ell_i}^* = \sum_{j = 1}^{\ell_i - 1} \alpha_j \partial_{\nu} \psi_{ij}^* \text{ in } L^q(\wti{\Gamma}) \quad \text{and} \quad \varphi^{*(1,3)}_{i \ell_i} = \sum_{j = 1}^{\ell_i - 1} \alpha_j \varphi^{*(1,3)}_{ij}  \text{ in } \bs{L}^q_{\sigma}(\omega)
		\end{equation}
		\nin with the same constants $\alpha, \dots, \alpha_{\ell_i - 1}$ in both expressions.
	\end{proof}
	\nin \uline{Claim: Statement \eqref{C.15} is false.} By contradiction, suppose that both linear combinations in \eqref{C.15} hold true. Next, define the $(d+1)$-vector $\ds \bs{\Phi}^* \equiv \{ \bs{\varphi}^*, \psi^* \}$ as in \eqref{4.34} (or \eqref{C.5}). Then, in view of \eqref{C.15}, we now obtain
	\begin{equation}\label{C.16}
	\bs{\varphi}^{*(1,3)} \equiv \left\{ \varphi^{*(1)}, \varphi^{*(3)} \right\} \equiv 0 \text{ in } \omega, \quad \partial_{\nu} \psi^*|_{\wti{\Gamma}} \equiv 0.
	\end{equation}
	\nin (in place of \eqref{4.35} or \eqref{C.6}). As in Theorem \ref{Thm-4.1} (or Theorem \ref{Thm-C.1}) the $\ds \bs{\Phi}^* \equiv \{ \bs{\varphi}^*, \psi^* \}$ is an eigenvector of $\BA_q^*$ and thus satisfies the corresponding PDE-version - that is problem (\hyperref[4.36]{4.36a-d}), augmented this time with the over-determined conditions in \eqref{C.16}. Next, invoke $\varphi^{*(3)} \equiv 0$ in $\omega$ from \eqref{C.16} to obtain, as this the proof of Theorem \ref{Thm-4.1}, the over-determined problem \eqref{4.37} for $\psi^*$ in $\omega$, with the conclusion that
	\begin{equation}\label{C.17}
		\psi^* \equiv 0 \text{ in } \omega,
	\end{equation}
	\nin as in \eqref{4.36} (or \eqref{C.7}). Next, the divergence condition \eqref{4.36c} yields in the view of \eqref{C.16}: $div \ \bs{\varphi}^* = \varphi^{*(1)}_{x_1} + \varphi^{*(2)}_{x_2} + \varphi^{*(3)}_{x_3} = \varphi^{*(2)}_{x_2} = 0$ in $\omega$. Hence $\varphi^{*(2)}(x_1, x_2, x_3) \equiv c(x_1, x_3)$ in $\omega$, where $c(x_1, x_3)$ is a function constant w.r.t. $x_2$ and depending only on $x_1$ and $x_3$ in $\omega$. Next, let $\ds P = \{ x_1(P), x_2(P), x_3(P) \}$ be an arbitrary point of $\omega$. Consider the plane $\pi_P$ passing through the point $P$ and parallel to the $\{x_1, x_3\}$-coordinate plane. As the point $\{ x_1, x_2(P), x_3 \}$ of $\omega$ runs over the plane $\pi_P$, the value $\ds \varphi^{*(2)}(x_1, x_2(P), x_3) = c(x_1, x_3)$ is independent of $x_2(P)$, as long as such plane $\pi_P$ intersects $\omega$. By definition of the small $\omega$ supported by $\wti{\Gamma}$, there is a non-empty open subset $\wti{\omega} \subset \omega$, where this happens: for all points $P$ in $\wti{\omega}$, the plane $\pi_P$ remains in $\wti{\omega}$ and hits the boundary $\wti{\Gamma}$, where condition \eqref{4.36d} applies for $\ds \varphi^{*(2)}|_{\Gamma} = 0$. Thus $\ds \varphi^{*(2)} \equiv 0$ in $\wti{\omega} \subset \omega$. Recalling \eqref{C.16} we conclude that
	\begin{equation}\label{C.18}
	\bs{\varphi}^* = \left\{ \varphi^{*(1)}, \varphi^{*(2)}, \varphi^{*(3)} \right\} \equiv 0 \text{ in } \wti{\omega} \subset \omega, \text{ along with } \psi^* \equiv 0 \text{ in } \wti{\omega} \text{ by } \eqref{C.17}.
	\end{equation}
	\nin We can then apply \cite[Theorem 5 for $\wti{\omega}$]{TW.1} and conclude that
	\begin{equation}\label{C.19}
	\bs{\varphi}^* = \left\{ \varphi^{*(1)}, \varphi^{*(2)}, \varphi^{*(3)} \right\} \equiv 0 \text{ in } \Omega, \  \psi^* \equiv 0 \text{ in } \Omega, \ p \equiv \text{ const in } \Omega.   
	\end{equation}
	\nin The rest of the proof proceeds as in Theorem \ref{Thm-4.1} following \eqref{4.40}. Theorem \ref{Thm-C.2} is established. \\
	\end{appendices}

	\nin The authors wish to thank the referee for much appreciated comments and suggestions.

\end{document}